\documentclass{article}
\pdfoutput=1

\usepackage{mathrsfs}
\usepackage{amsmath,amsthm}
\usepackage{amssymb}
\usepackage{microtype}
\usepackage{tikz}
\usepackage[pdftex,colorlinks,citecolor=black, linkcolor=black,urlcolor=black,bookmarks=false]{hyperref}

\usepackage{doi}
\usepackage{todonotes}
\usepackage{comment}

\usepackage{bbm}

\usepackage[margin=1in]{geometry}

\numberwithin{equation}{section}
\numberwithin{figure}{section}

\newtheorem{theorem}{Theorem}[section]
\newtheorem{proposition}[theorem]{Proposition}
\newtheorem{lemma}[theorem]{Lemma}
\newtheorem{corollary}[theorem]{Corollary}

\newtheorem{claim}[theorem]{Claim}
\theoremstyle{definition}
\newtheorem{definition}[theorem]{Definition}

\newtheorem{remark}[theorem]{Remark}

\newtheorem{assumption}[theorem]{Assumption}

\DeclareMathOperator{\spn}{span}

\newcommand{\FF}{\mathbb{F}}

\newcommand{\QQ}{\mathbb{Q}}

\newcommand{\ZZ}{\mathbb{Z}}

\def\d22{d_{2\to 2}}
\def\tdel22{{\widetilde{\delta}}_{2\to 2}}

\usepackage{dsfont}

\newcommand{\FP}{{\mathbb{F}_{\!P}}}
\newcommand{\FPnk}{{\mathbb{F}_{\!P}^{n+k-1}}}
\newcommand{\FPk}{{\mathbb{F}_{\!P}^{k}}}
\newcommand{\FPkmin}{{\mathbb{F}_{\!P}^{k-1}}}
\newcommand{\Fpn}{{\mathbb{F}_p^n}}
\newcommand{\Zmn}{{\mathbb{Z}_m^n}}

\newcommand{\Zn}{{\mathbb{Z}^n}}
\newcommand{\Fp}{{\mathbb{F}_p}}
\renewcommand{\l}{\ell}
\newcommand{\ones}{\mathds{1}^n}
\newcommand{\indic}{\mathbf{1}}
\renewcommand{\phi}{\varphi}
\newcommand{\thet}{\vartheta}
\newcommand{\HH}{\operatorname{H}}
\renewcommand{\P}{\operatorname{\mathbb{P}}}
\newcommand{\E}{\operatorname{\mathbb{E}}}
\newcommand{\su}{\subseteq}

\title{A lower bound for the $k$-multicolored sum-free problem in $\Zmn$}

\author{L\'{a}szl\'{o} Mikl\'{o}s Lov\'{a}sz\thanks{Department of Mathematics, MIT, Cambridge, MA 02142. Email \texttt{lmlovasz@mit.edu}. Research supported by NSF Postdoctoral Fellowship Award DMS-1705204.} \and Lisa Sauermann\thanks{Department of Mathematics, Stanford University, Stanford, CA 94305. Email: \texttt{lsauerma@stanford.edu}.}}

\begin{document}

\maketitle

\begin{abstract}

\noindent In this paper, we give a lower bound for the maximum size of a $k$-colored sum-free set in $\Zmn$, where $k\geq 3$ and $m\geq 2$ are fixed and $n$ tends to infinity. If $m$ is a prime power, this lower bound matches (up to lower order terms) the previously known upper bound for the maximum size of a $k$-colored sum-free set in $\Zmn$. This generalizes a result of Kleinberg-Sawin-Speyer for the case $k=3$ and as part of our proof we also generalize a result by Pebody that was used in the work of Kleinberg-Sawin-Speyer. Both of these generalizations require several key new ideas.
\end{abstract}

\section{Introduction}

In 2016, Ellenberg and Gijswijt \cite{EG16} made an enormous breakthrough on the ``cap-set problem''. This problem asks about the largest size of a subset of $\mathbb{F}_3^n$ that does not contain a three-term arithmetic progression. Ellenberg and Gijswijt \cite{EG16} proved that any such set has size at most $o(2.756^n)$. Their proof uses a new polynomial method developed by Croot, Lev and Pach \cite{CLP16} for the analogous problem in $\mathbb{Z}_4^n$. The preprint of Ellenberg and Gijswijt appeared just a few weeks after the one of Croot, Lev and Pach, and subsequently a lot more activity evolved around these new ideas (see \cite{BENNETT18,BCCGNSU16,DM18, ELLENBERG17, FL17, FLS17, FS18, GS16, GREEN17, HEGEDUS18, KSS17,LOVETT18, NASLUND18, NASLUND18b, NS17,  PEBODY17, PETROV16, PP18, SAWIN18, Taoblog16}).

Soon after the preprint of Ellenberg and Gijswijt \cite{EG16} appeared, Blasiak, Church, Cohn, Grochow, Naslund, Sawin, Umans \cite{BCCGNSU16} and independently Alon noticed that the argument of Ellenberg and Gijswijt can also be used to obtain an upper bound on the size of $k$-colored sum-free sets in $\Fpn$ with $k=3$ (see the following definition)\footnote{Here, and whenever we write $\Fp$ throughout this introduction, we assume $p$ to be prime. Since all problems we consider only use the additive structure of $\Fpn$, it is not interesting to consider $\Fp$ for prime powers $p$ instead of primes. We will however consider $\ZZ_m$ where $m$ is a prime power, or any integer.}.

\begin{definition} Let $G$ be an abelian group and let $k\geq 3$. A \emph{$k$-colored sum-free set} in $G$ is a collection of $k$-tuples $(x_{1,j}, x_{2,j}, \dots, x_{k,j})_{j=1}^L$ of elements of $G$ such that for all $j_1,\dots, j_k\in \lbrace 1,\dots, L\rbrace$
\[x_{1,j_1}+x_{2,j_2}+\dots +x_{k,j_k}=0 \quad \text{ if and only if }\quad  j_1=j_2=...=j_k.\]
The size of a $k$-colored sum-free set is the number of $k$-tuples it consists of.
\end{definition}

Alon, Shpilka and Umans \cite{ASU13} introduced the notion of $k$-colored sum-free sets in the case $k=3$ and studied its connections to certain approaches for fast matrix multiplication algorithms.

On his blog, Tao \cite{Taoblog16} published a reformulation of the proof of Ellenberg and Gijswijt, in which he introduced what was later called the slice rank of a tensor. Tao's slice rank method immediately gives upper bounds for the size of $k$-colored sum-free sets in $\Fpn$ for any $k\geq 3$. In order to state these upper bounds, set
\[\Gamma_{m,k}=\min_{0<\gamma<1}\frac{1+\gamma+\dots+\gamma^{m-1}}{\gamma^{(m-1)/k}}\]
for integers $m\geq 2$ and $k\geq 3$. Note that $\Gamma_{m,k}<m$ (since at $\gamma=1$ the function has value $m$ and positive derivative). Furthermore, the function tends to infinity when $\gamma\to 0$, hence the minimum value $\Gamma_{m,k}$ is indeed attained for some $0<\gamma_{m,k}<1$ (one can show that there is a unique $0<\gamma_{m,k}<1$ where the minimum is attained, but this is not necessary for our purposes).

Tao's slice rank method \cite{Taoblog16}, together with the arguments from Blasiak et al.\ \cite{BCCGNSU16}, gives the following upper bound for the size of $k$-colored sum-free sets in $\Zmn$ for prime powers $m$. For the reader's convenience, we give a proof of Theorem \ref{thmupperbound} in Section \ref{sect-upperbound} (see also \cite[Theorem 4]{NASLUND18}, which is a very similar theorem).

\begin{theorem} \label{thmupperbound}
For every prime power $m$ and every integer $k\geq 3$, the size of any $k$-colored sum-free set in $\Zmn$ is at most $(\Gamma_{m,k})^n$.
\end{theorem}

Our main result is the following lower bound for the maximum size of a $k$-colored sum-free set in $\Zmn$, where $k\geq 3$ and $m\geq 2$ are fixed and $n$ tends to infinity. If $m$ is a prime power, this lower bound matches (up to lower order terms) the upper bound in Theorem \ref{thmupperbound}. Thus, we essentially determine  the maximum size of a $k$-colored sum-free set in $\Zmn$, if $m$ is a fixed prime power and $n$ tends to infinity.

\begin{theorem} \label{thmmain}
Let $m\geq 2$ and $k\geq 3$ be fixed integers. Then there exists a $k$-colored sum-free set in $\Zmn$ with size at least $(\Gamma_{m,k})^{n-O(\sqrt{n})}$.
\end{theorem}

Note that in Theorem \ref{thmmain}, the constant factor of the $O(\sqrt{n})$-term does depend on $m$ and $k$.

The case $(k,m)=(3,2)$ in Theorem \ref{thmmain} was proved by Fu and Kleinberg \cite{FK14}, building on work of Coppersmith and
Winograd \cite{CW90} in the context of fast matrix multiplication algorithms. After Blasiak et al.\ \cite{BCCGNSU16} established the upper bound, Kleinberg, Sawin, and Speyer \cite{KSS17} proved Theorem \ref{thmmain} for $k=3$ and any $m\geq 2$. Their proof uses a statement that had been formulated as a conjecture in an earlier version of their paper and was then proved by Pebody \cite{PEBODY17}. In order to make the statement precise, we need some more notation.

For integers $r\geq 0$ and $\l\geq 2$, set
\[T_{r,\l}=\{ (a_1,a_2,...,a_\l) \in \mathbb{Z}^\l\mid a_1+\dots+a_\l= r,\  a_1,\dots,a_\l\geq 0\}.\]
Given a probability distribution $\tau$ on $T_{r,\l}$, one obtains $\l$ probability distributions on the set $\{ 0,\dots,r\}$ by taking the projections to the different coordinates.
For an $\l$-tuple $(a_1,\dots , a_\l)$ and a permutation $\sigma\in S_\l$, let $(a_1,\dots , a_\l)^\sigma = (a_{\sigma(1)},\dots,a_{\sigma(\l)})$
be the $\l$-tuple obtained from $(a_1,\dots,a_\l)$ by permuting the coordinates according to $\sigma$. A probability distribution $\tau$ on $T_{r,\l}$ is called \emph{$S_\l$-symmetric} if $\tau(a_1,\dots,a_\l)=\tau((a_1,\dots , a_\l)^{\sigma})$ for all $(a_1,\dots , a_\l)\in T_{r,\l}$ and all $\sigma\in S_\l$.
For an $S_\l$-symmetric probability distribution $\tau$ on $T_{r,\l}$, the $\l$ projections to the individual coordinates all give the same probability distribution $\mu(\tau)$ on $\{ 0,\dots,r\}$ and this distribution is called the \emph{marginal} of $\tau$. For every $a\in \lbrace 0,\dots,r\rbrace$, the distribution $\mu(\tau)$ satisfies
\[\mu(\tau)(a)=\sum_{\substack{a_2,\dots,a_\l\in \lbrace 0,\dots,r\rbrace\\a+a_2+\dots+a_\l=r}}\tau(a,a_2,\dots,a_\l).\]
For integers $m\geq 2$ and $k\geq 3$, let $\nu_{m,k}$ be the probability distribution on $\lbrace 0,\dots,m-1\rbrace$ given by
\[\nu_{m,k}(i)=\frac{\gamma_{m,k}^i}{1+\gamma_{m,k}+\dots+\gamma_{m,k}^{m-1}}.\]
The probability distribution $\nu_{m,k}$ has expectation $(m-1)/k$ and entropy $\log \Gamma_{m,k}$ (see Lemma \ref{lemmanupk}). One can also show that among all probability distributions on $\lbrace 0,\dots,m-1\rbrace$ with expectation $(m-1)/k$, the distribution $\nu_{m,k}$ has the maximum entropy, which gives some motivation for considering this particular distribution.

The conjecture in the first version of the paper \cite{KSS17} of Kleinberg, Sawin, and Speyer stated that for every $m\geq 2$ the probability distribution $\nu_{m,3}$ occurs as the marginal of an $S_3$-symmetric probability distribution on $T_{m-1,3}$. As mentioned above, this was proved by Pebody \cite{PEBODY17}. Norin \cite{NORIN16} also proposed a proof.

In order to prove Theorem \ref{thmmain}, we need a generalization of the result of Pebody to $k>3$. The following theorem generalizes a slightly stronger version of the statement to $k>3$.

\begin{theorem}\label{theo-marginal} For all integers $m\geq 2$ and $k\geq 3$, the probability distribution $\nu_{m,k}$ occurs as the marginal of an $S_k$-symmetric probability distribution $\tau_{m,k}$ on $T_{m-1,k}$ with $\tau_{m,k}(t)>0$ for every $t\in T_{m-1,k}$.
\end{theorem}

Our proof of Theorem \ref{theo-marginal} was inspired by the first version of Pebody's proof \cite{PEBODY16} of the conjecture of Kleinberg, Sawin, and Speyer. However, our proof is not a direct generalization of Pebody's work and if one restricts to $k=3$ in our proof, one obtains a significantly different proof. In particular, we avoid the large case analysis in \cite{PEBODY16}, which would become even larger when trying to generalize it to $k>3$. Pebody later replaced the first version of his proof by yet another, much shorter proof \cite{PEBODY17}. However, it seems to be difficult to find an equally short and clean argument for the case $k>3$.

Using the upper bound on the size of $3$-colored sum-free sets in $\FF_p^n$, Fox and the first author \cite{FL17} proved a polynomial bound for the arithmetic triangle removal lemma in $\mathbb{F}_p^n$. The result of Kleinberg, Sawin, and Speyer \cite{KSS17} then implies that the exponent in the bound is sharp. In the case of $k>3$, polynomial bounds for the arithmetic $k$-cycle removal lemma in $\mathbb{F}_p^n$ were given by Fox and both authors \cite{FLS17}. Similarly as in the triangle case, Theorem \ref{thmmain} implies lower bounds on the possible exponents in the arithmetic $k$-cycle removal lemma in $\mathbb{F}_p^n$, but they do not match the exponents that Fox and the authors obtained in \cite{FLS17}. It would be interesting to close this gap and determine the optimal exponent for the arithmetic $k$-cycle removal lemma in $\mathbb{F}_p^n$.

Arithmetic removal lemmas were introduced by Green in 2005 \cite{GREEN05}, and since then the problem of improving the bounds in arithmetic removal lemmas has been widely studied \cite{BGRS12, BX09, FOX11, FL17, FK14, HX15}. This is in part due to the close connection to property testing. Indeed, the construction of Fu and Kleinberg \cite{FK14} establishing Theorem \ref{thmmain} in the case $(k,m)=(3,2)$, as well as earlier work of Bhattacharya and Xie \cite{BX09} in this direction, were presented as proving limits on the possible efficiency of randomized algorithms that test triangle freeness in $\mathbb{F}_2^n$. Theorem \ref{thmmain} gives a limit on the possible efficiency of testing $k$-cycle-freeness in $\Fpn$. For more details see \cite{BX09}, \cite{FK14}.

It is worth noting that the proof of Theorem \ref{thmmain} is not a straightforward generalization of the work of Kleinberg, Sawin, and Speyer \cite{KSS17} for the case $k=3$. Their argument uses a random sampling process to find a large $3$-colored sum-free set within a certain collection of 3-tuples. Although our approach for proving Theorem \ref{thmmain} is the same as in \cite{KSS17}, serious challenges arise with the probabilistic sampling argument. The main difficulty in generalizing the work in \cite{KSS17} arises in proving Proposition \ref{prop-xprimecount}, which states that there cannot be too many special pairs of $k$-tuples with increased conditional probabilities in the sampling process. This fact makes the probabilistic sampling argument work. The proposition has a much simpler proof in the case $k=3$ than in the case $k>3$. The most crucial tool for the proof of the proposition for $k>3$ is an entropy inequality that we will introduce in Section \ref{sect-xprime-count}. Furthermore, in the case of $k=3$, one also only needs a much weaker version of Proposition \ref{proprounding}, and the corresponding argument is just a single paragraph in \cite{KSS17}.

A rough outline of the proof of Theorem \ref{thmmain} is as follows. We first reduce Theorem \ref{thmmain} to a similar statement for $k$-tuples of vectors in $\Zn$ summing to $(m-1)\cdot \ones$, the vector with each coordinate equal to $m-1$. We then start with a set $X_0$ of vectors in $\lbrace 0,\dots,m-1 \rbrace^n$ such that for each vector in $X_0$ the distribution of its entries is roughly the distribution $\nu_{m,k}$ on $\lbrace 0,\dots,m-1 \rbrace$ that we defined above. Using  Theorem \ref{theo-marginal}, we can ensure that there are $k$-tuples $(x_1,\dots,x_k)\in X_0^k$ with $x_1+\dots+x_k=(m-1)\cdot \ones$. We would like to find a large collection $(x_{1,j}, x_{2,j}, \dots, x_{k,j})_{j=1}^L$ of such $k$-tuples in $X_0^k$ such that the only solutions to $x_{1,j_1}+x_{2,j_2}+\dots +x_{k,j_k}=(m-1)\cdot \ones$ are $j_1=\dots=j_k$.

In order to do so, we perform a random sampling argument as in \cite{KSS17}. More specifically, we consider carefully chosen random subsets $X_1,\dots,X_k$ of $X_0$. Roughly speaking, $X_1,\dots,X_k\su X_0$ are obtained from a $k$-colored sum-free set in $\ZZ_P$ for a carefully chosen prime $P$ via considering inverse images under certain randomly chosen affine-linear maps $\ZZ^n \to \ZZ_P$. We then consider those $k$-tuples $(x_1,\dots,x_k)\in X_1\times\dots\times X_k$ with $x_1+\dots+x_k=(m-1)\cdot \ones$. We prove that, in expectation, there is a large number of ``isolated'' such $k$-tuples, i.e.\ $k$-tuples $(x_1,\dots,x_k)\in X_1\times\dots\times X_k$ with $x_1+\dots+x_k=(m-1)\cdot \ones$ that do not share the same element $x_i$ with any other such $k$-tuple in $X_1\times\dots\times X_k$ (for any $i$). These isolated $k$-tuples will form the desired collection $(x_{1,j}, x_{2,j}, \dots, x_{k,j})_{j=1}^L$ of $k$-tuples of vectors in $\Zn$. Although this is the same strategy as for $k=3$ in \cite{KSS17}, proving that the expected number of isolated $k$-tuples is large is much harder for $k>3$ than for $k=3$. The argument for $k>3$ crucially relies on our new entropy inequality in Section \ref{sect-xprime-count}.

This paper is organized as follows. The first part of the paper is devoted to proving Theorem \ref{thmmain} assuming Theorem \ref{theo-marginal}. We start with some preliminaries about entropy in Section \ref{sect-entropy-prelim}. Afterwards, we give the proof of Theorem \ref{thmmain} in Section \ref{sect-proof-main}, but we postpone several lemmas and propositions to Sections \ref{sect-xprime-count} and \ref{sect-lemmas1}. In particular, the proof of Propostion \ref{prop-xprimecount} is the main difficulty in the first part of this paper and takes up all of Section \ref{sect-xprime-count}. 
In the second part of the paper, starting from Section \ref{sect-marginal}, we prove Theorem \ref{theo-marginal}. Finally, in Section \ref{sect-upperbound} we give a proof of Theorem \ref{thmupperbound} for the reader's convenience.

\textit{Notation.} All logarithms are base $e$. The set of non-negative real numbers is denoted by $\mathbb{R}_{\geq 0}$, and $\ones$ denotes the all-ones vector with $n$ entries. For any integer $a$, let $\indic_a$ denote the indicator function of $a$. That is, $\indic_a(x)=1$ if $x=a$ and $\indic_a(x)=0$ otherwise.

\section{Preliminaries on Entropy}
\label{sect-entropy-prelim}

This section covers some preliminaries about entropy. Some facts are stated without proof, their proofs can be found, for example, in \cite[chapter 15.7]{AlonSpencer}.
All random variables in this section are assumed to be random variables defined on a finite ground set.

Given a probability distribution $\omega$ on a finite set $S$, the \emph{entropy} of $\omega$ is defined as
\[\HH(\omega)=\sum_{s \in S} -\omega({s}) \log \omega({s}).\]
Note that $\HH(\omega)\geq 0$. With a slight abuse of notation, we will also write $\HH(X)$ instead of $\HH(\omega)$, if $X$ is a random variable on $S$ with distribution $\omega$. Given several random variables $X_1,X_2,\dots,X_k$, we will write $\HH(X_1,X_2,\dots,X_k)$ for the entropy of the joint distribution of the variables $X_1,X_2,\dots,X_k$ (which is a distribution on $S^k$).

Given a finite set $S$, the uniform distribution on $S$ is the (unique) distribution $\omega$ on $S$ with maximum entropy. So for every probability distribution $\omega$ on $S$ we have $\HH(\omega)\leq \log(\vert S\vert)$.
For a given finite set $S$, entropy is a concave function on probability distributions on $S$. In other words, for any two  probability distributions $\omega_0$ and $\omega_1$ on $S$ and any real number $0\leq t\leq 1$ we have
\[\HH(t\omega_1+(1-t)\omega_0)\geq t\HH(\omega_1)+(1-t)\HH(\omega_0).\]

Suppose that $Y$ is a random variable on a finite set $S$, and $X$ is a random variable on any finite set. Then the \emph{conditional entropy} $\HH(X\mid Y)$ is defined as
\[\HH(X\mid Y)=\sum_{s \in S} \P(Y=s) \HH(X\mid Y=s).\]
Here, $\HH(X\mid Y=s)$ denotes the entropy of the conditional distribution of $(X|Y=s)$. One can show that
\begin{equation}\label{eq-entropy-conditional-difference}
\HH(X\mid Y)=\HH(X,Y)-\HH(Y),
\end{equation}
and therefore
\[\HH(X,Y)\geq \HH(Y).\]
Given a sequence of random variables $X_1,X_2,...,X_m$,  repeatedly applying (\ref{eq-entropy-conditional-difference}) yields
\[\HH(X_1,X_2,\dots,X_m)=\HH(X_1)+\HH(X_2\mid X_1)+\dots +\HH(X_m\mid X_1,X_2,...,X_{m-1}).\]
For any random variables $X,Y,Z$, we have
\[\HH(X\mid Y,Z) \le \HH(X\mid Y).\]
If $X$ and $Y$ are random variables such that $Y$ is completely determined by $X$, we have
\begin{equation}\label{eq-entropy-determined}
\HH(X)=\HH(X,Y)\geq \HH(Y)
\end{equation}
and
\begin{equation}\label{eq-entropy-determined-difference}
\HH(X\mid Y)=\HH(X,Y)-\HH(Y)=\HH(X)-\HH(Y).
\end{equation}

If $\tau$ is an $S_\l$-symmetric probability distribution on $T_{r,\l}$ for some $r\geq 0$ and $\l \geq 2$, then the marginal $\mu(\tau)$ is the projection of $\tau$ to the first coordinate. So by (\ref{eq-entropy-determined}) we have
\begin{equation}\label{eq-entropyprojection}
\HH(\mu(\tau))\leq \HH(\tau).
\end{equation}

The following lemma is a well-known approximation of multinomial coefficients, see, e.g., \cite[Lemma 3]{KSS17}.

\begin{lemma} \label{lemmaentropycounting}
Let $\omega$ be a probability distribution on a finite set $S$, and let $n$ be a positive integer. Assume that for every $s\in S$ the probability $\omega(s)$ is an integer multiple of $1/n$. Let $M$ be the number of sequences $s_1,\dots,s_n$ of elements of $S$ in which each element $s\in S$ occurs exactly $\omega(s)n$ times (this means, sampling a random element from the sequence $s_1,\dots,s_n$ recovers the probability distribution $\omega$ on $S$). Then $M$ satisfies
\[\frac{e^{\HH(\omega)n}}{e^{\vert S\vert}n^{\vert S\vert}}\leq M\leq e^{\HH(\omega)n}.\]
\end{lemma}

\begin{proof}We can assume that $\omega(s)>0$ for every $s\in S$, because we can delete all elements $s\in S$ with $\omega(s)=0$.
First, note that $M$ can be described as a multinomial coefficient:
\[M=\binom{n}{(\omega(s)n)_{s\in S}}.\]
For the lower bound, we use the simple Stirling approximation bounds, namely
\[\sqrt{2\pi \l} \left(\frac{\l}{e}\right)^\l \le \l! \le e \sqrt{\l} \left(\frac{\l}{e}\right)^\l.\]
for any positive integer $\l$. Now
\begin{multline*}
M=\binom{n}{(\omega(s)n)_{s\in S}}= \frac{n!}{\prod_{s\in S}(\omega(s)n)!} \ge \frac{\sqrt{2\pi n}\left(\frac{n}{e}\right)^n }{\prod_{s\in S}\left(e\sqrt{\omega(s) n}\left(\frac{\omega(s)n}{e}\right)^{\omega(s)n}\right)}\\ 
=\frac{\sqrt{2\pi}}{e^{\vert S\vert}}\cdot \frac{\sqrt{n}}{\sqrt{\prod_{s\in S} \omega(s)n}}\cdot \frac{\left(\frac{n}{e}\right)^n}{\prod_{s\in S}\left(\frac{\omega(s) n}{e}\right)^{\omega(s)n}} \ge 
\frac{1}{e^{\vert S\vert} n^{\vert S\vert}}\cdot \prod_{s\in S}\omega(s)^{-\omega(s)n}=\frac{e^{\HH(\omega)n}}{e^{\vert S\vert}n^{\vert S\vert}}.
\end{multline*}
For the upper bound, note that we have by the multinomial sum theorem
\[1=\left(\sum_{s\in S}\omega(s)\right)^n\geq \binom{n}{(\omega(s)n)_{s\in S}} \prod_{s\in S}\omega(s)^{\omega(s)n}.\]
Here we only considered those terms in the expansion of $(\sum_{s\in S}\omega(s))^n$ that contain each term $\omega(s)$ precisely $\omega(s)n$ times. Now, rearranging yields
\[\binom{n}{(\omega(s)n)_{s\in S}} \le \prod_{s\in S}\omega(s)^{-\omega(s)n}=e^{\HH(\omega)n},\]
as desired.\end{proof}

The next lemma is also about counting sequences of elements of a set $S$, but under more restrictive conditions.

\begin{lemma} \label{lemmaentropycounting2}
Let $f:S\to S'$ be any function between finite sets $S$ and $S'$. Furthermore, let $\omega$ be a probability distribution on $S$ and let $z$ be a random variable on $S$ with distribution $\omega$. Let $n$ be a positive integer and let us fix $s_1',\dots, s_n'\in S'$. Now let $M$ be the number of sequences $s_1,\dots,s_n$ of elements of $S$ in which each element $s\in S$ occurs exactly $\omega(s)n$ times and such that $f(s_j)=s_j'$ for $1\leq j\leq n$. Then $M$  satisfies
\[\ M\leq e^{\HH(z\mid f(z))n}=e^{\HH(z)n-\HH(f(z))n}.\]
\end{lemma}
\begin{proof}
For every $s'\in S'$, set
\[J_{s'}=\lbrace j\in \lbrace 1,\dots,n\rbrace\mid s_j'=s'\rbrace.\]
If $M=0$, the statement is trivially true. Hence we may assume that there exists at least one sequence $s_1,\dots,s_n\in S$ with the desired properties. Then in particular all the numbers $\omega(s)n$ for $s\in S$ are integers.

We claim that $\vert J_{s'}\vert=\sum_{s\in f^{-1}(s')}\omega(s)n$ for every $s'\in S'$. To see this, let us temporarily fix a sequence $s_1,\dots, s_n\in S$ satisfying all of the conditions in the lemma and let $s'\in S$. Since each $s\in f^{-1}(s')$ occurs exactly $\omega(s)n$ times in the sequence $s_1,\dots, s_n$, there are exactly $\sum_{s\in f^{-1}(s')}\omega(s)n$ choices for $j\in \lbrace 1,\dots,n\rbrace$ such that $s_j'=f(s_j)$ equals $s'$. Thus $\vert J_{s'}\vert=\sum_{s\in f^{-1}(s')}\omega(s)n$ as desired.

In order to form a sequence $s_1,\dots, s_n$ with the desired conditions, for each $s'\in S'$ we must distribute the elements of $f^{-1}(s')$ with the desired multiplicities among the index set $J_{s'}$. So it is not hard to see that
\[M=\prod_{s'\in S'}\binom{\vert J_{s'}\vert}{(\omega(s)n)_{s\in f^{-1}(s')}}.\]
Note that for each $s\in S$ we have $\omega(s)=\P(z=s)$ and therefore
\[\vert J_{s'}\vert=\sum_{s\in f^{-1}(s')}\omega(s)n=\sum_{s\in f^{-1}(s')}\P(z=s)\cdot n=\P(z\in f^{-1}(s'))\cdot n=\P(f(z)=s')\cdot n\]
for every $s'\in S'$. Furthermore if $s'\in S'$ and $s\in f^{-1}(s')$, then
\[\omega(s)n=\P(z=s)\cdot n=\frac{\P(z=s)}{\P(f(z)=s')}\cdot \vert J_{s'}\vert=\P(z=s\mid f(z)=s')\cdot \vert J_{s'}\vert.\]
In particular, $\P(z=s\mid f(z)=s')$ is an integer multiple of $1/\vert J_{s'}\vert$.
Now, we obtain
\[M=\prod_{s'\in S'}\binom{\vert J_{s'}\vert}{\left(\P(z=s\mid f(z)=s')\cdot \vert J_{s'}\vert\right)_{s\in f^{-1}(s')}}.\]
For each $s'\in S'$, we can apply the upper bound in Lemma \ref{lemmaentropycounting} to the distribution of $z$ conditioned on $f(z)=s'$ and obtain
\[\binom{\vert J_{s'}\vert}{\left(\P(z=s\mid f(z)=s')\cdot \vert J_{s'}\vert\right)_{s\in f^{-1}(s')}}\leq \exp(\HH(z\mid f(z)=s')\cdot \vert J_{s'}\vert).\]
Thus,
\[M\leq \exp\left(\sum_{s'\in S'}\HH(z\mid f(z)=s')\cdot \vert J_{s'}\vert\right)=\exp\left(\sum_{s'\in S'}\HH(z\mid f(z)=s')\P(f(z)=s')\cdot n\right)
= \exp(\HH(z\mid f(z)) n),\]
as desired. Note that $\HH(z\mid f(z))=\HH(z)-\HH(f(z))$ by (\ref{eq-entropy-determined-difference}).\end{proof}

Finally, we need one more lemma. Basically, this lemma states that perturbing a probability distribution slightly does not change the entropy very much.

\begin{lemma} \label{lemmaentropiesclose}
Let $\omega_0$ and $\omega_1$ be two distributions on a finite set $S$, and suppose that $c>0$ satisfies $\omega_0(s)\geq c$ and $\omega_1(s)\geq c$ for every $s\in S$. Then
\[|\HH(\omega_1)-\HH(\omega_0)| \le \|\omega_1-\omega_0\|_1\log(1/c).\]
\end{lemma}
\begin{proof}
For any real number $0\leq t\leq 1$, let $\omega_t=t\omega_1+(1-t)\omega_0$ (note that for $t=0$ and $t=1$ we indeed recover $\omega_0$ and $\omega_1$). Then for each $0\leq t\leq 1$ and each $s\in S$ we have
\[\omega_t(s)=t\omega_1(s)+(1-t)\omega_0(s)=t(\omega_1(s)-\omega_0(s))+\omega_0(s).\]
In particular, $\omega_t(s)=t\omega_1(s)+(1-t)\omega_0(s)\geq c$. For each $0\leq t\leq 1$, set $f(t)=\HH(\omega_t)$, so
\[f(t)=-\sum_{s\in S}\omega_t(s)\log \omega_t(s).\]
It is easy to check that $f$ is a continuous function on the interval $[0,1]$ and it is differentiable on the open interval $(0,1)$. For each $t\in (0,1)$ we have (using that $\sum_s \omega_0(s)=\sum_s \omega_1(s)=1$)
\[f'(t)=-\sum_{s\in S} \left((\omega_1(s)-\omega_0(s))\log\omega_t(s)+\omega_t(s)\frac{1}{\omega_t(s)}(\omega_1(s)-\omega_0(s))\right)=-\sum_{s\in S}(\omega_1(s)-\omega_0(s))\log\omega_t(s).\]
Hence
\[\vert f'(t)\vert =\left\vert\sum_{s\in S} (\omega_1(s)-\omega_0(s))\log (1/\omega_t(s))\right\vert\le \sum_{s\in S} \vert \omega_1(s)-\omega_0(s)\vert\log (1/c) = \Vert\omega_1-\omega_0\Vert_1\log(1/c).\]
By the mean value theorem, we now obtain
\[\vert\HH(\omega_1)-\HH(\omega_0)\vert = \vert f(1)-f(0)\vert \le \sup_{t \in (0,1)} \vert f'(t)\vert \le \|\omega_1-\omega_0\|_1\log(1/c),\]
as desired.\end{proof}

\section{Proof of Theorem \ref{thmmain}}
\label{sect-proof-main}

The goal of this section is to prove Theorem \ref{thmmain}, that means to show that there is a sufficiently large $k$-colored sum-free set in $\Zmn$. We will often use the bound $\vert T_{m-1,k}\vert\leq m^k$, which follows from the fact that $T_{m-1,k}\su \lbrace 0,\dots,m-1\rbrace^k$. The following proposition states that there is a large collection of $k$-tuples in $\Zn$ with certain conditions, the first of which is very similar to the condition for a $k$-colored sum-free set.

\begin{proposition} \label{proplowerboundinZ}
Let $m\geq 2$ and $k\geq 3$ be fixed, and let $n$ be divisible by $k$ and tending to infinity. Then there exists a collection of $k$-tuples $(x_{1,j}, x_{2,j}, \dots, x_{k,j})_{j=1}^L$ of elements of $\Zn$ with size $L\geq (\Gamma_{m,k})^{n-O(\sqrt{n})}$ such that
\begin{itemize}
\item for all $j_1,\dots, j_k\in \lbrace 1,\dots, L\rbrace$
\[x_{1,j_1}+x_{2,j_2}+\dots +x_{k,j_k}=(m-1)\cdot \ones \quad \text{ if and only if }\quad  j_1=j_2=...=j_k,\]
\item all coordinates of all $x_{i,j}$ are in the set $\lbrace 0,1,\dots,m-1\rbrace$, and
\item for each $x_{i,j}$ the sum of its $n$ coordinates equals $(m-1)n/k$.
\end{itemize}
\end{proposition}

Let us now prove that Proposition \ref{proplowerboundinZ} implies Theorem \ref{thmmain}. Afterwards, we will prove Proposition \ref{proplowerboundinZ}.

\begin{proof}[Proof of Theorem \ref{thmmain} assuming Proposition \ref{proplowerboundinZ}]
Let $m\geq 2$ and $k\geq 3$ be fixed. Assume for now that $n$ is divisible by $k$.  Let us consider a collection of $k$-tuples $(x_{1,j}, x_{2,j}, \dots, x_{k,j})_{j=1}^L$ of elements of $\Zn$ as in Proposition \ref{proplowerboundinZ} (in particular, $L\geq (\Gamma_{m,k})^{n-O(\sqrt{n})}$). Using this collection, we will construct a $k$-colored sum-free set $(y_{1,j}, y_{2,j}, \dots, y_{k,j})_{j=1}^L$ in $\Zmn$. For $1\leq i\leq k-1$, and for any $j$, let $y_{i,j}$ be the projection of $x_{i,j}\in \Zn$ to $\Zmn$. For $i=k$, let $y_{k,j}$ be the projection of $x_{k,j}-(m-1)\cdot \ones$ to $\Zmn$.

Let us now check that $(y_{1,j}, y_{2,j}, \dots, y_{k,j})_{j=1}^L$ is indeed a $k$-colored sum-free set in $\Zmn$. Let $j_1,\dots,j_k\in \lbrace 1,\dots, L\rbrace$ be such that $y_{1,j_1}+y_{2,j_2}+\dots +y_{k,j_k}=0$ in $\Zmn$. By the choice of the $y_{i,j}$, each of the $n$ coordinates of the vector
\[x_{1,j_1}+x_{2,j_2}+\dots +x_{k-1,j_{k-1}}+x_{k,j_k}-(m-1)\cdot\ones\]
is divisible by $m$. Thus, every coordinate of $x_{1,j_1}+x_{2,j_2}+\dots +x_{k,j_k}\in \Zn$ has remainder $m-1$ upon division by $m$. Since each coordinate is non-negative, this implies in particular that each coordinate of $x_{1,j_1}+x_{2,j_2}+\dots +x_{k,j_k}$ is at least $m-1$. On the other hand, the sum of the coordinates of each $x_{i,j_i}$ equals $(m-1)n/k$, so the sum of the coordinates of $x_{1,j_1}+x_{2,j_2}+\dots +x_{k,j_k}$ equals $(m-1)n$. Since each is at least $m-1$, this implies that all the coordinates of $x_{1,j_1}+x_{2,j_2}+\dots +x_{k,j_k}$ are equal to $m-1$. Hence
\[x_{1,j_1}+x_{2,j_2}+\dots +x_{k,j_k}=(m-1)\cdot\ones.\]
So by the first property listed in Proposition \ref{proplowerboundinZ} we must have $j_1=j_2=...=j_k$.

For the converse, assume that $j_1=j_2=...=j_k$. Then 
\[x_{1,j_1}+x_{2,j_2}+\dots +x_{k-1,j_{k-1}}+x_{k,j_k}-(m-1)\cdot\ones=0\]
in $\Zn$, hence $y_{1,j_1}+y_{2,j_2}+\dots +y_{k,j_k}=0$ in $\Zmn$ as well.

Thus, we have constructed a $k$-colored sum-free set in $\Zmn$ of size 
\[L\geq (\Gamma_{m,k})^{n-O(\sqrt{n})},\]
if $n$ is divisible by $k$. Note that by embedding $\Zmn$ into $\mathbb{Z}_m^{n'}$ for $n'>n$, any $k$-colored sum-free set in $\Zmn$ gives rise to a $k$-colored sum-free set in $\mathbb{Z}_m^{n'}$ of the same size. Thus, for all $n$ we obtain a $k$-colored sum-free set in $\Zmn$ of size at least $(\Gamma_{m,k})^{n-O(\sqrt{n})}$.\end{proof}

The rest of this section will be devoted to proving Proposition \ref{proplowerboundinZ}. However, in some sense this section covers only the general steps for the proof, the major difficulty lies in the proofs of several lemmas and propositions that are postponed to the next two sections.

\begin{proof}[Proof of Proposition \ref{proplowerboundinZ}] Let $m\geq 2$ and $k\geq 3$ be fixed and let $n$ be divisible by $k$ and sufficiently large (in terms of $m$ and $k$). Our goal is to find a sufficiently large collection of $k$-tuples with the properties listed in Proposition \ref{proplowerboundinZ}.
Recall that in the introduction we defined a specific probability distribution $\nu_{m,k}$ on $\lbrace 0,\dots, m-1\rbrace$.

\begin{lemma}\label{lemmanupk}The probability distribution $\nu_{m,k}$ has expectation $\E(\nu_{m,k})=(m-1)/k$, and its entropy is $\HH(\nu_{m,k})=\log \Gamma_{m,k}$.
\end{lemma}
We will prove Lemma \ref{lemmanupk} in Section \ref{sect-lemmas1}. In fact, one can also prove that $\nu_{m,k}$ has the highest entropy among all probability distributions on $\lbrace 0,\dots, m-1\rbrace$ with expectation $(m-1)/k$. Although we will not need this fact for our argument, it still provides motivation as to why the distribution $\nu_{m,k}$ is relevant.

In order to use $\nu_{m,k}$ for constructing the desired collection of $k$-tuples in $\Zn$, we first need to round the probabilities in $\nu_{m,k}$ to rational numbers with denominator $n$. The following proposition basically states that a suitable rounding of $\nu_{m,k}$ exists.

\begin{proposition}\label{proprounding} Let $m\geq 2$, $k\geq 3$ and let  $n$ be divisible by $k$ and sufficiently large  (in terms of $m$ and $k$). Then there exist probability distributions $\nu$ on $\lbrace 0,\dots, m-1\rbrace$ and $\tau$ on $T_{m-1,k}$ with the following conditions:
\begin{itemize}
\item Each probability $\nu(i)$ for $i\in \lbrace 0,\dots, m-1\rbrace$ is an integer multiple of $1/n$.
\item $\nu$ has expectation $\E(\nu)=(m-1)/k$.
\item $\HH(\nu)\geq \log\Gamma_{m,k}-C_{m,k}/n$.
\item Each probability $\tau(t)$ for $t\in T_{m-1,k}$ is an integer multiple of $1/n$.
\item $\tau$ is $S_k$-symmetric and has marginal $\nu$.
\item We have $\HH(\tau')\leq \HH(\tau) + (D_{m,k}\log n)/n$ for every probability distribution $\tau'$ on $T_{m-1,k}$ with the property that all of the $k$ coordinate projections of $\tau'$ are equal to $\nu$.
\end{itemize}
Here $C_{m,k}>0$ and $D_{m,k}>0$ are constants only depending on $m$ and $k$.
\end{proposition}

The key ingredients for the proof of Proposition \ref{proprounding} are Theorem \ref{theo-marginal} and Lemma \ref{lemmanupk}. Starting from there, one needs to perform several rounding steps in order to obtain Proposition \ref{proprounding}. We will postpone the technical details of these rounding steps to Section \ref{sect-lemmas1}. The proof of Theorem \ref{theo-marginal} will be given in the second part of this paper, starting from Section \ref{sect-marginal}.

Let us fix $\nu$ and $\tau$ as in Proposition \ref{proprounding}. Now, let $X_0 \subseteq \lbrace 0,\dots,m-1\rbrace^{n}$ consist of those vectors ${x \in \lbrace 0,\dots,m-1\rbrace^{n}}$ that have exactly $\nu(i)n$ coordinates equal to $i$ for every $i=0,\dots,m-1$ (i.e. that contain exactly $\nu(0)n$ zeros, exactly $\nu(1)n$ ones and so on). Then for any $x\in X_0$, choosing one of the coordinates of $x$ uniformly at random recovers the probability distribution $\nu$ on $\lbrace 0,\dots,m-1\rbrace$.

Note that for each $x\in X_0$ all of its $n$ coordinates are in the set $\lbrace 0,\dots,m-1\rbrace$ and their sum is equal to $0\cdot \nu(0)n+\dots+(m-1)\cdot \nu(m-1)n=\E(\nu)n=n(m-1)/k$. So if we choose our desired collection of $k$-tuples in such a way that $x_{i,j}\in X_0$ for all $i$ and $j$, then the second and third condition in Proposition \ref{proplowerboundinZ} will automatically be satisfied.

In order to obtain the desired collection of $k$-tuples, we will use a probabilistic sampling argument. To describe this sampling, let us first (using Bertrand's Postulate) choose a prime number $P$ with
\begin{equation}\label{eq-def-P}
4n^{D_{m,k}+2m^k}\exp\left(\frac{\HH(\tau)-\HH(\nu)}{k-2}\cdot n\right) \le P\le n^{D_{m,k}+3m^k}\exp\left(\frac{\HH(\tau)-\HH(\nu)}{k-2}\cdot n\right).
\end{equation}
By (\ref{eq-entropyprojection}), we have $\HH(\tau)\geq \HH(\mu(\tau))=\HH(\nu)$. Hence, as long as $n$ is large enough, we have
\[P\geq n^{m^k}\geq (2k)!m^{2k}.\]
On the other hand $\HH(\tau)\leq \log (\vert T_{m-1,k}\vert)\leq \log (m^k)=k\log m$. Therefore, as long as $n$ is sufficiently large, we have $P\leq \exp(2k(\log m)n)$, so
\begin{equation}\label{eq-P-upper-bound}
\log P\leq 2k(\log m)n.
\end{equation}

We now define functions $\tilde{g}_1, \dots, \tilde{g}_k: X_0\to\ZZ^{n+k-1}$. For any $x\in X_0$, consider its coordinates $x=(x^{(1)},x^{(2)},...,x^{(n)})$. For $1\le i\le k-1$, we define $\tilde{g}_i(x)=(\tilde{g}_i(x)^{(1)},\tilde{g}_i(x)^{(2)},\dots, \tilde{g}_i(x)^{(n+k-1)}) \in \ZZ^{n+k-1}$ by
\[\tilde{g}_i(x)^{(j)}=\begin{cases}
x^{(j)} & \text{if } j \le n\\
1 & \text{if } j=n+i\\
0 & \text{if } j>n \text{ and }j \ne n+i.
\end{cases}\]
In other words, for $1\le i\le k-1$, the vector $\tilde{g}_i(x)$ can be obtained from $x$ by attaching the $i$-th standard basis vector in $\ZZ^{k-1}$ at the end.

For $i=k$, we define $\tilde{g}_k(x)=(\tilde{g}_k(x)^{(1)},\tilde{g}_k(x)^{(2)},\dots, \tilde{g}_k(x)^{(n+k-1)}) \in \ZZ^{n+k-1}$ by
\[\tilde{g}_k(x)^{(j)}=\begin{cases}
x^{(j)}-(m-1) & \text{if } j \le n\\
-1 & \text{if } j>n.
\end{cases}
\]

Note that for each $1\leq i\leq k$ and each $x\in X_0 \su \lbrace 0,\dots, m-1\rbrace^{n}$, all the coordinates of $\tilde{g}_i(x)$ have absolute value at most $m-1$. 
Let us also define functions $g_1, \dots, g_k: X_0\to \FPnk$ by taking $g_i(x)$ to be the projection of $\tilde{g}_i(x)$ to $\FPnk$ for every $1\leq i\leq k$ and every $x\in X_0$.
Given $x_1,\dots,x_k \in X_0$ with $x_1+\dots+x_k=(m-1)\cdot \ones$ in $\ZZ^n$ , we have $\tilde{g}_1(x_1)+\dots+\tilde{g}_k(x_k)=0$ in $\ZZ^{n+k-1}$ and therefore $g_1(x_1)+\dots+g_k(x_k)=0$ in $\FPnk$. However, $g_1(x_1),\dots, g_{k-1}(x_{k-1})$ are linearly independent over $\FP$ (because their last $k-1$ coordinates form the standard basis vectors in $\FPkmin$).

For our probabilistic sampling argument, we will use a large $k$-colored sum-free set in $\FP$. The following lemma ensures the existence of such a large $k$-colored sum-free set. The proof relies on a lemma by Alon {\cite[Lemma 3.1]{Alon01}}, which he proved using a modification of Behrend's construction \cite{Behrend46}.
\begin{lemma} \label{lemmaBehrendZM}
There exists a $k$-colored sum-free set $(y_{1,j}, y_{2,j}, \dots, y_{k,j})_{j=1}^R$ in $\FP$ of size
\begin{equation}\label{eq-R-large}
R \ge P \cdot \exp(-12\sqrt{\log P \log k}).
\end{equation}
\end{lemma}
\begin{proof} Recall that $P\geq k\geq 3$. By {\cite[Lemma 3.1]{Alon01}}, there exists a subset $Y\su \lbrace 1,\dots,\lfloor P/k\rfloor\rbrace$ of size at least
\[\vert Y\vert\geq \frac{\lfloor P/k\rfloor}{e^{10\sqrt{\log (\lfloor P/k\rfloor)\log (k-1)}}}\geq \frac{P}{2k}\cdot \exp(-10\sqrt{\log P \log k})\geq P \cdot \exp(-12\sqrt{\log P \log k})\]
such that the only solutions in $Y$ to the integer equation $y_1+\dots+y_{k-1}=(k-1)y_k$ satisfy $y_1=\dots=y_{k-1}=y_k$. Note that for $y_1,\dots,y_k\in Y\su \lbrace 1,\dots,\lfloor P/k\rfloor\rbrace$, both sides of the equation $y_1+\dots+y_{k-1}=(k-1)y_k$ are integers between $1$ and $P-1$. Hence the equation $y_1+\dots+y_{k-1}=(k-1)y_k$ holds over $\FP$ if and only if it holds in the integers. So let us interpret $Y$ as a subset of $\FP$. Then still the only solutions in $Y$ to the equation $y_1+\dots+y_{k-1}=(k-1)y_k$ over $\FP$ are $y_1=\dots=y_{k-1}=y_k$

Now, for each $y\in Y\su \FP$ consider the $k$-tuple $(y,\dots,y,-(k-1)y)\in \FPk$. Altogether these $k$-tuples form a $k$-colored sum-free set in $\FP$ of size $\vert Y\vert\geq P \cdot \exp(-12\sqrt{\log P \log k})$.\end{proof}

Let $(y_{1,j}, y_{2,j}, \dots, y_{k,j})_{j=1}^R$ be a $k$-colored sum-free set in $\FP$ as in the lemma. For $i=1,\dots, k$ set
\[Y_i=\lbrace y_{i,j} \mid 1\leq j\leq R\rbrace.\]
Since for a fixed $i$, the vectors $y_{i,j}$ for $1\leq j\leq R$ are all distinct, we have $\vert Y_i\vert =R$. Furthermore, by the definition of $k$-colored sum-free set, any $k$-tuple in $Y_1\times\dots\times Y_k$ summing to zero needs to be of the form $(y_{1,j}, y_{2,j}, \dots, y_{k,j})$ for some $j$.

In order to perform the desired sampling, let us now choose a random linear map $f:\FPnk \rightarrow \FP$. More precisely, we choose $f$ uniformly at random among all linear maps $\FPnk \rightarrow \FP$.
\begin{claim}\label{claim-independence}
Suppose that $v_1,v_2,\dots,v_\l \in \FPnk$ are linearly independent over $\FP$. Then the images $f(v_1),\dots,f(v_\l)$ are probabilistically independent and uniformly distributed over $\FP$.
\end{claim}
\begin{proof}
We can extend $v_1,v_2,\dots,v_\l$ to a basis $v_1,v_2,\dots,v_{n+k-1}$ of  $\FPnk$. Note that we can model the random choice of $f$ by mapping each $v_i$ to a random element of $\FP$ (uniformly, and independently for all $1\leq i\leq n+k-1$). By this description of the random experiment, the claim is clearly true.\end{proof}

We now define, for $1\leq i\leq k$,
\[X_i=\{x \in X_0 \mid f(g_i(x)) \in Y_i\}.\]

\begin{definition}\label{deficandidate}A \emph{candidate $k$-tuple} is a $k$-tuple $(x_1,x_2,\dots,x_k) \in X_1 \times X_2 \times \dots \times X_k$ with $x_1+x_2+\dots+ x_k=(m-1)\cdot \ones$. A candidate $k$-tuple is \emph{isolated}, if there is no other candidate $k$-tuple $(x_1',x_2',\dots,x_k')$ with $x_i'=x_i$ for some $1\leq i\leq k$.
\end{definition}

Let $(x_{1,j}, x_{2,j}, \dots, x_{k,j})_{j=1}^L$ be the collection of all isolated candidate $k$-tuples, where $L$ is the total number of isolated candidate $k$-tuples. Note that since the sets $X_i$ depend on the choice of the random map $f$, the notions in Definition \ref{deficandidate} also depend on this choice. In particular, $L$ is a random variable that depends on the random map $f$.
We claim that the collection $(x_{1,j}, x_{2,j}, \dots, x_{k,j})_{j=1}^L$ will always satisfy the first condition in Proposition \ref{proplowerboundinZ}.  For each $j\in \lbrace 1,\dots, L\rbrace$, we have $x_{1,j}+\dots+ x_{k,j}=(m-1)\cdot \ones$ by the first part of Definition \ref{deficandidate}. Also note that $x_{i,j}\neq x_{i,j'}$ whenever $j\neq j'$, since all the $(x_{1,j}, x_{2,j}, \dots, x_{k,j})$ are isolated. Now suppose that we have $j_1,\dots, j_k\in \lbrace 1,\dots, L\rbrace$ with $x_{1,j_1}+x_{2,j_2}+\dots +x_{k,j_k}=(m-1)\cdot \ones$. As $x_{i,j_i}\in X_i$ for each $1\leq i\leq k$, we obtain that $(x_{1,j_1}, x_{2,j_2},\dots ,x_{k,j_k})$ is a candidate $k$-tuple. But since all the candidate $k$-tuples $(x_{1,j}, x_{2,j}, \dots, x_{k,j})_{j=1}^L$ in the collection are isolated, $(x_{1,j_1}, x_{2,j_2},\dots ,x_{k,j_k})$ must be one of the $k$-tuples in the collection. Hence $j_1=\dots=j_k$. So the isolated candidate $k$-tuples indeed form a collection satisfying the first condition in Proposition \ref{proplowerboundinZ}.

Each of the vectors $x_{i,j}$ in the $k$-tuples in the collection satisfies $x_{i,j}\in X_i\su X_0$. As we have already seen above, this implies that the $n$ coordinates of $x_{i,j}$ are all in the set $\lbrace 0,\dots,m-1\rbrace$ and their sum is equal to $n(m-1)/k$. Thus, the second and third condition in Proposition \ref{proplowerboundinZ} are satisfied for the collection of isolated candidate $k$-tuples. So it only remains to prove that for at least one choice of the random map $f$, the number $L$ of isolated candidate $k$-tuples is sufficiently large. In particular, it suffices to show that the expected value of $L$ is sufficiently large.

The following proposition states that certain $k$-tuples have a good probability of being isolated candidate $k$-tuples. Note that if $x_1,\dots,x_k\in \lbrace 0,\dots,m-1\rbrace^n$ with $x_1+\dots+ x_k=(m-1)\cdot \ones$, then for every $1\leq j\leq n$, the $j$-th coordinates of these vectors satisfy $x_1^{(j)}+\dots+ x_k^{(j)}=m-1$. Hence $(x_1^{(j)},\dots,x_k^{(j)})\in T_{m-1,k}$ for every $1\leq j\leq n$.
\begin{proposition}\label{prop-isolated}
Let $x_1,\dots,x_k\in \lbrace 0,\dots,m-1\rbrace^n$ with $x_1+\dots+ x_k=(m-1)\cdot \ones$. Assume that for every $t\in T_{m-1,k}$ the number of $j\in \lbrace 1,\dots,n\rbrace$ with $(x_1^{(j)},\dots,x_k^{(j)})=t$ is precisely $\tau(t)n$. Then $x_1,\dots,x_k\in X_0$ and furthermore the probability that $(x_1,\dots,x_k)$ is an isolated candidate $k$-tuple is at least
\[\exp\left(\HH(\nu)n-\HH(\tau)n-25km\sqrt{n}\right).\]
\end{proposition}

We postpone the proof of this proposition for a little while, in order to first finish the proof of Proposition \ref{proplowerboundinZ}.

\begin{claim} \label{claimnumbertaudist} The number of $k$-tuples $(x_1,\dots,x_k)$ satisfying the assumptions of Proposition \ref{prop-isolated} is at least $e^{\HH(\tau)n}n^{-2m^k}$.\end{claim}
\begin{proof} We can form a $k$-tuple $x_1,\dots,x_k\in \lbrace 0,\dots,m-1\rbrace^n$ with $x_1+\dots+ x_k=(m-1)\cdot \ones$ coordinate by coordinate by choosing some $t_j\in T_{m-1,k}$ for each $1\leq j\leq n$ and then setting $(x_1^{(j)},\dots,x_k^{(j)})=t_j$. So in order for $(x_1,\dots,x_k)$ to satisfy the assumptions of Proposition \ref{prop-isolated}, we just need to make sure that for each $t\in T_{m-1,k}$ the number of $j\in \lbrace 1,\dots,n\rbrace$ with $t_j=t$ is exactly $\tau(t)n$. So it suffices to prove that there are at least $e^{\HH(\tau)n-\sqrt{n}}$ sequences $t_1,\dots,t_n$ of elements of $T_{m-1,k}$ in which each element $t\in T_{m-1,k}$ occurs exactly $\tau(t)n$ times. By Lemma \ref{lemmaentropycounting} (recall from Proposition \ref{proprounding} that each probability in $\tau$ is an integer multiple of $1/n$),
the number of such sequences is indeed at least $e^{\HH(\tau)n}e^{-\vert T_{m-1,k}\vert}n^{-\vert T_{m-1,k}\vert}\geq e^{\HH(\tau)n}n^{-2m^k}$.
\end{proof}

Combining Claim \ref{claimnumbertaudist} and Proposition \ref{prop-isolated}, we obtain that the expected value of the number $L$ of isolated candidate $k$-tuples is at least
\[
e^{\HH(\tau)n}n^{-2m^k}\exp\left(\HH(\nu)n-\HH(\tau)n-25km\sqrt{n}\right)\\
= e^{\HH(\nu)n-O(\sqrt{n})} \ge (\Gamma_{m,k})^n e^{-C_{m,k}-O(\sqrt{n})}=(\Gamma_{m,k})^{n-O(\sqrt{n})}
.\]
Here, we used $\HH(\nu)\geq \log \Gamma_{m,k}-C_{m,k}/n$. This finishes the proof of Proposition \ref{proplowerboundinZ}.\end{proof}

We will now prove Proposition \ref{prop-isolated}, apart from postponing the proofs of Lemma \ref{lemma-linearalgebra} and Proposition \ref{prop-xprimecount} to the next two sections. While Lemma \ref{lemma-linearalgebra} is a relatively easy linear algebra statement, proving Proposition \ref{prop-xprimecount} is the main difficulty in the first part of this paper. The proof will take up all of Section \ref{sect-xprime-count}.

\begin{proof}[Proof of Proposition \ref{prop-isolated}] Let us fix $x_1,\dots,x_k\in \lbrace 0,\dots,m-1\rbrace^n$ with $x_1+\dots+ x_k=(m-1)\cdot \ones$ and such that for every $t\in T_{m-1,k}$ the number of $j\in \lbrace 1,\dots,n\rbrace$ with $(x_1^{(j)},\dots,x_k^{(j)})=t$ is precisely $\tau(t)n$.

First, we need to prove that $x_i\in X_0$ for $1\leq i\leq k$. Let us assume that $i=1$, the other cases are analogous. In order to show $x_1\in X_0$, we need to check that for every $a=0,\dots,m-1$ the vector $x_1$ has exactly $\nu(a)n$ coordinates equal to $a$. Recall that for each $(a_1,\dots,a_k)\in T_{m-1,k}$ the number of $j\in \lbrace 1,\dots,n\rbrace$ with $(x_1^{(j)},\dots,x_k^{(j)})=(a_1,\dots,a_k)$ is precisely $\tau(a_1,\dots,a_k)n$. Hence the number of $j\in \lbrace 1,\dots,n\rbrace$ with $x_1^{(j)}=a$ is precisely
\[ \sum_{\substack{a_2,\dots,a_k\in \lbrace 0,\dots,m-1\rbrace\\a+a_2+\dots+a_k=m-1}}\tau(a,a_2,\dots,a_k)n=\mu(\tau)(a)n=\nu(a)n.\]
Here we used that $\nu$ is the marginal of $\tau$ (see Proposition \ref{proprounding}). So the vector $x_1$ has indeed exactly $\nu(a)n$ coordinates equal to $a$. Thus, $x_1\in X_0$, and analogously $x_2,\dots,x_k\in X_0$.

Now we need to prove the desired lower bound for the probability that $(x_1,\dots,x_k)$ is an isolated candidate $k$-tuple. Since we already assumed $x_1+\dots+ x_k=(m-1)\cdot \ones$, the $k$-tuple $(x_1,\dots,x_k)$ will be a candidate $k$-tuple if and only if $x_i\in X_i$ for $i=1,\dots,k$. If $(x_1,\dots,x_k)$ is a candidate $k$-tuple, it is isolated if there does not exist another candidate $k$-tuple $(x_1',x_2',\dots,x_k')$ with $x_i'=x_i$ for some $1\leq i\leq k$. Note that for any such $(x_1',x_2',\dots,x_k')$, we would need to have $(x_1',x_2',\dots,x_k')\in X_1\times\dots\times X_k\su X_0^k$. So the probability that $(x_1,\dots,x_k)$ is an isolated candidate $k$-tuple is at least
\begin{equation}\label{eq-expression-prob}
\P[(x_1,\dots,x_k) \in X_1 \times \dots \times X_k]-\sum_{(x_1',\dots,x_k')} \P[(x_1,\dots,x_k),(x_1',\dots,x_k') \in X_1 \times \dots\times X_k],
\end{equation}
where the sum is over all $(x_1',x_2',...,x_k')\in X_0^k$ with $x_1'+\dots+x_k'=(m-1)\cdot \ones$ and $x_i=x_i'$ for some $i$ but with $(x_1',\dots,x_k')\neq (x_1,\dots,x_k)$.

Let us first determine the first term, namely the probability that $x_i\in X_i$ for $i=1,\dots,k$. Recall that by the definition of $X_i$, we have $x_i\in X_i$ if and only if $f(g_i(x_i))\in Y_i$. Also recall that $g_1(x_1),\dots,g_{k-1}(x_{k-1})\in \FPnk$ are linearly independent over $\FP$. Hence by Claim \ref{claim-independence} the images $f(g_1(x_1)),\dots,f(g_{k-1}(x_{k-1}))$ are probabilistically independent and uniformly distributed over $\FP$.
Recall that $x_1+\dots+x_k=(m-1)\cdot \ones$ in $\ZZ^n$ implies $g_1(x_1)+\dots+g_k(x_k)=0$ in $\FPnk$ and therefore $f(g_1(x_1))+\dots+f(g_k(x_k))=0$ in $\FP$. If $(x_1,\dots,x_k) \in X_1 \times \dots \times X_k$, then $(f(g_1(x_1)),\dots,f(g_k(x_k)))\in Y_1 \times \dots \times Y_k$, so $(f(g_1(x_1)),\dots,f(g_k(x_k)))$ needs to be one of the $k$-tuples in the multi-colored sum-free set $(y_{1,j}, y_{2,j}, \dots, y_{k,j})_{j=1}^R$. There are $R$ choices for $j\in \lbrace1,\dots,R\rbrace$, and for each of these choices the probability of $(f(g_1(x_1)),\dots,f(g_k(x_k)))= (y_{1,j}, \dots, y_{k,j})$ equals $(1/P)^{k-1}$ (for each $i=1,\dots,k-1$ the probability of $f(g_i(x_{i}))=y_{i,j}$ is $1/P$ and these events are all independent. If they all happen, then due to $f(g_1(x_1))+\dots+f(g_k(x_k))=0$ we automatically have $f(g_k(x_{k}))=y_{k,j}$). Hence
\begin{equation}\label{eq-prob-1}
\P[(x_1,\dots,x_k) \in X_1 \times \dots \times X_k]=\frac{R}{P^{k-1}}.
\end{equation}

Now, for the second term, let  $(x_1',x_2',...,x_k')\in X_0^k$ with $x_1'+\dots+x_k'=(m-1)\cdot \ones$ and $x_i=x_i'$ for some $i$ but with $(x_1',\dots,x_k')\neq (x_1,\dots,x_k)$.  Let $d=\dim\spn_\QQ(x_1-x_1',\dots,x_k-x_k')$. Note that $1\leq d\leq k-2$, since $x_i-x_i'=0$ for some $i$ (but not for all $i$), and $\sum_i (x_i-x_i')=0$. 
By the following lemma, the dimension of the span of $g_1(x_1),\dots ,g_k(x_k),g_1(x_1'),\dots,g_k(x_k')\in \FPnk$ equals $k-1+d$. We remark that this lemma uses the fact that $P$ is large (in terms of $m$ and $k$).

\begin{lemma}\label{lemma-linearalgebra}
Suppose that $(x_1,\dots,x_k),(x_1',\dots,x_k')\in X_0^k$ satisfy $x_1+\dots+x_k=x_1'+\dots+x_k'=(m-1)\cdot \ones$. Then \[\dim\spn_{\FP}\left(g_1(x_1),\dots ,g_k(x_k),g_1(x_1'),\dots,g_k(x_k')\right)=k-1+\dim\spn_\QQ \left(x_1-x_1',\dots ,x_k-x_k'\right).\]
\end{lemma}

We postpone the proof of this lemma to Section \ref{sect-lemmas1}. By the lemma, we can choose $k-1+d$ linearly independent vectors among $g_1(x_1),\dots ,g_k(x_k),g_1(x_1'),\dots,g_k(x_k')$. By Claim \ref{claim-independence} the images under $f$ of these $k-1+d$ linearly independent vectors will be independently uniformly distributed in $\FP$.

We need to find an upper bound for the probability $\P[(x_1,\dots,x_k),(x_1',\dots,x_k') \in X_1 \times \dots\times X_k]$. We have $(x_1,\dots,x_k),(x_1',\dots,x_k') \in X_1 \times \dots\times X_k$ if and only if $(f(g_1(x_1)),\dots,f(g_k(x_k))), (f(g_1(x_1')),\dots,f(g_k(x_k')))\in Y_1 \times \dots \times Y_k$. Since $f(g_1(x_1))+\dots+f(g_k(x_k))=0$ and $f(g_1(x_1'))+\dots+f(g_k(x_k'))=0$ in $\FP$, this can only happen if both of $(f(g_1(x_1)),\dots,f(g_k(x_k)))$ and $(f(g_1(x_1')),\dots,f(g_k(x_k')))$ are $k$-tuples from the collection $(y_{1,j}, y_{2,j}, \dots, y_{k,j})_{j=1}^R$. Since $x_i=x_i'$ for some $i$, we also have $f(g_i(x_i))=f(g_i(x_i'))$. Therefore, as no two $k$-tuples in the collection $(y_{1,j}, y_{2,j}, \dots, y_{k,j})_{j=1}^R$ share the same $i$-th vector, we must have
\[(f(g_1(x_1)),\dots,f(g_k(x_k)))=(f(g_1(x_1')),\dots,f(g_k(x_k')))=(y_{1,j}, y_{2,j}, \dots, y_{k,j})\]
for some $j\in \lbrace1,\dots,R\rbrace$. For each of the $R$ choices of $j$, the probability of satisfying the equation above is at most $(1/P)^{k-1+d}$ 
(since among $g_1(x_1),\dots ,g_k(x_k),g_1(x_1'),\dots,g_k(x_k')$ there are $k-1+d$ linearly independent vectors, and each of them has probability $1/P$ to have the desired image under the map $f$)\footnote{One can actually show that this is always true with equality, but we will not need this.}. Hence
\begin{equation}\label{eq-prob-2}
\P[(x_1,\dots,x_k),(x_1',\dots,x_k') \in X_1 \times \dots\times X_k]\leq \frac{R}{P^{k-1+d}}.
\end{equation}

The following proposition gives an upper bound for the number of different choices of $(x_1',x_2',...,x_k')\in X_0^k$ that we need to consider for each value of $1\leq d\leq k-2$. We postpone the proof of Proposition \ref{prop-xprimecount} to Section \ref{sect-xprime-count}. This proof is the major part of the work of deriving Theorem \ref{thmmain} from Theorem \ref{theo-marginal}.

\begin{proposition}\label{prop-xprimecount}Let $(x_1,\dots,x_k)\in X_0^k$ with $x_1+\dots+x_k=(m-1)\cdot \ones$ be fixed such that for every $t\in T_{m-1,k}$ the number of $j\in \lbrace 1,\dots,n\rbrace$ with $(x_1^{(j)},\dots,x_k^{(j)})=t$ is exactly $\tau(t)n$. Then for each $1\leq d\leq k-2$, then there are at most
\[ n^{D_{m,k}+2m^k}\exp\left(\frac{\HH(\tau)-\HH(\nu)}{k-2}\cdot dn\right)\]
different $k$-tuples $(x_1',x_2',...,x_k')\in X_0^k$ satisfying $x_1'+\dots+x_k'=(m-1)\cdot \ones$, $\dim\spn_\QQ(x_1-x_1',\dots,x_k-x_k')=d$,
and $x_i=x_i'$ for some $i$.
\end{proposition}

For every $1\leq d\leq k-2$ we have by (\ref{eq-def-P})
\[\frac{n^{D_{m,k}+2m^k}\exp\left(\frac{\HH(\tau)-\HH(\nu)}{k-2}\cdot dn\right)}{P^d}\leq \frac{n^{D_{m,k}+2m^k}\exp\left(\frac{\HH(\tau)-\HH(\nu)}{k-2}\cdot dn\right)}{4^dn^{dD_{m,k}+2dm^k}\exp\left(\frac{\HH(\tau)-\HH(\nu)}{k-2}\cdot dn\right)}\leq \frac{1}{4^d}.\] Hence by Proposition \ref{prop-xprimecount} and (\ref{eq-prob-2}) the big sum in (\ref{eq-expression-prob}) is at most
\[\sum_{d=1}^{k-2}n^{D_{m,k}+2m^k}\exp\left(\frac{\HH(\tau)-\HH(\nu)}{k-2}\cdot dn\right)\cdot \frac{R}{P^{k-1+d}}\leq \frac{R}{P^{k-1}}\sum_{d=1}^{k-2}\frac{1}{4^d}\leq \frac{R}{2P^{k-1}}.\]
Recalling (\ref{eq-prob-1}), (\ref{eq-R-large}), (\ref{eq-def-P}), and (\ref{eq-P-upper-bound}), we obtain that the probability that $(x_1,\dots,x_k)$ is an isolated candidate $k$-tuple is at least
\begin{multline*}
\frac{R}{P^{k-1}}-\frac{R}{2P^{k-1}}=\frac{R}{2P^{k-1}}=\frac{R}{P}\cdot \frac{1}{2P^{k-2}}
\geq e^{-12\sqrt{\log P \log k}}\cdot \frac{1}{2}\cdot n^{-(D_{m,k}+3m^k)(k-2)}\exp(-(\HH(\tau)-\HH(\nu))n)\\
\geq n^{-k(D_{m,k}+3m^k)}e^{-12\sqrt{2k(\log m)n \log k}}\exp(\HH(\nu)n-\HH(\tau)n)
\geq \exp\left(\HH(\nu)n-\HH(\tau)n-25km\sqrt{n}\right),
\end{multline*}
where in the last step we used that $n$ is sufficiently large in terms of $m$ and $k$. This finishes the proof of Proposition \ref{prop-isolated}.\end{proof}

\section{Proof of Proposition \ref{prop-xprimecount}}
\label{sect-xprime-count}

In this section, we will prove Proposition \ref{prop-xprimecount}. The first two subsections contain preparations for the proof, and in the third subsection we will actually prove Proposition \ref{prop-xprimecount}.

\subsection{An entropy inequality}

In this subsection, we will establish an inequality between entropies which will be relevant in the process of proving Proposition \ref{prop-xprimecount}.

For any $\l\geq 2$, let $V_\l\su \QQ^\l$ denote the subspace consisting of those vectors $v=(v_1,\dots,v_\l)$ with $v_1+\dots+v_\l=0$. Note that $V_\l$ is a hyperplane in $\QQ^\l$, so it has dimension $\l-1$.
To any $v=(v_1,\dots,v_\l)\in V_\l$ we can associate a map $\tilde{v}: T_{r,\l}\to \mathbb{Q}$ by setting $\tilde{v}(a_1,\dots,a_\l)=v_1a_1+\dots+v_\l a_\l$ for every $(a_1,\dots,a_\l)\in T_{r,\l}$ (here, $r\geq 0$ is a non-negative integer).
Note that for fixed $r\geq 0$ and $\l\geq 2$, any $(a_1,\dots,a_\l)\in T_{r,\l}$ is uniquely determined by its values $\tilde{v}(a_1,\dots,a_\l)$ for all $v\in V_\l$. Indeed, if we are given $\tilde{v}(a_1,\dots,a_\l)$ for all $v\in V_\l$, we in particular know the values $a_j-a_1$ for $j=2,\dots,\l$. Using $a_1+\dots+a_\l=r$, this determines $(a_1,\dots,a_\l)$.

\begin{lemma} \label{lemmaSlimagesdetermine}
	Given any non-zero vector $v\in V_\l$, the vectors $v^\sigma$ for $\sigma \in S_\l$ span the entire space $V_\l$.
\end{lemma}
\begin{proof}
	Let $v=(v_1,\dots,v_\l)$ and note that $v\neq 0$ and $v_1+\dots+v_\l=0$ implies that the coordinates $v_1,\dots,v_\l$ are not all equal. So we may assume without loss of generality that $v_1\neq v_2$. Let $V_\l'$ be the subspace of $V_\l$ spanned by the vectors $v^\sigma$ for $\sigma \in S_\l$. By definition, the subspace $V_\l'$ is stable under permutations of the coordinates. We need to show that $V_\l'=V_\l$.
	Considering the transposition  $\sigma=(1,2)\in S_\l$, we have
	\[v^\sigma-v=(v_2,v_1, v_3,\dots,v_\l)-(v_1,v_2, v_3,\dots,v_\l)=(v_2-v_1, v_1-v_2,0,\dots,0)=(v_2-v_1)\cdot (1,-1,0,\dots,0).\]
	Therefore $(v_2-v_1)\cdot (1,-1,0,\dots,0)\in V_\l'$. Hence, using $v_1\neq v_2$, we obtain $(1,-1,0,\dots,0)\in V_\l'$. So all permutations of the vector $(1,-1,0,\dots,0)$ are contained in $V_\l'$ as well. It is easy to see that all the permutations of $(1,-1,0,\dots,0)$ generate $V_\l$, hence $V_\l'=V_\l$ as desired.\end{proof}

\begin{lemma} \label{lemmasubspaceentropylowerbound}
	Let $r\geq 0$ and $\l\geq 2$. Suppose that $\pi$ is an $S_\l$-symmetric distribution on $T_{r,\l}$ and that $z$ is a random variable on $T_{r,\l}$ with distribution $\pi$. Then for any subspace $W\su V_\l$, we have
	\[\HH((\tilde{w}(z))_{w\in W}) \ge \frac{\dim W}{\l-1} \HH(\pi).\]
\end{lemma}

\begin{proof}
	We prove the lemma by strong induction on $\dim W$. The case $\dim W=0$ is trivial. Suppose therefore that $\dim W>0$ and that we have already proved the lemma for every subspace of $V_\l$ with smaller dimension.

	For every $\sigma\in S_\l$, we obtain a subspace $W^{\sigma}=\lbrace w^\sigma\mid w\in W\rbrace\su V_\l$ from $W$ by permuting all vectors $w\in W$ according to $\sigma$. Clearly, $\dim W^{\sigma}=\dim W$.
	Given $w\in W$ and $\sigma\in S_\l$, we have $\tilde{w}^\sigma(t)=\tilde{w}(t^{\sigma^{-1}})$ for every $t\in T_{r,\l}$. Therefore, as $\pi$ is $S_\l$-symmetric, we can conclude that $(\tilde{w}^\sigma(z))_{w\in W}$ and $(\tilde{w}(z))_{w\in W}$ have the same distribution. Hence
	\begin{equation}\label{eq-entropy-subspace-permut}
	\HH((\tilde{w}(z))_{w\in W^{\sigma}})=\HH((\tilde{w}^\sigma(z))_{w\in W})=\HH((\tilde{w}(z))_{w\in W})
	\end{equation}
	for every $\sigma\in S_\l$.
	
	As $\dim W>0$, the subspace $W\su V_\l$ contains a non-zero vector $v$. By Lemma \ref{lemmaSlimagesdetermine}, the vectors $v^\sigma$ for $\sigma \in S_\l$ span the entire space $V_\l$. In particular, we have
	\[\sum_{\sigma\in S_\l}W^\sigma=V_\l.\]
	Let $\sigma_1,\sigma_2,...,\sigma_q \in S_\l$ be a sequence of permutations with $W^{\sigma_1}+\dots+W^{\sigma_q}=V_\l$ and such that there is no shorter sequence of permutations with that property. In particular, for every $j=1,\dots,q$ we have $W^{\sigma_j}\nsubseteq\ W^{\sigma_1}+\dots+W^{\sigma_{j-1}}$, because otherwise $\sigma_j$ could be omitted from the sequence. Since $\l\geq 2$, we have $V_\l\neq 0$, and hence $q\geq 1$.
	For $j=1,\dots, q$ set
	\[W_j=W^{\sigma_1}+\dots+W^{\sigma_j},\]
	and set $W_0=0$. Then $W_q=V_l$ and $W_{j-1} \ne W_{j}$ for $j=1,\dots, q$ (so $W_{j-1}$ is a proper subspace of $W_j$).
	Let $U_j=W^{\sigma_j}\cap W_{j-1}$ for $j=1,\dots, q$. Then
	\[\dim W_j=\dim(W_{j-1}+W^{\sigma_j})=\dim W_{j-1}+\dim W^{\sigma_j}-\dim U_j=\dim W_{j-1}+\dim W-\dim U_j.\]
	Hence $\dim U_j=\dim W-(\dim W_j-\dim W_{j-1})$ for $j=1,\dots, q$.
	In particular $\dim U_j<\dim W$, so by the inductive assumption we have
	\begin{equation}\label{eq-entropy-lemma-inductive}
	\HH((\tilde{w}(z))_{w\in U_j}) \ge \frac{\dim U_j}{\l-1}\HH(\pi)=\frac{\dim W-(\dim W_j-\dim W_{j-1})}{\l-1}\HH(\pi).
	\end{equation}
	
	Recall that any $t\in T_{r,\l}$ is uniquely determined by its images $\tilde{v}(t)$ for all $v\in V_\l$. Hence the random variable $z$ is determined by $(\tilde{w}(z))_{w\in W_q}$ (recall that $W_q=V_\l$). Conversely, $(\tilde{w}(z))_{w\in W_q}$ is clearly also detemined by $z$. Therefore we have $\HH(\pi)=\HH(z)=\HH((\tilde{w}(z))_{w\in W_q})$. Also note that if $w=0$, then $\HH(\tilde{w}(z))=0$, hence $\HH((\tilde{w}(z))_{w\in W_0})=0$ (recall that $W_0=0$).
	Using $W_0\su W_1\su\dots\su W_q$, we therefore obtain
	\[\HH(\pi)=\HH((\tilde{w}(z))_{w\in W_q})-\HH((\tilde{w}(z))_{w\in W_0})=\sum_{j=1}^{q}\big[\HH((\tilde{w}(z))_{w\in W_j})- \HH((\tilde{w}(z))_{w\in W_{j-1}})\big].\]
	For every $j=1,\dots,q$, we have $W_j=W_{j-1}+W^{\sigma_j}$. Hence, $(\tilde{w}(z))_{w\in W_j}$ is completely determined by $(\tilde{w}(z))_{w\in W_{j-1}}$ and $(\tilde{w}(z))_{w\in W^{\sigma_j}}$. So $\HH((\tilde{w}(z))_{w\in W_j})=\HH((\tilde{w}(z))_{w\in W_{j-1}}, (\tilde{w}(z))_{w\in W^{\sigma_j}})$ and we obtain
	\[\HH(\pi)=\sum_{j=1}^{q}\big[\HH((\tilde{w}(z))_{w\in W_{j-1}}, (\tilde{w}(z))_{w\in W^{\sigma_j}})- \HH((\tilde{w}(z))_{w\in W_{j-1}})\big]=\sum_{j=1}^{q}\HH((\tilde{w}(z))_{w\in W^{\sigma_j}} \mid (\tilde{w}(z))_{w\in W_{j-1}}).\]
	Using $U_j\su W_{j-1}$, this gives
	\[\HH(\pi)=\sum_{j=1}^{q}\HH((\tilde{w}(z))_{w\in W^{\sigma_j}} \mid (\tilde{w}(z))_{w\in W_{j-1}})\leq \sum_{j=1}^{q}\HH((\tilde{w}(z))_{w\in W^{\sigma_j}} \mid (\tilde{w}(z))_{w\in U_{j}}).\]
	Now, using $U_j\su W^{\sigma_j}$, (\ref{eq-entropy-subspace-permut}) and (\ref{eq-entropy-lemma-inductive}), we obtain
	\begin{multline*}
	\HH(\pi)\leq \sum_{j=1}^{q}\big[\HH((\tilde{w}(z))_{w\in W^{\sigma_j}}) - \HH((\tilde{w}(z))_{w\in U_{j}})\big]=\sum_{j=1}^{q}\big[\HH((\tilde{w}(z))_{w\in W}) - \HH((\tilde{w}(z))_{w\in U_{j}})\big]\\
	=q\HH((\tilde{w}(z))_{w\in W})-\sum_{j=1}^{q}\HH((\tilde{w}(z))_{w\in U_{j}})
	\leq q\HH((\tilde{w}(z))_{w\in W})-\sum_{j=1}^{q}\frac{\dim W-(\dim W_j-\dim W_{j-1})}{\l-1}H(\pi)\\
	=q\HH((\tilde{w}(z))_{w\in W})- \frac{q\dim W}{\l-1}\HH(\pi)+\frac{\sum_{j=1}^{q}(\dim W_j-\dim W_{j-1})}{\l-1}\HH(\pi).
	\end{multline*}
	Recall that $\dim W_0=0$ and $\dim W_q=\dim V_\l=\l-1$. Thus,
	\[\sum_{j=1}^{q}(\dim W_j-\dim W_{j-1})=\dim W_q-\dim W_0=\l-1\]
	and we obtain
	\[\HH(\pi)\leq q\HH((\tilde{w}(z))_{w\in W}) - \frac{q\dim W}{\l-1}\HH(\pi)+\HH(\pi).\]
	Now, rearranging gives
	\[\frac{q\dim W}{l-1}\HH(\pi) \leq q\HH((\tilde{w}(z))_{w\in W}) .\]
	Since $q > 0$, this proves the lemma.
\end{proof}

We can deduce the following corollary from Lemma \ref{lemmasubspaceentropylowerbound}. Here, $\nu$ and $\tau$ are the distributions on $\lbrace 0,\dots, m-1\rbrace$ and $T_{m-1,k}$, respectively, that we fixed earlier using Proposition \ref{proprounding}. Also recall that $k\geq 3$.

\begin{corollary} \label{corollarysubspaceentropylowerbound}
	Let $z_\tau$ be a random variable on $T_{m-1,k}$ with distribution $\tau$. Then for any subspace $W\su V_k$ with $(1,\dots,1,-(k-1))\in W$ we have 
	\[\HH((\tilde{w}(z_\tau))_{w\in W})\geq \HH(\nu)+\frac{\dim W-1}{k-2}(\HH(\tau)-\HH(\nu)).\]
\end{corollary}

\begin{proof}
	Set $v=(1,\dots,1,-(k-1))\in W$. Furthermore, let
	\[W'=\lbrace (w_1,\dots,w_k)\in W\mid w_k=0\rbrace.\]
	Then $W'$ is a subspace of $W$ and $\dim W'=\dim W-1$. Note that $W'$ and $v$ together span the entire space $W$. Hence $(\tilde{w}(z_\tau))_{w\in W}$ is completely determined by $(\tilde{w}(z_\tau))_{w\in W'}$ and $\tilde{v}(z_\tau)$,
	so
	\[\HH((\tilde{w}(z_\tau))_{w\in W})=\HH((\tilde{w}(z_\tau))_{w\in W'}, \tilde{v}(z_\tau)).\]
	Furthermore, note that when writing $(z_\tau^{(1)},\dots,z_\tau^{(k)})$ for the coordinates of $z_\tau$, we have
	\[\tilde{v}(z_\tau)=z_\tau^{(1)}+\dots+z_\tau^{(k-1)}-(k-1)z_\tau^{(k)}=(z_\tau^{(1)}+\dots +z_\tau^{(k)})-kz_\tau^{(k)}=(m-1)-kz_\tau^{(k)}.\]
	So the value $\tilde{v}(z_\tau)$ is in one-to-one correspondence with the last coordinate $z_\tau^{(k)}$ of $z_\tau$. In particular, $\HH(\tilde{v}(z_\tau))=\HH(z_\tau^{(k)})$. But note that the projection $z_\tau^{(k)}$ of $z_\tau$ to the last coordinate has distribution $\nu$ (since the marginal of $\tau$ is $\nu$), so $\HH(\tilde{v}(z_\tau))=\HH(z_\tau^{(k)})=\HH(\nu)$.
	
	Now, we have
	\begin{multline*}
	\HH((\tilde{w}(z_\tau))_{w\in W})-\HH(\nu)=\HH((\tilde{w}(z_\tau))_{w\in W'}, \tilde{v}(z_\tau))-\HH(\tilde{v}(z_\tau))=\HH((\tilde{w}(z_\tau))_{w\in W'}\mid \tilde{v}(z_\tau))\\
	=\sum_{s\in \tilde{v}(T_{m-1,k})} \P(\tilde{v}(z_\tau)=s)\HH((\tilde{w}(z_\tau))_{w\in W'}\mid \tilde{v}(z_\tau)=s).
	\end{multline*}
	Using the one-to-one correspondence between the values $\tilde{v}(z_\tau)$ and  $z_\tau^{(k)}$, the right-hand side can be rewritten in terms of $z_\tau^{(k)}$ instead of $\tilde{v}(z_\tau)$. So we obtain
	\[\HH((\tilde{w}(z_\tau))_{w\in W})-\HH(\nu)= \sum_{a\in \lbrace 0,\dots,m-1\rbrace} \P(z_\tau^{(k)}=a)\HH((\tilde{w}(z_\tau))_{w\in W'}\mid z_\tau^{(k)}=a).\]
	For each $a\in \lbrace 0,\dots,m-1\rbrace$, the probability distribution of $z_\tau$ conditioned on $z_\tau^{(k)}=a$ gives an $S_{k-1}$-symmetric probability distribution on $T_{m-1-a,k-1}$ (by omitting the last coordinate $z_\tau^{(k)}=a$). Note that $k-1\geq 2$ and that for all $w\in W'$ the last coordinate is zero. By omitting this last coordinate zero, we can interpret $W'$ as a subspace of $V_{k-1}\su \QQ^{k-1}$. Note that for $w\in W'$ the value of $\tilde{w}(z_\tau)$ remains the same when we omit the last coordinate of both $w$ and $z_\tau$. Hence Lemma \ref{lemmasubspaceentropylowerbound} applied to the probability distribution of $((z_\tau^{(1)},\dots,z_\tau^{(k-1)})\mid z_\tau^{(k)}=a)$ on $T_{m-1-a,k-1}$ gives
	\[\HH((\tilde{w}(z_\tau))_{w\in W'}\mid z_\tau^{(k)}=a)\geq \frac{\dim W'}{(k-1)-1}\HH((z_\tau^{(1)},\dots,z_\tau^{(k-1)})\mid z_\tau^{(k)}=a)=\frac{\dim W'}{k-2}\HH(z_\tau\mid z_\tau^{(k)}=a)\]
	for every $a\in \lbrace 0,\dots,m-1\rbrace$. Thus,
	\begin{multline*}
	\HH((\tilde{w}(z_\tau))_{w\in W})-\HH(\nu)\geq \frac{\dim W'}{k-2}\sum_{a\in \lbrace 0,\dots,m-1\rbrace} \P(z_\tau^{(k)}=a)\HH(z_\tau\mid z_\tau^{(k)}=a)\\
	=\frac{\dim W'}{k-2}\HH(z_\tau\mid z_\tau^{(k)})=\frac{\dim W-1}{k-2}(\HH(z_\tau)-\HH(z_\tau^{(k)}))=\frac{\dim W-1}{k-2}(\HH(z_\tau)-\HH(\nu)),
	\end{multline*}
	where we again used $\HH(z_\tau^{(k)})=\HH(\nu)$. Now rearranging gives the desired inequality.
\end{proof}

\subsection{More preparations for the proof of Proposition \ref{prop-xprimecount}}

This subsection establishes the key ingredient for the proof of Proposition \ref{prop-xprimecount}, namely Lemma \ref{lemma-ingredient}. We will state and prove this lemma at the end of this subsection, building on the first lemma of this subsection and on the results of the previous subsection. The actual proof of Proposition \ref{prop-xprimecount} in the next subsection will only use Lemma \ref{lemma-ingredient}, but the other results from this subsection and the previous subsection are needed in order to prove Lemma \ref{lemma-ingredient}.

During this entire subsection, we will operate under the following assumption, which reflects the assumption of Proposition \ref{prop-xprimecount}.
\begin{assumption}\label{assumption-preparation}
	We assume that $(x_1,\dots,x_k)\in X_0^k$ is fixed with $x_1+\dots+x_k=(m-1)\cdot \ones$ and such that for every $t\in T_{m-1,k}$ the number of $j\in \lbrace 1,\dots,n\rbrace$ with $(x_1^{(j)},\dots,x_k^{(j)})=t$ is exactly $\tau(t)n$.  Furthermore, let $z_\tau$ be a random variable on $T_{m-1,k}$ with distribution $\tau$.
\end{assumption}
As in the last subsection, to any vector $w=(w_1,\dots,w_k)\in \mathbb{Q}^k$ we can associate a map $\tilde{w}: T_{m-1,k}\to \mathbb{Q}$ by setting $\tilde{w}(a_1,\dots,a_k)=w_1a_1+\dots+w_ka_k$ for every $(a_1,\dots,a_k)\in T_{m-1,k}$.
Further note that for any $x_1',\dots,x_k'\in \lbrace 0,\dots,m-1\rbrace^n$ with $x_1'+\dots+x_k'=(m-1)\cdot \ones$, we have $x_1'^{(j)}+\dots +x_k'^{(j)}=m-1$ for each coordinate $j=1,\dots,n$. Hence $(x_1'^{(j)},\dots,x_k'^{(j)})\in T_{m-1,k}$ for each $j=1,\dots,n$.

\begin{lemma} \label{lemmasubspacecyclecount}
	Fix a subspace $W\su \mathbb{Q}^k$ and a probability distribution $\pi$ on $T_{m-1,k}$. Then there are at most
	\[\exp(\HH(\pi)n-\HH((\tilde{w}(z_\tau))_{w\in W})n)\]
	different $k$-tuples $(x_1',\dots,x_k')$ with $x_1',\dots,x_k'\in \lbrace 0,\dots,m-1\rbrace^n$, $x_1'+\dots+x_k'=(m-1)\cdot \ones$ and
	\[w_1x_1'+\dots+w_kx_k'=w_1x_1+\dots+w_kx_k\]
	for all $w=(w_1,\dots,w_k)\in W$ and such that for every $t\in T_{m-1,k}$ the number of $j\in \lbrace 1,\dots,n\rbrace$ with $(x_1'^{(j)},\dots,x_k'^{(j)})=t$ is exactly $\pi(t)n$.
\end{lemma}

\begin{proof}Let $M$ be the number of $k$-tuples $(x_1',\dots,x_k')$ with the properties listed in the lemma. We need to prove that $M\leq \exp(\HH(\pi)n-\HH((\tilde{w}(z_\tau))_{w\in W})n)$. If $M=0$, this is trivially true, so we can assume that there is at least one $k$-tuple $(x_1',\dots,x_k')$ with the properties listed in the lemma.
	
	Let $z_\pi$ be a random variable on $T_{m-1,k}$ with distribution $\pi$.
	
	\begin{claim}\label{claim-H-tau-pi}
		$\HH((\tilde{w}(z_\pi))_{w\in W})=\HH((\tilde{w}(z_\tau))_{w\in W})$.
	\end{claim}
	\begin{proof} We assumed that there is at least one $k$-tuple $(x_1',\dots,x_k')$ with the properties listed in the lemma, so let $(x_1',\dots,x_k')$ be such a $k$-tuple. Then for each $j\in \lbrace 1,\dots,n\rbrace$ we have $w_1x_1'^{(j)}+\dots+ w_kx_k'^{(j)}=w_1x_1^{(j)}+\dots+ w_kx_k^{(j)}$ for all $w=(w_1,\dots,w_k)\in W$. In other words, for every $j\in \lbrace 1,\dots,n\rbrace$ we have $\tilde{w}(x_1'^{(j)},\dots, x_k'^{(j)})=\tilde{w}(x_1^{(j)},\dots, x_k^{(j)})$ for all $w\in W$. Note that if we choose $j\in \lbrace 1,\dots,n\rbrace$ uniformly at random, then $(x_1'^{(j)},\dots, x_k'^{(j)})\in T_{m-1,k}$ will be distributed according to $\pi$ and $(x_1^{(j)},\dots, x_k^{(j)})\in T_{m-1,k}$ will be distributed according to $\tau$. Hence the distributions of $(\tilde{w}(z_\pi))_{w\in W}$ and of $(\tilde{w}(z_\tau))_{w\in W}$ must agree, and in particular they must have the same entropy. Thus, $\HH((\tilde{w}(z_\pi))_{w\in W})=\HH((\tilde{w}(z_\tau))_{w\in W})$.
	\end{proof}
	
	Note that specifying a $k$-tuple $(x_1',\dots,x_k')$ with $x_1',\dots,x_k'\in \lbrace 0,\dots,m-1\rbrace^n$ and $x_1'+\dots+x_k'=(m-1)\cdot \ones$ is the same as specifying $t_1,\dots,t_n\in T_{m-1,k}$ and setting $(x_1'^{(j)},\dots, x_k'^{(j)})=t_j$ for $1\leq j\leq n$.
	As we saw in the proof of the last claim, having
	\[w_1x_1'+\dots+w_kx_k'=w_1x_1+\dots+w_kx_k\]
	for all $w=(w_1,\dots,w_k)\in W$ is equivalent to having $\tilde{w}(x_1'^{(j)},\dots, x_k'^{(j)})=\tilde{w}(x_1^{(j)},\dots, x_k^{(j)})$ for all $w\in W$ and $j=1,\dots,n$. Hence the sequence $t_1,\dots,t_n\in T_{m-1,k}$ must satisfy $\tilde{w}(t_j)=\tilde{w}(x_1^{(j)},\dots, x_k^{(j)})$ for all $w\in W$ and $j=1,\dots,n$.
	Therefore, the number of $k$-tuples $(x_1',\dots,x_k')$ with the conditions in Lemma \ref{lemmasubspacecyclecount} equals the number of sequences $t_1,\dots, t_n$ of elements of $T_{m-1,k}$ in which each element $t\in T_{m-1,k}$ occurs exactly $\pi(t)n$ times and such that $\tilde{w}(t_j)=\tilde{w}(x_1^{(j)},\dots, x_k^{(j)})$ for all $w\in W$ and $1\leq j\leq n$. By Lemma \ref{lemmaentropycounting2} applied to the map $f:T_{m-1,k}\to \QQ^{\vert W\vert}$ given by $f(t)=(\tilde{w}(t))_{w\in W}$, this number is at most
	\[\exp(\HH(z_\pi)n-\HH((\tilde{w}(z_\pi))_{w\in W})n)=\exp(\HH(\pi)n-\HH((\tilde{w}(z_\tau))_{w\in W})n).\]
	Here we used Claim \ref{claim-H-tau-pi}. This finishes the proof of Lemma \ref{lemmasubspacecyclecount}.
\end{proof}

Now, we can finally prove Lemma \ref{lemma-ingredient}, which will be the key ingredient for the proof of Proposition \ref{prop-xprimecount}.

\begin{lemma}\label{lemma-ingredient} Let $(x_1,\dots,x_k)\in X_0^k$ be fixed as in Assumption \ref{assumption-preparation}, and furthermore let $W\su V_k$ be a fixed subspace with $(1,\dots,1,-(k-1))\in W\su \QQ^k$. Then there are at most
	\[\exp\left((k-1-\dim W)\cdot \frac{\HH(\tau)-\HH(\nu)}{k-2}\cdot n + (D_{m,k}+m^k)\log n\right)\]
	different $k$-tuples $(x_1',\dots,x_k')\in X_0^k$  with $x_1'+\dots+x_k'=(m-1)\cdot \ones$ and
	\[w_1x_1'+\dots+w_kx_k'=w_1x_1+\dots+w_kx_k\]
	for all $w=(w_1,\dots,w_k)\in W$.
\end{lemma}
\begin{proof}
	Note that by Corollary \ref{corollarysubspaceentropylowerbound} we have
	\[\HH((\tilde{w}(z_\tau))_{w\in W})\geq \HH(\nu)+\frac{\dim W-1}{k-2}(\HH(\tau)-\HH(\nu)).\]
	
	For $x_1',\dots, x_k'\in X_0\su \lbrace 0,\dots,m-1\rbrace^k$ with $x_1'+\dots+x_k'=(m-1)\cdot \ones$ we have $(x_1'^{(j)},\dots,x_k'^{(j)})\in T_{m-1,k}$ for every $j\in \lbrace 1,\dots,n\rbrace$. So for every $k$-tuple $(x_1',\dots,x_k')\in X_0^k$ with the conditions in the lemma, we can define a probability distribution $\pi$ on $T_{m-1,k}$ by considering $(x_1'^{(j)},\dots,x_k'^{(j)})\in T_{m-1,k}$ where $j\in \lbrace 1,\dots,n\rbrace$ is chosen uniformly at random. Note that then for every $t\in T_{m-1,k}$ the number of $j\in \lbrace 1,\dots,n\rbrace$ with $(x_1'^{(j)},\dots,x_k'^{(j)})=t$ is exactly $\pi(t)n$.
	
	The projection of the probability distribution $\pi$ on $T_{m-1,k}$ to the first coordinate can be described by considering $x_1'^{(j)}$ where $j\in \lbrace 1,\dots,n\rbrace$ is chosen uniformly at random. Since $x_1'\in X_0$, choosing a coordinate $x_1'^{(j)}$ of $x_1'$ uniformly at random gives the probability distribution $\nu$ on $\lbrace 0,\dots,m-1\rbrace$. Hence the projection of the probability distribution $\pi$  to the first coordinate is equal to $\nu$. Analogously, we can see that all the other coordinate projections of $\pi$ are equal to $\nu$ as well. Hence, by the last condition in Proposition \ref{proprounding}, we must have $\HH(\pi)\leq \HH(\tau) + (D_{m,k}\log n)/n.$
	
	Clearly, every probability in $\pi$ is an integer multiple of $1/n$. Hence the number of possibilities for the probability distribution $\pi$ on $T_{m-1,k}$ is at most $n^{\vert T_{m-1,k}\vert}\leq n^{m^k}$.
	If we fix one of these probability distributions $\pi$, then by Lemma \ref{lemmasubspacecyclecount} there are at most
	\[\exp(\HH(\pi)n-\HH((\tilde{w}(z_\tau))_{w\in W})n)\leq \exp\left(\HH(\tau)n + D_{m,k}\log n-\HH(\nu)n-\frac{\dim W-1}{k-2}(\HH(\tau)-\HH(\nu))n\right).\]
	different $k$-tuples $(x_1',\dots,x_k')\in X_0^k$ with the conditions in Lemma \ref{lemma-ingredient} that give rise to this particular probability distribution $\pi$ on $T_{m-1,k}$. Thus, all in all the number of $k$-tuples $(x_1',\dots,x_k')\in X_0^k$ with the conditions in Lemma \ref{lemma-ingredient} is at most
	\begin{multline*}
	n^{m^k}\exp\left(\HH(\tau)n + D_{m,k}\log n-\HH(\nu)n-\frac{\dim W-1}{k-2}(\HH(\tau)-\HH(\nu))n\right)\\
	=\exp\left(\frac{k-1-\dim W}{k-2}(\HH(\tau)-\HH(\nu))n + (D_{m,k}+m^k)\log n\right)
	\end{multline*}
	This finishes the proof of Lemma \ref{lemma-ingredient}.\end{proof}

\subsection{Proof of Proposition \ref{prop-xprimecount}}

\begin{proof}[Proof of Proposition \ref{prop-xprimecount}] Recall that $(x_1,\dots,x_k)\in X_0^k$ with $x_1+\dots+x_k=(m-1)\cdot \ones$ is fixed such that for every $t\in T_{m-1,k}$ the number of $j\in \lbrace 1,\dots,n\rbrace$ with $(x_1^{(j)},\dots,x_k^{(j)})=t$ is exactly $\tau(t)n$. So $(x_1,\dots,x_k)\in X_0^k$ satisfies Assumption \ref{assumption-preparation}. Fix $d$ with $1\leq d\leq k-2$. We claim that there are at most
	\[ km^{2k^2}n^{D_{m,k}+m^k}\exp\left(\frac{\HH(\tau)-\HH(\nu)}{k-2}\cdot dn\right)\]
	different $k$-tuples $(x_1',x_2',...,x_k')\in X_0^k$ satisfying $x_1'+\dots+x_k'=(m-1)\cdot \ones$, $\dim\spn_\QQ(x_1-x_1',\dots,x_k-x_k')=d$,
	and $x_i=x_i'$ for some $i$. Since $km^{2k^2} \le n^{m^k}$ for $n$ sufficiently large, this implies the desired bound.
	
	Let us focus on those $(x_1',x_2',...,x_k')\in X_0^k$ with $x_k=x_k'$. We will show that there are at most
	\[ m^{2k^2}n^{D_{m,k}+m^k}\exp\left(\frac{\HH(\tau)-\HH(\nu)}{k-2}\cdot dn\right)\]
	different $k$-tuples $(x_1',x_2',...,x_k')\in X_0^k$ with $x_1'+\dots+x_k'=(m-1)\cdot \ones$, $\dim\spn_\QQ(x_1-x_1',\dots,x_k-x_k')=d$, and $x_k=x_k'$. Analogously, we get the same upper bound when replacing the condition $x_k=x_k'$ with $x_i=x_i'$ for any fixed $1\leq i\leq k-1$ (note that Corollary \ref{corollarysubspaceentropylowerbound} and Lemma \ref{lemma-ingredient} have the assumption $(1,\dots,1,-(k-1))\in W$, but they can be proved analogously if one replaces $(1,\dots,1,-(k-1))$ by one of its permutations). By a union bound, this proves the claim above, and thus Proposition \ref{prop-xprimecount}.
	
	In order to simplify notation, let us call a  $k$-tuple $(x_1',x_2',...,x_k')\in X_0^k$ \emph{relevant} if  $x_1'+\dots+x_k'=(m-1)\cdot \ones$, $\dim\spn_\QQ(x_1-x_1',\dots,x_k-x_k')=d$, and $x_k=x_k'$. Our goal is to prove that there are at most
	\[m^{2k^2}n^{D_{m,k}+m^k}\exp\left(\frac{\HH(\tau)-\HH(\nu)}{k-2}\cdot dn\right)\]
	relevant $k$-tuples $(x_1',x_2',...,x_k')\in X_0^k$.
	
	For every relevant $k$-tuple $(x_1',x_2',...,x_k')\in X_0^k$, let us consider the subspace $W^{*}\su\QQ^k$ given by
	\[W^{*}=\lbrace (w_1,\dots,w_k)\in \QQ^k\mid w_1\cdot (x_1-x_1')+\dots+w_k\cdot (x_k-x_k')=0\rbrace.\]
	
	\begin{claim}\label{claim-dimension-Wstar}
		$\dim_{\QQ}W^{*}=k-d$.
	\end{claim}
	\begin{proof} Note that $W^{*}$ is the null-space of the $(n\times k)$-matrix with columns $x_1-x_1',\dots, x_k-x_k'$. Since $\dim\spn_\QQ(x_1-x_1',\dots,x_k-x_k')=d$, this matrix has rank $d$. Therefore the dimension of its null-space is $k-d$.\end{proof}
	
	Note that $W^{*}$ depends on the relevant $k$-tuple $(x_1',x_2',...,x_k')\in X_0^k$ and different relevant $k$-tuples can give different subspaces $W^{*}\su \mathbb{Q}^k$. However, the following claim gives an upper bound for the total number of different subspaces $W^{*}$ that can occur.
	
	\begin{claim}\label{claim-possible-Wstar}
		For fixed $(x_1,\dots,x_k)$ and fixed $d$, the number of possible subspaces $W^{*}\su \mathbb{Q}^k$ is at most $m^{2k^2}$.
	\end{claim}
	\begin{proof}Recall that every possible $W^{*}$ occurs as the null-space of an $(n\times k)$-matrix $A$ with columns ${x_1-x_1',\dots, x_k-x_k'}$ for some relevant $k$-tuple $(x_1',\dots,x_k')$. Each row of this matrix is of the form
		\[(x_1^{(j)}-x_1'^{(j)},\dots,x_k^{(j)}-x_k'^{(j)})=(x_1^{(j)},\dots,x_k^{(j)})-(x_1'^{(j)},\dots,x_k'^{(j)})\]
		for some $1\leq j\leq n$. Note that $(x_1^{(j)},\dots,x_k^{(j)}), (x_1'^{(j)},\dots,x_k'^{(j)})\in T_{m-1,k}$, so each row of $A$ is a vector from the set $\lbrace t-t'\mid t,t'\in T_{m-1,k}\rbrace$. Furthermore, since
		\[\operatorname{rank}A=\dim\spn_\QQ(x_1-x_1',\dots,x_k-x_k')=d,\]
		we can select $d$ linearly independent rows of the matrix $A$ and the matrix $A'$ formed by these $d$ rows has the same null-space as $A$. Hence $W^{*}$ occurs as the null-space of a $(d\times k)$-matrix $A'$ such that each row of $A'$ is from the set $\lbrace t-t'\mid t,t'\in T_{m-1,k}\rbrace$. Since this set has size at most $\vert T_{m-1,k}\vert^2\leq (m^k)^2=m^{2k}$, there are at most $(m^{2k})^d=m^{2kd}\leq m^{2k^2}$ possibilities to form such a matrix $A'$. Hence there are at most $m^{2k^2}$ possible subspaces $W^{*}\su \mathbb{Q}^k$.
	\end{proof}
	
	Recall that we defined the hyperplane $V_k=\lbrace (v_1,\dots,v_k)\in \QQ^k\mid v_1+\dots+v_k=0\rbrace \su \QQ^k$.
	For each relevant $k$-tuple $(x_1',\dots,x_k')\in X_0^k$, let us consider the subspace $W\su \QQ^k$ given by
	\[W=W^{*}\cap V_k.\]
	Clearly, $W\su V_k$. By Claim \ref{claim-possible-Wstar} there are at most $m^{2k^2}$ possibilities for $W^{*}$, hence there are also at most $m^{2k^2}$ possibilities for $W$.
	
	\begin{claim}$\dim_{\QQ}W=k-d-1$.
	\end{claim}
	\begin{proof}Note that we have
		\[1\cdot (x_1-x_1')+\dots+1\cdot (x_k-x_k')=(x_1+\dots+x_k)-(x_1'+\dots+x_k')=(m-1)\cdot\ones-(m-1)\cdot\ones=0,\]
		hence $(1,\dots,1)\in W^{*}$. Clearly $(1,\dots,1)\not\in V_k$, so we have $W^{*}\not\su V_k$. Since $V_k\su \QQ^k$ is a hyperplane (that means $\dim V_k=k-1$), this implies $\dim W=\dim(W^{*}\cap V_k)=\dim W^{*}-1$. By Claim \ref{claim-dimension-Wstar} this yields $\dim W=k-d-1$, as desired.
	\end{proof}
	
	For any relevant $k$-tuple $(x_1',\dots,x_k')$ we have $x_k-x_k'=0$ and therefore
	\[1\cdot (x_1-x_1')+\dots+1\cdot (x_{k-1}-x_{k-1}')-(k-1)\cdot (x_k-x_k')=1\cdot (x_1-x_1')+\dots+1\cdot (x_k-x_k')=0.\]
	Hence $(1,\dots,1,-(k-1))\in W^{*}$, and by $(1,\dots,1,-(k-1))\in V_k$ we obtain ${(1,\dots,1,-(k-1))\in W}$ for every relevant $k$-tuple $(x_1',\dots,x_k')$.
	Furthermore, if $(x_1',\dots,x_k')\in X_0^k$ is a relevant $k$-tuple giving rise to the subspace $W\su V_k$, then for every $(w_1,\dots,w_k)\in W$ we have $w_1\cdot (x_1-x_1')+\dots+w_k\cdot (x_k-x_k')=0$ and therefore
	\[w_1x_1'+\dots+w_kx_k'=w_1x_1+\dots+w_kx_k.\]
	Also recall that every relevant $k$-tuple $(x_1',\dots,x_k')\in X_0^k$ must satisfy $x_1'+\dots+x_k'=(m-1)\cdot \ones$.
	Hence Lemma \ref{lemma-ingredient} implies that for every possible $W\su V_k$, there can be at most
	\[\exp\left((k-1-\dim W)\cdot \frac{\HH(\tau)-\HH(\nu)}{k-2}\cdot n + (D_{m,k}+m^k)\log n\right)=n^{D_{m,k}+m^k}\exp\left(d\cdot \frac{\HH(\tau)-\HH(\nu)}{k-2}\cdot n\right)\]
	different relevant $k$-tuples $(x_1',\dots,x_k')\in X_0^k$ giving rise to this subspace $W$. We also saw above that in total there are at most $m^{2k^2}$ possibilities for $W$. Hence all in all, the number of relevant $k$-tuples $(x_1',\dots,x_k')\in X_0^k$ is at most
	\[m^{2k^2}n^{D_{m,k}+m^k}\exp\left(\frac{\HH(\tau)-\HH(\nu)}{k-2}\cdot dn\right),\] 
	as desired. This completes the proof of Proposition \ref{prop-xprimecount}.\end{proof}

\section{Proof of Proposition \ref{proprounding} and of Lemmas \ref{lemmanupk} and \ref{lemma-linearalgebra}}
\label{sect-lemmas1}

\subsection{Proof of Lemma \ref{lemmanupk}}
\begin{proof}[Proof of Lemma \ref{lemmanupk}]Recall that $0<\gamma_{m,k}<1$ was chosen to minimize
\[\frac{1+\gamma+\dots+\gamma^{m-1}}{\gamma^{(m-1)/k}}=\sum_{i=0}^{m-1}\gamma^{i-(m-1)/k}.\]
Hence the derivative of this function must be zero at the point $\gamma=\gamma_{m,k}$, so
\[\sum_{i=0}^{m-1}\left(i-\frac{m-1}{k}\right)\gamma_{m,k}^{i-1-(m-1)/k}=0.\]
Multiplying by $\gamma_{m,k}^{1+(m-1)/k}$  and rearranging yields
\begin{equation}\label{eq-gammapk}
\sum_{i=0}^{m-1}i\gamma_{m,k}^{i}=\frac{m-1}{k}\sum_{i=0}^{m-1}\gamma_{m,k}^{i}.
\end{equation}
Now, recall that we defined the probability distribution $\nu_{m,k}$ on $\lbrace 0,\dots, m-1\rbrace$ by
\[\nu_{m,k}(i)=\frac{\gamma_{m,k}^i}{1+\gamma_{m,k}+\dots+\gamma_{m,k}^{m-1}}.\]
Hence by (\ref{eq-gammapk}) we obtain
\[\E(\nu_{m,k})=\sum_{i=0}^{m-1} i\nu_{m,k}(i)=\frac{\sum_{i=0}^{m-1} i\gamma_{m,k}^i}{1+\gamma_{m,k}+\dots+\gamma_{m,k}^{m-1}}=\frac{m-1}{k}\cdot\frac{\sum_{i=0}^{m-1} \gamma_{m,k}^i}{1+\gamma_{m,k}+\dots+\gamma_{m,k}^{m-1}}=\frac{m-1}{k},\]
as desired. Using this, we get
\begin{multline*}
\HH(\nu_{m,k})=-\sum_{i=0}^{m-1}\nu_{m,k}(i)\log \nu_{m,k}(i)=-\sum_{i=0}^{m-1}\nu_{m,k}(i)(i\log \gamma_{m,k}- \log(1+\gamma_{m,k}+\dots+\gamma_{m,k}^{m-1}))\\
=-\left(\log \gamma_{m,k}\cdot \sum_{i=0}^{m-1}i\nu_{m,k}(i)\right)+\log(1+\gamma_{m,k}+\dots+\gamma_{m,k}^{m-1})=
-\log \gamma_{m,k}\cdot \frac{m-1}{k}+\log(1+\gamma_{m,k}+\dots+\gamma_{m,k}^{m-1})\\
=\log\left(\frac{1+\gamma+\dots+\gamma^{m-1}}{\gamma^{(m-1)/k}}\right)=\log \Gamma_{m,k},
\end{multline*}
also as desired.\end{proof}

\subsection{Proof of Proposition \ref{proprounding}}
Before going into the proof of Proposition \ref{proprounding}, we will first prove two easy lemmas.

\begin{lemma}\label{lemma-expectation-marginal}
If $\tau$ is an $S_k$-symmetric distribution on $T_{m-1,k}$, then the marginal $\mu(\tau)$ of $\tau$ has expectation $\E(\mu(\tau))=(m-1)/k$.
\end{lemma}
\begin{proof}Let $z_\tau$ be a random variable on $T_{m-1,k}$ with distribution $\tau$. Note that the each of the $k$ individual coordinates of $z_\tau$ has distribution $\mu(\tau)$. Hence each coordinate has expectation $\E(\mu(\tau))$, so the sum of the $k$ coordinates of $z_\tau$ has expectation $k\E(\mu(\tau))$. On the other hand, the sum of the coordinates of $z_\tau$ is always equal to $m-1$, hence $k\E(\mu(\tau))=m-1$, which means that $\E(\mu(\tau))=(m-1)/k$.
\end{proof}

Let $U$ be the vector space formed by all functions $u:T_{m-1,k}\to \mathbb{R}$ satisfying $u(t)=u(t^\sigma)$ for all $t\in T_{m-1,k}$ and $\sigma\in S_k$ (i.e. $u$ is $S_k$-symmetric), $\sum_{t\in T_{m-1,k}}u(t)=0$ as well as
\begin{equation}\label{eq-property-u}
\sum_{\substack{a_2,\dots,a_k\in \lbrace 0,\dots,m-1\rbrace\\a+a_2+\dots+a_k=m-1}}u(a,a_2,\dots,a_k)=0
\end{equation}
for every $a\in \lbrace 0,\dots,m-1\rbrace$. Note that $U$ can be interpreted as a subspace of $\mathbb{R}^{\vert T_{m-1,k}\vert}$ and that $U\su \mathbb{R}^{\vert T_{m-1,k}\vert}$ has a basis consisting of points with integer coordinates. Hence there exists some $d_{m,k}>0$ such that for each $u\in U$ we can find $u'\in U$ such that $u'$ has integer coordinates and $\Vert u-u'\Vert_1\leq d_{m,k}$. Basically, $d_{m,k}$ is the maximum distance of any point in $U$ from its closest integer lattice point in $U$. Note that $d_{m,k}>0$ is a constant that only depends on $U$ and its integer lattice, hence it ultimately only depends on $m$ and $k$ (and not on $n$).

\begin{lemma}\label{lemma-lattice-U}For every integer $n$ and every $u\in U$, we can find $u'\in U$ such that $\Vert u-u'\Vert_1\leq d_{m,k}/n$ and all coordinates of $u'\in U\su \mathbb{R}^{\vert T_{m-1,k}\vert}$ are integer multiples of $1/n$.
\end{lemma}
\begin{proof}Let us apply the definition of $d_{m,k}$ to the point $n\cdot u\in U$. There exists a point $\tilde{u}\in U$ with integer coordinates and $\Vert n\cdot u-\tilde{u}\Vert_1\leq d_{m,k}$. Now set $u'=\tilde{u}/n\in U$. Then all coordinates of $u'$ are integer multiples of $1/n$ and furthermore
\[\Vert u-u'\Vert_1=\frac{1}{n}\Vert n\cdot u-\tilde{u}\Vert_1 \leq \frac{d_{m,k}}{n},\]
as desired.\end{proof}

Now, we are ready for the proof of Proposition \ref{proprounding}.

\begin{proof}[Proof of Proposition \ref{proprounding}] Recall that $\nu_{m,k}$ is a probability distribution on $\lbrace 0,\dots,m-1\rbrace$ and by Lemma \ref{lemmanupk} we have $\E(\nu_{m,k})=(m-1)/k$ and $\HH(\nu_{m,k})=\log \Gamma_{m,k}$. Furthermore, by Theorem \ref{theo-marginal}, we can fix an $S_k$-symmetric probability distribution $\tau_{m,k}$ on $T_{m-1,k}$ with marginal $\mu(\tau_{m,k})=\nu_{m,k}$ and with $\tau_{m,k}(t)>0$ for every $t\in T_{m-1,k}$. Set $0<c_{m,k}<1$ to be the minimum of the finitely many values $\tau_{m,k}(t)>0$ for $t\in T_{m-1,k}$. Note that $c_{m,k}$ only depends on $m$ and $k$ (but not on $n$).

Our goal is to find a rounded version $\nu$ of $\nu_{m,k}$ (and an appropriate $\tau$) with the properties in Proposition \ref{proprounding}. Recall that $n$ is sufficiently large and furthermore divisible by $k$.
First, let us form a rounded version $\tilde{\tau}$ of the probability distribution $\tau_{m,k}$ on $T_{m-1,k}$.

\begin{claim} There is an $S_k$-symmetric probability distribution $\tilde{\tau}$ on $T_{m-1,k}$ such that for every $t\in T_{m-1,k}$ the probability $\tilde{\tau}(t)$ is an integer multiple of $1/n$ and furthermore $\vert \tilde{\tau}(t)-\tau_{m,k}(t)\vert \leq m^k/n$ for every $t\in T_{m-1,k}$.
\end{claim}
\begin{proof} For this proof only, let $T_{m-1,k}'$ denote the set of those $t\in T_{m-1,k}$ that are not permutations of $(m-1,0,\dots,0)$. Then the set $T_{m-1,k}\setminus T_{m-1,k}'$ consists precisely of the $k$ permutations of $(m-1,0,\dots,0)$.

We can now define $\tilde{\tau}$ as follows: For every $t\in T_{m-1,k}'$  let us round the value $\tau_{m,k}(t)$ down to the next integer multiple of $k/n$ to obtain $\tilde{\tau}(t)$. It remains to define $\tilde{\tau}(t)$ for $t\in T_{m-1,k}\setminus T_{m-1,k}'$ (that means $t$ is one of the $k$ permutations of $(m-1,0,\dots,0)$). For those $t$, set
\[\tilde{\tau}(t)=\frac{1}{k}\left(1-\sum_{t'\in T_{m-1,k}'}\tilde{\tau}(t')\right)\geq \frac{1}{k}\left(1-\sum_{t'\in T_{m-1,k}'}\tau_{m,k}(t')\right)\geq 0,\]
Then we have
\[\sum_{t\in T_{m-1,k}}\tilde{\tau}(t)=\sum_{t'\in T_{m-1,k}'}\tilde{\tau}(t')+k\cdot \frac{1}{k}\left(1-\sum_{t'\in T_{m-1,k}'}\tilde{\tau}(t')\right)=1\]
and furthermore $\tilde{\tau}(t)\geq 0$ for every $t\in T_{m-1,k}$. Thus, $\tilde{\tau}$ is indeed a probability distribution on $T_{m-1,k}$. It is easy to see that $\tilde{\tau}$ is $S_k$-symmetric. Furthermore,
for every $t\in T_{m-1,k}'$, the probability $\tilde{\tau}(t)$ is an integer multiple of $k/n$, and so in particular of $1/n$. Since $n$ is divisible by $k$, we obtain that $1-\sum_{t'\in T_{m-1,k}'}\tilde{\tau}(t')$ is an integer multiple of $k/n$. Hence $\tilde{\tau}(t)$ is also an integer multiple of $1/n$ if $t\in T_{m-1,k}\setminus T_{m-1,k}'$.

Finally, for every $t\in T_{m-1,k}'$ we have $\vert \tilde{\tau}(t)-\tau_{m,k}(t)\vert \leq k/n\leq m^k/n$. Hence for $t\in T_{m-1,k}\setminus T_{m-1,k}'$ we have
\begin{multline*}
\vert \tilde{\tau}(t)-\tau_{m,k}(t)\vert=\left\vert\frac{1}{k}\left(1-\sum_{t'\in T_{m-1,k}'}\tilde{\tau}(t')\right)-\frac{1}{k}\left(1-\sum_{t'\in T_{m-1,k}'}\tau_{m,k}(t')\right)\right\vert\leq \sum_{t'\in T_{m-1,k}'}\frac{1}{k} \vert \tilde{\tau}(t')-\tau_{m,k}(t')\vert\\
\leq \sum_{t' \in T'_{m-1,k}} \frac{1}{k}\cdot \frac{k}{n}
=\frac{\vert T_{m-1,k}'\vert}{k}\cdot \frac{k}{n}\leq \frac{m^k}{n}.
\end{multline*}
All in all we obtain  $\vert \tilde{\tau}(t)-\tau_{m,k}(t)\vert \leq m^k/n$ for every $t\in T_{m-1,k}$.
\end{proof}

Now, let $\nu$ be the marginal of $\tilde{\tau}$. Then each probability $\nu(a)$ for $a\in \lbrace 0,\dots, m-1\rbrace$ satisfies
\[\nu(a)=\sum_{\substack{a_2,\dots,a_k\in \lbrace 0,\dots,m-1\rbrace\\a+a_2+\dots+a_k=m-1}}\tilde{\tau}(a,a_2,\dots,a_k).\]
So $\nu(a)$ is the sum of several probabilities $\tilde{\tau}(t)$ for certain $t\in T_{m-1,k}$ and is therefore an integer multiple of $1/n$. Furthermore, by Lemma \ref{lemma-expectation-marginal} we have $\E(\nu)=\E(\mu(\tilde{\tau}))=(m-1)/k$. Thus, $\nu$ satisfies the first two properties listed in Proposition \ref{proprounding}.

Since $\nu_{m,k}$ is the marginal of $\tau_{m,k}$, we also have 
\[\nu_{m,k}(a)=\sum_{\substack{a_2,\dots,a_k\in \lbrace 0,\dots,m-1\rbrace\\a+a_2+\dots+a_k=m-1}}\tau_{m,k}(a,a_2,\dots,a_k).\]
Hence for every $a\in \lbrace 0,\dots, m-1\rbrace$,
\[\vert \nu(a)-\nu_{m,k}(a)\vert\leq \sum_{\substack{a_2,\dots,a_k\in \lbrace 0,\dots,m-1\rbrace\\a+a_2+\dots+a_k=m-1}}\vert \tilde{\tau}(a,a_2,\dots,a_k)-\tau_{m,k}(a,a_2,\dots,a_k)\vert\leq \sum_{\substack{a_2,\dots,a_k\in \lbrace 0,\dots,m-1\rbrace\\a+a_2+\dots+a_k=m-1}}\frac{m^k}{n}\leq \frac{m^{2k-1}}{n}.\]
Thus, we obtain
\[\Vert \nu-\nu_{m,k}\Vert_1=\sum_{a=0}^{m-1}\vert \nu(a)-\nu_{m,k}(a)\vert\leq m\cdot \frac{m^{2k-1}}{n}=\frac{m^{2k}}{n}.\]
Note that for every $a\in \lbrace 0,\dots, m-1\rbrace$ we have (recall $0<\gamma_{m,k}<1$)
\[\nu_{m,k}(a)=\frac{\gamma_{m,k}^a}{1+\gamma_{m,k}+\dots+\gamma_{m,k}^{m-1}}\geq \frac{\gamma_{m,k}^{m-1}}{m}.\]
Thus, as long as $n$ is sufficiently large, we obtain for every $a\in \lbrace 0,\dots, m-1\rbrace$
\[\nu(a)\geq \nu_{m,k}(a)-\vert\nu(a)-\nu_{m,k}(a)\vert\geq \frac{\gamma_{m,k}^{m-1}}{m}-\frac{m^{2k}}{n}> \frac{\gamma_{m,k}^{m-1}}{2m}.\]
Now Lemma \ref{lemmaentropiesclose} yields
\[\vert \HH(\nu)-\HH(\nu_{m,k})\vert\leq \frac{m^{2k}}{n}\log\left(\frac{2m}{\gamma_{m,k}^{m-1}}\right).\]
So if we set $C_{m,k}=m^{2k}\log(2m/\gamma_{m,k}^{m-1})$, we obtain using Lemma \ref{lemmanupk}
\[\HH(\nu)\geq \HH(\nu_{m,k})-\vert \HH(\nu)-\HH(\nu_{m,k})\vert\geq \log \Gamma_{m,k}-C_{m,k}/n.\]
Thus, $\nu$ satisfies the third property in Proposition \ref{proprounding}.

We now need to find a probability distribution $\tau$ on $T_{m-1,k}$ that satisfies the last three properties listed in Proposition \ref{proprounding}. We will define $\tau$ in several steps.
Note that as long as $n$ is large enough, we have $c_{m,k}-m^k/n>c_{m,k}/2$ and therefore
\[\tilde{\tau}(t)\geq \tau_{m,k}(t)-\vert \tilde{\tau}(t)-\tau_{m,k}(t)\vert\geq c_{m,k}-\frac{m^k}{n}>\frac{c_{m,k}}{2}\]
for every $t\in T_{m-1,k}$.

Recall that $\nu$ is the marginal of $\tilde{\tau}$, so the $k$ coordinate projections of $\tilde{\tau}$ are all equal to $\nu$. Now, let $\tau_0$ be a probability distribution on $T_{m-1,k}$ with maximal entropy under the condition that all of the $k$ coordinate projections of $\tau_0$ are equal to $\nu$. For every $\sigma\in S_k$, let $\tau_0^\sigma$ be the probability distribution on $T_{m-1,k}$ given by permuting $\tau_0$ according to $\sigma$, that is $\tau_0^\sigma(t)=\tau_0(t^\sigma)$ for every $t\in T_{m-1,k}$. Then the $k$ coordinate projections of $\tau_0^\sigma$ are permutations of the $k$ coordinate projections of $\tau$ and therefore also all equal to $\nu$. Furthermore, by symmetry, we have $\HH(\tau_0^\sigma)=\HH(\tau_0)$ for each $\sigma\in S_k$. Now set
\[\tau_1=\frac{1}{k!}\sum_{\sigma\in S_k}\tau_0^\sigma.\]
Then $\tau_1$ is an $S_k$-symmetric probability distribution on $T_{m-1,k}$ and each of its $k$ coordinate projections equals $\nu$, so $\tau_1$ has marginal $\nu$. Furthermore, by concavity of the entropy function, we have
\begin{equation}\label{eq-tau0-tau1}
\HH(\tau_1)=\HH\left(\frac{1}{k!}\sum_{\sigma\in S_k}\tau_0^\sigma\right)
\geq \frac{1}{k!}\sum_{\sigma\in S_k}\HH(\tau_0^\sigma)=\frac{1}{k!}\sum_{\sigma\in S_k}\HH(\tau_0)=\HH(\tau_0).
\end{equation}
By the choice of $\tau_0$ this actually implies $\HH(\tau_1)=\HH(\tau_0)$, although this is not relevant for our argument.

Since we assumed that $n$ is sufficiently large, we may assume $n>2(d_{m,k}+1)/c_{m,k}$ and set
\[\tau_2=\frac{2(d_{m,k}+1)}{c_{m,k}n} \tilde{\tau}+\left(1- \frac{2(d_{m,k}+1)}{c_{m,k}n}\right)\tau_1.\]
Then $\tau_2$ is also an $S_k$-symmetric probability distribution on $T_{m-1,k}$ with marginal $\nu$ (since both $\tilde{\tau}$ and $\tau_1$ have these properties). Furthermore, by concavity of the entropy function, we have
\begin{multline}\label{eq-tau1-tau2}
\HH(\tau_2)\geq \frac{2(d_{m,k}+1)}{c_{m,k}n} \HH(\tilde{\tau})+\left(1- \frac{2(d_{m,k}+1)}{c_{m,k}n}\right)\HH(\tau_1)\geq \HH(\tau_1)-\frac{2(d_{m,k}+1)}{c_{m,k}n}\HH(\tau_1)\\
\geq \HH(\tau_1)-\frac{2k(d_{m,k}+1)\log m}{c_{m,k}n},
\end{multline}
where we used that $\HH(\tau_1)\leq \log(\vert T_{m-1,k}\vert)\leq \log(m^k)=k\log m$.

Recall that for sufficiently large $n$ we have $\tilde{\tau}(t)\geq c_{m,k}/2$ for all $t\in T_{m-1,k}$ and hence
\[\tau_2(t)=\frac{2(d_{m,k}+1)}{c_{m,k}n} \tilde{\tau}(t)+\left(1- \frac{2(d_{m,k}+1)}{c_{m,k}n}\right)\tau_1(t)\geq \frac{2(d_{m,k}+1)}{c_{m,k}n}\cdot \frac{c_{m,k}}{2}=\frac{d_{m,k}+1}{n}\]
for each $t\in T_{m-1,k}$.
Since both $\tau_2$ and $\tilde{\tau}$ are $S_k$-symmetric probability distributions on $T_{m-1,k}$ with marginal $\nu$, their difference $\tau_2-\tilde{\tau}: T_{m-1,k}\to \mathbb{R}$ lies in the space $U$ defined above. So by Lemma \ref{lemma-lattice-U} we can find some $u\in U$ with $\Vert (\tau_2-\tilde{\tau})-u\Vert_1\leq d_{m,k}/n$ such that $u(t)$ is an integer multiple of $1/n$ for each $t\in T_{m-1,k}$. Now define $\tau: T_{m-1,k}\to \mathbb{R}$ by
\[\tau=\tilde{\tau}+u.\]
For each $t\in T_{m-1,k}$, both $\tilde{\tau}(t)$ and $u(t)$ are integer multiples of $1/n$, and hence so is $\tau(t)$.
We also have that
\[\Vert \tau_2-\tau\Vert_1=\Vert \tau_2-\tilde{\tau}-u\Vert_1\leq d_{m,k}/n.\]
In particular, for every $t\in T_{m-1,k}$ we have $\vert \tau(t)-\tau_2(t)\vert \leq d_{m,k}/n$ and therefore
\[\tau(t)\geq \tau_2(t)-\frac{d_{m,k}}{n}\geq \frac{d_{m,k}+1}{n}-\frac{d_{m,k}}{n}=\frac{1}{n}.\]
In particular, all values of $\tau$ are positive. Furthermore $\sum_{t\in T_{m-1,k}}\tilde{\tau}(t)=1$ and $\sum_{t\in T_{m-1,k}}u(t)=0$, hence $\sum_{t\in T_{m-1,k}}\tau(t)=1$, so $\tau$ is a probability distribution on $T_{m-1,k}$. Since both $\tilde{\tau}$ and $u$ are $S_k$-symmetric, the probability distribution $\tau$ is also $S_k$-symmetric. Finally, the marginal of $\tau$ is $\nu$, because for every $a\in \lbrace 0,\dots,m-1\rbrace$ we have, using (\ref{eq-property-u}),
\[\nu(a)=\sum_{\substack{a_2,\dots,a_\l\in \lbrace 0,\dots,r\rbrace\\a+a_2+\dots+a_\l=r}}\tilde{\tau}(a,a_2,\dots,a_\l)=\sum_{\substack{a_2,\dots,a_\l\in \lbrace 0,\dots,r\rbrace\\a+a_2+\dots+a_\l=r}}\tau(a,a_2,\dots,a_\l).\]

Note that Lemma \ref{lemmaentropiesclose} yields
\begin{equation}\label{eq-tau2-tau}
\vert \HH(\tau_2)-\HH(\tau)\vert\leq \Vert \tau_2-\tau\Vert_1\log n\leq (d_{m,k}\log n)/n
\end{equation}

It remains to check the last condition in Proposition \ref{proprounding}. Let $\tau'$ be a probability distribution on $T_{m-1,k}$ such that each of the $k$ coordinate projections on $\tau'$ equals $\nu$. By the choice of $\tau_0$, we have $\HH(\tau')\leq \HH(\tau_0)$ . Hence from (\ref{eq-tau0-tau1}), (\ref{eq-tau1-tau2}) and (\ref{eq-tau2-tau}) we obtain
\[\HH(\tau')\leq \HH(\tau_0)\leq \HH(\tau_1)\leq \HH(\tau_2)+\frac{2k(d_{m,k}+1)\log m}{c_{m,k}n}\leq \HH(\tau)+\frac{d_{m,k}\log n}{n}+\frac{2k(d_{m,k}+1)\log m}{c_{m,k}n}.\]
If $n$ is sufficiently large, this yields $\HH(\tau')\leq \HH(\tau) + (D_{m,k}\log n)/n$ with $D_{m,k}=2d_{m,k}$.
\end{proof}

\subsection{Proof of Lemma \ref{lemma-linearalgebra}}

\begin{proof}[Proof of Lemma \ref{lemma-linearalgebra}]
Let us examine the span of $\left(\tilde{g}_1(x_1),\dots,\tilde{g}_k(x_k),\tilde{g}_1(x_1'),\dots,\tilde{g}_k(x_k') \right)$ in $\mathbb{Q}^{n+k-1}$. To simplify notation, let $u_i=\tilde g_i(x_i)$ and $u_i'=\tilde g_i(x_i')$. First, recall that
$x_1+\dots+x_k=x_1'+\dots+x_k'=(m-1)\cdot \ones$ implies $u_1+\dots +u_k=0$ and $u_1'+\dots +u_k'=0$. 
Therefore, we have
\begin{multline*}
\spn_\QQ \left(u_1,\dots,u_k,u_1',\dots,u_k'\right)=\spn_\QQ \left(u_1,\dots,u_{k-1},u_1',\dots,u_{k-1}'\right)\\
=\spn_\QQ \left(u_1,\dots,u_{k-1},u_1-u_1',\dots,u_{k-1}-u_{k-1}'\right).
\end{multline*}
Let us examine the last $k-1$ coordinates of the vectors on the right-hand side. Each $u_i=\tilde g_i(x_i)$ for $1\leq i\leq k-1$ is the $i$-th standard basis vectors when restricted to the last $k-1$ coordinates. On the other hand, each $u_i-u_i'=\tilde{g}_{i}(x_{i})-\tilde{g}_{i}(x_{i}')$ has only zeros in the last $k-1$ coordinates. Hence all linear relations between the vectors on the right-hand side above are between $u_1-u_1',\dots, u_{k-1}-u_{k-1}'$. This implies that
\begin{multline*}
\dim\spn_\QQ \left(u_1,\dots,u_k,u_1',\dots,u_k'\right)
=\dim \spn_\QQ \left(u_1,\dots,u_{k-1},u_1-u_1',\dots,u_{k-1}-u_{k-1}'\right)\\
=k-1+\dim\spn_\QQ \left(u_1-u_1',\dots,u_{k-1}-u_{k-1}'\right).
\end{multline*}
By the definition of $\tilde{g}_i$ we actually know that for $1\leq i\leq k-1$, the vector $u_i-u_i'=\tilde{g}_i(x_i)-\tilde{g}_i(x_i')\in \mathbb{Z}^{n+k-1}$ is the same as $x_i-x_i'\in \mathbb{Z}^n$ with $k-1$ zeros attached at the end. So
\[\dim\spn_\QQ \left(u_1-u_1',\dots,u_{k-1}-u_{k-1}'\right)=\dim\spn_\QQ \left(x_1-x_1',\dots,x_{k-1}-x_{k-1}'\right)\]
and therefore
\begin{multline*}
\dim \spn_\QQ \left(\tilde{g}_1(x_1),\dots,\tilde{g}_k(x_k),\tilde{g}_1(x_1'),\dots,\tilde{g}_k(x_k')\right)=\dim\spn_\QQ \left(u_1,\dots,u_k,u_1',\dots,u_k'\right)\\
=k-1+\dim\spn_\QQ \left(x_1-x_1',\dots,x_{k-1}-x_{k-1}'\right)
=k-1+\dim\spn_\QQ \left(x_1-x_1',\dots ,x_k-x_k'\right),
\end{multline*}
where in the last step we used that $(x_1-x_1')+\dots +(x_k-x_k')=0$.

It remains to show that  
\[\dim\spn_{\FP}\left(g_1(x_1),\dots ,g_k(x_k),g_1(x_1'),\dots,g_k(x_k')\right)=\dim \spn_\QQ \left(\tilde{g}_1(x_1),\dots,\tilde{g}_k(x_k),\tilde{g}_1(x_1'),\dots,\tilde{g}_k(x_k')\right).\]
Recall that by definition the vectors on the left-hand side are just the projections of the vectors on the right-hand side from $\mathbb{Z}^{n+k-1}$ to $\FPnk$. Hence the dimension on the left-hand side is at most as large as the dimension on the right-hand side. However, if the dimension on the right-hand side is $\l$, we can take $\l$ independent vectors from the set $\tilde{g}_1(x_1),\dots,\tilde{g}_k(x_k),\tilde{g}_1(x_1'),\dots,\tilde{g}_k(x_k')$. Consider the $(\l \times (n+k-1))$-matrix (with entries in $\mathbb{Z}$) whose rows are the chosen $\l$ vectors. As the rank of that matrix over $\mathbb{Q}$ is $\l$, there exists an $(\l \times \l)$-submatrix whose determinant is a nonzero integer. Since this determinant has absolute value at most $\l!(m-1)^\l \le (2k)!(m-1)^{2k} < P$, this implies that the determinant must be nonzero over $\FP$. So the chosen $\l$ vectors are also independent over $\FP$ and the dimension on the right-hand side is at least $\l$. This proves the desired equality of the two dimensions.
\end{proof}

\section{Proof of Theorem \ref{theo-marginal}}
\label{sect-marginal}

The next three sections are devoted to proving Theorem \ref{theo-marginal}. From now on, we consider $k\geq 3$ to be fixed.

\subsection{A generalization of Theorem \ref{theo-marginal}}

Instead of proving Theorem \ref{theo-marginal} directly, we will prove the following more general statement that applies to a certain class of probability distributions on $\lbrace 0,\dots,n\rbrace$ with expectation $n/k$. Here $n\geq 0$ is any non-negative integer. We will take $n=m-1$ and use the fact that $\nu_{m,k}$ has expectation $(m-1)/k$ to obtain Theorem \ref{theo-marginal}.

\begin{theorem}\label{theo-strict}
Let $n\geq 0$ be an integer and let $\psi$ be a probability distribution on $\lbrace 0,\dots,n\rbrace$ with expectation $n/k$ that satisfies $\psi(0)> \psi(1)> \dots > \psi(n)>0$. If $n\geq k$, let us also assume that $2\psi(\lfloor n/k\rfloor)<\psi(\lfloor n/k\rfloor-1)+\psi(\lceil n/k\rceil)$. 
Then $\psi$ occurs as the marginal of an $S_k$-symmetric probability distribution $\tau$ on $T_{n,k}$ with $\tau(t)>0$ for every $t\in T_{n,k}$.
\end{theorem}

\begin{proof}[Proof of Theorem \ref{theo-marginal} assuming Theorem \ref{theo-strict}] Let $n=m-1$ and recall that $\nu_{m,k}$ is a probability distribution on $\lbrace 0,\dots,m-1\rbrace$ with expectation $(m-1)/k$ (see Lemma \ref{lemmanupk}). Since
\[\nu_{m,k}(i)=\frac{\gamma_{m,k}^i}{1+\gamma_{m,k}+\dots+\gamma_{m,k}^{m-1}}.\]
for $i=0,\dots,m-1$ and $0<\gamma_{m,k}<1$, we clearly have $\nu_{m,k}(0)> \nu_{m,k}(1)> \dots > \nu_{m,k}(m-1)>0$. So in order to be able to apply Theorem \ref{theo-strict}, we just need to check that
\[2\nu_{m,k}(\lfloor (m-1)/k\rfloor)< \nu_{m,k}(\lfloor (m-1)/k\rfloor-1)+\nu_{m,k}(\lceil (m-1)/k\rceil)\]
if $m-1\geq k$. If $m-1$ is divisible by $k$, this is clearly true. So assume $m-1\geq k$ and that $m-1$ is not divisible by $k$. Setting $\l=\lfloor (m-1)/k\rfloor$ to simplify notation (note that $\l\geq 1$), we need to check that
\[2\frac{\gamma_{m,k}^\l}{\gamma_{m,k}^0+\dots+\gamma_{m,k}^{m-1}} < \frac{\gamma_{m,k}^{\l-1}}{\gamma_{m,k}^0+\dots+\gamma_{m,k}^{m-1}}+\frac{\gamma_{m,k}^{\l+1}}{\gamma_{m,k}^0+\dots+\gamma_{m,k}^{m-1}}.\]
This is indeed true, since $2\gamma_{m,k}^\l< \gamma_{m,k}^{\l-1}+\gamma_{m,k}^{\l+1}$.
Thus, $\nu_{m,k}$ satisfies the assumptions of Theorem \ref{theo-strict} with $n=m-1$ and therefore occurs as the marginal of an $S_k$-symmetric probability distribution $\tau_{m,k}$ on $T_{m-1,k}$ with $\tau_{m,k}(t)>0$ for every $t\in T_{m-1,k}$
\end{proof}

In the next subsection, we will introduce scaled distributions, following  Pebody \cite[Section 2]{PEBODY16}. They provide a useful framework for the proof of Theorem \ref{theo-strict}. In the third subsection we will state Proposition \ref{propo-tame-main}, the main proposition for the proof of Theorem \ref{theo-strict}, as well as the key lemmas for the proof of this proposition. All of this will be formulated in the framework of scaled distributions introduced in the next subsection. In the fourth subsection we will finally see how Theorem \ref{theo-strict} follows from Proposition \ref{propo-tame-main}.

Section \ref{sect-marginal-proof} will be devoted to proving Proposition \ref{propo-tame-main} and Section \ref{sect-lemmas2} to proving the key lemmas stated in  Subsection \ref{sect-overview}. This will then complete the proof of Theorem \ref{theo-strict} and thereby establish Theorem \ref{theo-marginal}.

\subsection{Scaled distributions}
\label{sect-scaled}

Here, we will introduce scaled distributions, the framework in which the proof of Theorem \ref{theo-strict} will operate. Everything in this subsection follows the first half of Section 2 of Pebody's paper \cite{PEBODY16}. Throughout this subsection, $n$ is an arbitrary non-negative integer.

\begin{definition}[\cite{PEBODY16}] A \emph{scaled distribution} $\psi$ on $\lbrace 0,\dots,n\rbrace$ is a map $\lbrace 0,\dots,n\rbrace\to \mathbb{R}_{\geq 0}$. A scaled distribution $\psi$ on $\lbrace 0,\dots,n\rbrace$ has \emph{mean} $n/k$ if $\sum_{i=0}^{n} i\psi(i)=\frac{n}{k} \sum_{i=0}^{n}\psi(i)$.
\end{definition}

If $\psi$ is a scaled distribution that can be written as a linear combination of scaled distributions $\psi_1$ and $\psi_2$, and if $\psi_1$ and $\psi_2$ both have mean $n/k$, then $\psi$ also has mean $n/k$.

Note that any probability distribution $\psi$ on $\lbrace 0,\dots,n\rbrace$ can be interpreted as a scaled distribution. Furthermore, to any non-zero scaled distribution $\psi$ on $\lbrace 0,\dots,n\rbrace$ we can associate an actual probability distribution $\overline{\psi}$ by setting
\[\overline{\psi}(i)=\frac{\psi(i)}{\psi(0)+\dots+\psi(n)}.\]
Note that $\psi$ has mean $n/k$ (according to the definition above) if and only if $\overline{\psi}$ has expectation $n/k$ (according to the usual definition in probability theory).

Following Pebody \cite{PEBODY16}, let a scaled distribution on $\lbrace 0,\dots,n\rbrace$ be called \emph{$n$-simple} if it is of the form $\indic_{a_1}+\indic_{a_2}+\dots+\indic_{a_k}$ for some $(a_1, a_2,\dots,a_k)\in T_{n,k}$ (that  means $a_1, a_2,\dots ,a_k\in \lbrace 0,\dots,n\rbrace$ with $a_1+a_2+\dots+a_k=n$). Note that any $n$-simple scaled distribution has mean $n/k$.

The following lemma is a variation of Lemma 4 in \cite{PEBODY16}. Although Lemma 4 in \cite{PEBODY16} is only for the case $k=3$ and has a slightly different statement, basically the same proof works here.

\begin{lemma}[see Lemma 4 in \cite{PEBODY16}]\label{lemma-pebody-simple}
Let $\psi$ be a non-zero scaled distribution on $\lbrace 0,\dots,n\rbrace$, and let $\overline{\psi}$ be the associated probability distribution as defined above. Then the following two statements are equivalent.
\begin{itemize}
\item[(1)] The scaled distribution $\psi$ can be written as a positive linear combination of all the $n$-simple scaled distributions $\indic_{a_1}+\indic_{a_2}+\dots+\indic_{a_k}$ for all $(a_1, a_2,\dots,a_k)\in T_{n,k}$.
\item[(2)] The probability distribution  $\overline{\psi}$ occurs as the marginal of an $S_k$-symmetric probability distribution $\tau$ on $T_{n,k}$ with $\tau(t)>0$ for every $t\in T_{n,k}$.
\end{itemize}
\end{lemma}
\begin{proof} Note that for any $S_k$-symmetric symmetric probability distribution $\tau$ on $T_{n,k}$, its marginal $\mu(\tau)$ is given by
\[\mu(\tau)=\sum_{(a_1,\dots,a_k)\in T_{n,k}}\tau(a_1,\dots,a_k)\indic_{a_1}=\dots=\sum_{(a_1,\dots,a_k)\in T_{n,k}}\tau(a_1,\dots,a_k)\indic_{a_k}.\]

First let us assume (1), so 
\[\psi=\sum_{(a_1,\dots,a_k)\in T_{n,k}}\lambda(a_1,\dots,a_k)\cdot (\mathbf{1}_{a_1}+\dots+\mathbf{1}_{a_k})\]
with $\lambda(a_1,\dots,a_k)> 0$ for all $(a_1,\dots,a_k)\in T_{n,k}$. Note that
\[\psi(0)+\dots+\psi(n)=\sum_{(a_1,\dots,a_k)\in T_{n,k}}k\lambda(a_1,\dots,a_k)=k\sum_{t\in T_{n,k}}\lambda(t).\]
Now consider the probability distribution $\tau$  on $T_{n,k}$ given by
\[\tau(a_1,\dots,a_k)=\frac{\sum_{\sigma\in S_k}\lambda((a_1,\dots,a_k)^\sigma)}{k!\sum_{t\in T_{n,k}}\lambda (t)}.\]
Clearly, $\tau$ is $S_k$-symmetric and $\tau(t)>0$ for every $t\in T_{n,k}$. Furthermore its marginal $\mu(\tau)$ is given by
\begin{multline*}
\mu(\tau)=\frac{1}{k}\sum_{(a_1,\dots,a_k)\in T_{n,k}}\tau(a_1,\dots,a_k)\cdot(\indic_{a_1}+\dots+\indic_{a_k})\\
=\frac{\sum_{(a_1,\dots,a_k)\in T_{n,k}}\sum_{\sigma\in S_k}\lambda((a_1,\dots,a_k)^\sigma)\cdot(\indic_{a_1}+\dots+\indic_{a_k})}{k\cdot k!\sum_{t\in T_{n,k}}\lambda (t)}\\
=\frac{k!\sum_{(a_1,\dots,a_k)\in T_{n,k}}\lambda(a_1,\dots,a_k)\cdot(\indic_{a_1}+\dots+\indic_{a_k})}{k!(\psi(0)+\dots+\psi(n))}=\frac{\psi}{\psi(0)+\dots+\psi(n)}=\overline{\psi}.
\end{multline*}
Thus $\overline{\psi}$ is the marginal of the $S_k$-symmetric probability distribution $\tau$ on $T_{n,k}$.

For the converse, assume that $\overline{\psi}$  is the marginal of some $S_k$-symmetric probability distribution $\tau$ on $T_{n,k}$ with $\tau(t)>0$ for every $t\in T_{n,k}$. Then
\[\overline{\psi}=\sum_{(a_1,\dots,a_k)\in T_{n,k}}\tau(a_1,\dots,a_k)\indic_{a_1}=\dots=\sum_{(a_1,\dots,a_k)\in T_{n,k}}\tau(a_1,\dots,a_k)\indic_{a_k},\]
so
\[\overline{\psi}=\sum_{(a_1,\dots,a_k)\in T_{n,k}}\frac{\tau(a_1,\dots,a_k)}{k}(\indic_{a_1}+\dots+\indic_{a_k}).\]
Thus, $\overline{\psi}$ is a positive linear combination of all the $n$-simple scaled distributions. Now, by rescaling we can see that $\psi$ is also a positive linear combination of all the $n$-simple scaled distributions.
\end{proof}

\subsection{Proof Overview}
\label{sect-overview}

In this subsection, we will state the key ingredients for the proof of Theorem \ref{theo-strict}. First, we make the following definition that will be at the heart of the proof.

\begin{definition}\label{defi-tame}
Let $n\geq 0$ be an integer. A scaled distribution $\psi$ on $\lbrace 0,\dots,n\rbrace$ is called \emph{$n$-tame} if the following four statements hold.
\begin{itemize}
\item[(i)] $\psi$ has mean $n/k$, that is $\sum_{i=0}^{n} i\psi(i)=\frac{n}{k} \sum_{i=0}^{n}\psi(i)$.
\item[(ii)] If $n\geq 1$, then $\psi(1)\geq \psi(2)\geq \dots\geq \psi(n)$.
\item[(iii)] If $n\geq k$, then $\psi(0)\geq (k-1)\psi(n)+(k-2)\psi(n-1)+\dots+\psi(n-k+2)$.
\item[(iv)] If $n\geq 2k$, then $2\psi(\lfloor n/k\rfloor)\leq \psi(\lfloor n/k\rfloor-1)+\psi(\lceil n/k\rceil)$.
\end{itemize}
\end{definition}

In other words, if $n=0$, then $\psi$ only needs to satisfy condition (i). If $1\leq n<k$, the $\psi$ needs to satisfy (i) and the inequality in (ii). If $k\leq n<2k$, then in addition $\psi$ also needs to satisfy the inequality in (iii). And if $n\geq 2k$, then $\psi$ needs to satisfy (i) and all the three inequalities in (ii), (iii) and (iv).
Note that if $n$ is divisible by $k$, then condition (iv) is already implied by condition (ii), since $\lfloor n/k\rfloor=\lceil n/k\rceil$.

We are now ready to state the main proposition for the proof of Theorem \ref{theo-strict}:

\begin{proposition}\label{propo-tame-main}
Every $n$-tame scaled distribution is a non-negative linear combination of $n$-simple scaled distributions.
\end{proposition}

The proof of Proposition \ref{propo-tame-main} will be by strong induction on $n$ with base cases $0\leq n\leq k$, and reducing from $n$ to $n-k$ in each step.
We will prove Proposition \ref{propo-tame-main} in Section \ref{sect-marginal-proof}. Here, we just state the key lemmas for the proof of Proposition \ref{propo-tame-main}. These lemmas will be proved in Section \ref{sect-lemmas2}.

\begin{lemma}\label{lemma-alpha-j} Let $n\geq 1$ and let $j\in \lbrace 1,\dots,n\rbrace$ with $j+1\geq 2n/k$. Then there exists a scaled distribution $\alpha_j$ on $\lbrace 0,\dots,n\rbrace$ satisfying the following six properties.
\begin{itemize}
\item[(a)] $\alpha_j$ is a non-negative linear combination of $n$-simple scaled distributions.
\item[(b)] $\alpha_j$ has mean $n/k$.
\item[(c)] $\alpha_j(1)\leq \alpha_j(2)\leq \dots\leq \alpha_j(j)$.
\item[(d)] $\alpha_j(j+1)=\alpha_j(j+2)= \dots= \alpha_j(n)=0$.
\item[(e)] $\alpha_j(j)\neq 0$.
\item[(f)] If $n\geq 2k$, then $2\alpha_j(\lfloor n/k\rfloor)\geq \alpha_j(\lfloor n/k\rfloor-1)+\alpha_j(\lceil n/k\rceil)$.
\end{itemize}
\end{lemma}

\begin{lemma}\label{lemma-slack}Let $n\geq k$ be an integer and let $\psi$ be an $n$-tame scaled distribution on $\lbrace 0,\dots,n\rbrace$. Then there exists an $n$-tame scaled distribution $\phi$ with
\begin{equation}\label{eq-phi-equal}
\phi(0)= (k-1)\phi(n)+(k-2)\phi(n-1)+\dots+\phi(n-k+2)
\end{equation}
(this means that $\phi$ has equality in condition (iii) Definition \ref{defi-tame}) and such that $\psi-\phi$ is a non-negative linear combination of $n$-simple scaled distributions.
\end{lemma}

\begin{lemma}\label{lemma-inequality}Let $n\geq 2k$ be an integer and let $\phi$ be an $n$-tame scaled distribution on $\lbrace 0,\dots,n\rbrace$ with
\[\phi(0)= (k-1)\phi(n)+(k-2)\phi(n-1)+\dots+\phi(n-k+2).\]
Then
\begin{multline*}
\phi(1)\geq \phi(n-1)+2\phi(n-2)+\dots+(k-2)\phi(n-k+2)\\
+(k-1)\phi(n-k+1)+(k-2)\phi(n-k)+\dots+\phi(n-2k+3).
\end{multline*}
\end{lemma}

\begin{remark}Recall that Proposition \ref{propo-tame-main} states that every $n$-tame scaled distribution $\psi$ is a non-negative linear combination of $n$-simple scaled distributions. It is not hard to see that conditions (i) and (iii) in Definition \ref{defi-tame} are necessary in order for $\psi$ to be a non-negative linear combination of $n$-simple scaled distributions. Conditions (ii) and (iv) on the other hand were chosen because they make the inductive proof of Proposition \ref{propo-tame-main} work. Condition (iv) seems very unnatural and is in fact only needed to make Lemma \ref{lemma-inequality} true. The inequality there can fail by just a slight amount if one does not have this condition (and if $n$ is not divisible by $k$). One can probably replace (iv) by a different condition to ensure the inequality in Lemma \ref{lemma-inequality} (this particular condition was chosen because it is easy to check for the probability distribution $\nu_{m,k}$). However, omitting condition (iv) entirely would make Proposition \ref{propo-tame-main} false.
\end{remark}

\subsection{Deriving Theorem \ref{theo-strict} from Proposition \ref{propo-tame-main}}
\label{check-actual-conj}

First, let us prove the following corollary of Lemma \ref{lemma-inequality}.

\begin{corollary}\label{coro1}
Let $n\geq 0$ be an integer and let $\psi$ be a scaled distribution on $\lbrace 0,\dots,n\rbrace$ with mean $n/k$ that satisfies $\psi(0)\geq \psi(1)\geq \dots\geq \psi(n)$. If $n\geq k$, let us also assume that $2\psi(\lfloor n/k\rfloor)\leq \psi(\lfloor n/k\rfloor-1)+\psi(\lceil n/k\rceil)$. Then $\psi$ is $n$-tame.
\end{corollary}

\begin{proof} We  just need to show that condition (iii) in Definition \ref{defi-tame} is satisfied, all the other conditions are clear from the assumptions of the corollary.
So let us assume that $n\geq k$ (otherwise (iii) is vacuous). Let us define a scaled distribution $\phi$ on $\lbrace 0,\dots, n+k\rbrace$ by taking $\phi(0)=0$ and $\phi(n+2)=\phi(n+3)=\dots=\phi(n+k)=0$ and $\phi(i)=\psi(i-1)$ for $i=1,\dots,n+1$.
We claim that $\phi$ is $(n+k)$-tame:
\begin{itemize}
\item[(i)] $\phi$ has mean $(n+k)/k$, because
\[\sum_{i=0}^{n+k}i\phi(i)=\sum_{i=1}^{n+1}i\phi(i)=\sum_{i=0}^{n}(i+1)\psi(i)=\frac{n}{k}\sum_{i=0}^{n}\psi(i)+\sum_{i=0}^{n}\psi(i)=\frac{n+k}{k}\sum_{i=1}^{n+1}\phi(i)=\frac{n+k}{k}\sum_{i=0}^{n+k}\phi(i).\]
Here we used that $\psi$ has mean $n/k$.
\item[(ii)] From $\psi(0)\geq \psi(1)\geq \dots\geq \psi(n)$ we obtain $\phi(1)\geq \phi(2)\geq \dots\geq \phi(n+1)$, and together with $\phi(n+2)=\phi(n+3)=\dots=\phi(n+k)=0$ this gives $\phi(1)\geq \phi(2)\geq \dots\geq \phi(n+k)$.
\item[(iii)] We have
\begin{equation}\label{phi-all-zero}
\phi(0)= (k-1)\phi(n+k)+(k-2)\phi(n+k-1)+\dots+\phi(n+2),
\end{equation}
because all these terms are zero by the definition of $\phi$.
\item[(iv)] Recall that we assumed $n\geq k$. So by the assumptions on $\psi$ we have $2\psi(\lfloor n/k\rfloor)\leq \psi(\lfloor n/k\rfloor-1)+\psi(\lceil n/k\rceil)$ and therefore
\begin{multline*}
2\phi(\lfloor (n+k)/k\rfloor)=2\phi(\lfloor n/k\rfloor+1)=2\psi(\lfloor n/k\rfloor)\\
\leq\psi(\lfloor n/k\rfloor-1)+\psi(\lceil n/k\rceil)=\phi(\lfloor (n+k)/k\rfloor-1)+\phi(\lceil (n+k)/k\rceil).
\end{multline*}
\end{itemize}
So $\phi$ is indeed $(n+k)$-tame.
Note that $n+k\geq 2k$, since we assumed $n\geq k$. So by (\ref{phi-all-zero}), we can apply Lemma \ref{lemma-inequality} and obtain
\begin{multline*}
\phi(1)\geq \phi(n+k-1)+2\phi(n+k-2)+\dots+(k-2)\phi(n+2)\\
+(k-1)\phi(n+1)+(k-2)\phi(n)+\dots+\phi(n-k+3).
\end{multline*}
When plugging in the definition of $\phi$, we get $\psi(0)\geq (k-1)\psi(n)+(k-2)\psi(n-1)+\dots+\psi(n-k+2)$. Thus, $\psi$ indeed satisfies condition (iii) in Definition \ref{defi-tame}.\end{proof}

Now, we are ready to derive Theorem \ref{theo-strict} from Proposition \ref{propo-tame-main}.

\begin{proof}[Proof of Theorem \ref{theo-strict}]
Let $\psi$ be a probability distribution on $\lbrace 0,\dots,n\rbrace$ with expectation $n/k$ and $\psi(0)> \psi(1)> \dots > \psi(n)>0$. If $n\geq k$, we also assume that $2\psi(\lfloor n/k\rfloor)<\psi(\lfloor n/k\rfloor-1)+\psi(\lceil n/k\rceil)$. We need to prove that $\psi$ occurs as the marginal of an $S_k$-symmetric probability distribution $\tau$ on $T_{n,k}$ with $\tau(t)>0$ for every $t\in T_{n,k}$. By Lemma \ref{lemma-pebody-simple}, it is enough to show that $\psi$, interpreted as a scaled distribution, can be written as a positive linear combination of all the $n$-simple scaled distributions $\indic_{a_1}+\indic_{a_2}+\dots+\indic_{a_k}$ for all $(a_1, a_2,\dots,a_k)\in T_{n,k}$.

So let us interpret $\psi$ as a scaled distribution and let us take some small $x>0$ such that
\[\psi'=\psi-x\cdot \sum_{(a_,\dots,a_k)\in T_{n,k}}(\indic_{a_1}+\dots+\indic_{a_k})\]
is a scaled distribution with $\psi'(0)\geq \psi'(1)\geq \dots\geq \psi'(n)\geq 0$ and with $2\psi'(\lfloor n/k\rfloor)\leq \psi'(\lfloor n/k\rfloor-1)+\psi'(\lceil n/k\rceil)$ if $n\geq k$. Note that $\psi'$ has mean $n/k$, since both $\psi$ and all the $n$-simple scaled distributions $\indic_{a_1}+\dots+\indic_{a_k}$ have mean $n/k$. So by Corollary \ref{coro1}, the scaled distribution $\psi'$ is $n$-tame. Hence, by Proposition \ref{propo-tame-main}, $\psi'$ is a non-negative linear combination on $n$-simple scaled distributions. Now, using $x>0$ and
\[\psi=\psi'+x\cdot \sum_{(a_,\dots,a_k)\in T_{n,k}}(\indic_{a_1}+\dots+\indic_{a_k}),\]
this gives a way to express $\psi$ as a positive linear combination of all the $n$-simple scaled distributions $\indic_{a_1}+\indic_{a_2}+\dots+\indic_{a_k}$ for all $(a_1, a_2,\dots,a_k)\in T_{n,k}$. This finishes the proof of Theorem \ref{theo-strict}.\end{proof}

The next section will be devoted to proving Proposition \ref{propo-tame-main}, assuming the lemmas stated in Subsection \ref{sect-overview}. These lemmas will be proved in Section \ref{sect-lemmas2}.

\section{Proof of Proposition \ref{propo-tame-main}}
\label{sect-marginal-proof}

We will prove Proposition \ref{propo-tame-main} by strong induction on $n$, in each step reducing any instance of the statement for $n$ to an instance of the statement for $n-k$. The base cases $n= 0,\dots,k$ will be treated in the first subsection. The induction step will be performed in the second subsection.

\subsection{The base cases of the induction}

Here we will treat the cases $n=0,\dots,k$ of Proposition \ref{propo-tame-main}.

If $n=0$, then we have as desired
\[\psi = (\psi(0)/k)(\underbrace{\indic_0+\dots+\indic_0}_{k}).\]

Now assume that $1\leq n\leq k$ and keep $n$ fixed. Then Proposition \ref{propo-tame-main} is true by the following claim.

\begin{claim} Let $1\leq n\leq k$. Any scaled distribution $\psi$ on $\lbrace 0,\dots,n\rbrace$ satisfying both
\begin{itemize}
\item[(i)] $\psi$ has mean $n/k$ and
\item[(ii)] $\psi(1)\geq \psi(2)\geq \dots\geq \psi(n)$.
\end{itemize}
can be written as a non-negative linear combination of $n$-simple scaled distributions.
\end{claim}

\begin{proof} 
Suppose that there are counterexamples to the claim for some fixed $1\leq n\leq k$. Then we can find a scaled distribution $\psi$ satisfying (i) and (ii) such that $\psi$ is not a non-negative linear combination of $n$-simple scaled distributions. Clearly, such a $\psi$ is not identically zero. For any such $\psi$, let $j\in \lbrace 0,\dots,n\rbrace$ be chosen maximal with $\psi(j)>0$. Among all possible scaled distributions $\psi$ that contradict the claim, let us choose one where this $j$ is minimal. Thus, $\psi$ is a scaled distribution on $\lbrace 0,\dots,n\rbrace$ satisfying (i) and (ii), but $\psi$ cannot be written as a non-negative linear combination of $n$-simple scaled distributions. Furthermore, we have $\psi(j+1)=\dots=\psi(n)=0$, but $\psi(j)>0$. Finally, by the choice of $\psi$, any scaled distribution $\phi$ satisfying (i) and (ii) and $\phi(j)=\phi(j+1)=\dots=\phi(n)=0$ can be written as a non-negative linear combination of $n$-simple scaled distributions.

Note that $j\neq 0$, because otherwise $\psi$ is supported just on zero and can therefore not have mean $n/k$. Thus, $j\in \lbrace 1,\dots,n\rbrace$ and in particular $j+1\geq 2\geq 2n/k$. Thus, Lemma \ref{lemma-alpha-j} gives a scaled distribution $\alpha_j$ satisfying the conditions (a) to (f). Now choose $x\in \mathbb{R}_{\geq 0}$ maximal such that both $\psi(0)-x\alpha_j(0)\geq 0$ and $\psi(j)-x\alpha_j(j)\geq 0$ (such a maximal $x$ exists since $x=0$ satisfies the conditions, but due to $\alpha_j(j)>0$ by (e) there is an upper bound for $x$). Note that for this maximal $x$ we have $\psi(0)-x\alpha_j(0)=0$ or $\psi(j)-x\alpha_j(j)= 0$.

Now set $\phi=\psi-x\alpha_j$. We claim that $\phi$ is a scaled distribution satisfying (i) and (ii). Note that we have
\[\phi(0)=\psi(0)-x\alpha_j(0)\geq 0\]
and
\[\phi(j)=\psi(j)-x\alpha_j(j)\geq 0.\]
Furthermore by (ii) and (c) we have
$\psi(1)\geq \psi(2)\geq \dots\geq \psi(j)$ and $\alpha_j(1)\leq \alpha_j(2)\leq \dots\leq \alpha_j(j)$, hence
\[\phi(1)\geq \phi(2)\geq \dots\geq \phi(j)\geq 0.\]
From $\psi(j+1)=\dots=\psi(n)=0$ and (d) we can deduce that $\phi(j+1)=\dots=\phi(n)=0$. In particular, we have established that $\phi$ has non-negative values, so it is a scaled distribution. We also have
\[\phi(1)\geq \phi(2)\geq \dots\geq \phi(j)\geq 0=\phi(j+1)=\dots=\phi(n),\]
so $\phi$ satisfies (ii). Finally $\phi=\psi-x\alpha_j$ has mean $n/k$ because both $\psi$ and $\alpha_j$ have mean $n/k$ (see (i) and (b)). Thus, $\phi$ also satisfies (i). 

We claim that $\phi$ is a non-negative linear combination of $n$-simple scaled distributions. Recall that by the choice of $x$ we have $\psi(0)-x\alpha_j(0)=0$ or $\psi(j)-x\alpha_j(j)= 0$. This means $\phi(0)=0$ or $\phi(j)=0$. If $\phi(j)=0$, then $\phi$ is a scaled distribution satisfying (i) and (ii) and $\phi(j)=\phi(j+1)=\dots=\phi(n)=0$. We saw above that by the choice of $\psi$ this indeed implies that $\phi$ is a non-negative linear combination of $n$-simple scaled distributions.

So suppose now that $\phi(0)=0$. Then $\phi$ is supported on a subset of $\lbrace 1,\dots n\rbrace$, but has mean $n/k\leq 1$. This is only possible if $n=k$ and $\phi$ is supported on $\lbrace 1\rbrace$. But then
\[\phi = (\phi(1)/k)(\underbrace{\indic_1+\dots+\indic_1}_{k})\]
and therefore $\phi$ is a non-negative linear combination of $n$-simple scaled distributions.

So we have shown in any case that $\phi$ is a non-negative linear combination of $n$-simple scaled distributions. But now, using (a), we see that $\psi=\phi+x\alpha_j$ is also a non-negative linear combination of $n$-simple scaled distributions. This contradicts our choice of $\psi$. Hence there cannot be any counterexamples to the claim, so the claim is true.\end{proof}

\subsection{The induction step}
We will  now perform the induction step for proving Proposition \ref{propo-tame-main}. Let us assume that $n\geq k+1$ and that we have already proved Proposition \ref{propo-tame-main} for all smaller values of $n$.
Let $\psi$ be an $n$-tame scaled distribution on $\lbrace 0,\dots,n\rbrace$. We need to show that $\psi$ can be written as a non-negative linear combination of $n$-simple scaled distributions.

First, we apply Lemma \ref{lemma-slack} to obtain an $n$-tame scaled distribution $\phi$ satisfying (\ref{eq-phi-equal}) and such that $\psi-\phi$ is a non-negative linear combination of $n$-simple scaled distributions.
If we can write $\phi$ as a non-negative linear combination of $n$-simple scaled distributions, then $\psi=\phi+ (\psi-\phi)$ will also be a non-negative linear combination of $n$-simple scaled distributions. Hence it suffices to show that $\phi$ is a non-negative linear combination of $n$-simple scaled distributions.

If $n\geq 2k$, then we can apply Lemma \ref{lemma-inequality} and obtain
\begin{multline}\label{eq-phi1}
\phi(1)-\phi(n-1)-2\phi(n-2)-\dots-(k-2)\phi(n-k+2)\\
\geq (k-1)\phi(n-k+1)+(k-2)\phi(n-k)+\dots+\phi(n-2k+3).
\end{multline}
Let us now define a new scaled distribution $\thet$ on $\lbrace 0,\dots,n\rbrace$ by
\[\thet = \phi - \sum_{i=0}^{k-2}\phi(n-i)(\indic_{n-i}+\underbrace{\indic_1+\dots+\indic_1}_{i}+\underbrace{\indic_0+\dots+\indic_0}_{k-1-i}).\]
Note that $\thet$ satisfies
\begin{equation}\label{eq-thet1}
\sum_{i=0}^{n} i\thet(i)=\frac{n}{k} \sum_{i=0}^{n}\thet(i),
\end{equation}
since $\thet$ is a linear combination of $\phi$ and certain $n$-simple scaled distributions, all of which have mean $n/k$.
Let us show that $\thet$ indeed has non-negative values. For $\thet(0)$, we have
\[\thet(0)=\phi(0)-(k-1)\phi(n)-(k-2)\phi(n-1)-\dots-\phi(n-k+2)=0,\]
where for the second equality sign we used (\ref{eq-phi-equal}). We also have that
\[\thet(n)=\thet(n-1)=\dots=\thet(n-k+2)=0.\]
Finally, $\thet(i)=\phi(i)\geq 0$ for $2\leq i\leq n-k+1$. 
Therefore, $\thet(i)\geq 0$ for $i=0$ and for $2\leq i\leq n$. Let us now check that $\thet(1)\geq 0$ as well. If $k+1\leq n<2k$, using (\ref{eq-thet1}) and $\thet(0)=0$, we have
\[\sum_{i=2}^{n}\left(i-\frac{n}{k}\right)\thet(i) =  \left(\frac{n}{k}-1\right)\thet(1)+\frac{n}{k}\thet(0)=\left(\frac{n}{k}-1\right)\thet(1).\]
Since the left-hand side is non-negative, the right-hand side must also be non-negative and therefore $\thet(1)\geq 0$.
If $n\geq 2k$, then (\ref{eq-phi1}) implies
\begin{multline}\label{eq-phi2}
\thet(1)=\phi(1)-\phi(n-1)-2\phi(n-2)-\dots-(k-2)\phi(n-k+2)\\
\geq (k-1)\phi(n-k+1)+(k-2)\phi(n-k)+\dots+\phi(n-2k+3)\\
= (k-1)\thet(n-k+1)+(k-2)\thet(n-k)+\dots+\thet(n-2k+3).
\end{multline}
and in particular $\thet(1)\geq 0$. This proves that $\thet$ has non-negative values and is therefore a scaled distribution.

We have already seen that it suffices to prove that $\phi$ is a non-negative linear combination of $n$-simple scaled distributions. Given that
\[\phi=\thet+ \sum_{i=0}^{k-2}\phi(n-i)(\indic_{n-i}+\underbrace{\indic_1+\dots+\indic_1}_{i}+\underbrace{\indic_0+\dots+\indic_0}_{k-1-i}),\]
it is therefore sufficient to show that $\thet$ is a non-negative linear combination of $n$-simple scaled distributions.

Using property (ii) of $\phi$ together with $\thet(i)=\phi(i)$ for $2\leq i\leq n-k+1$, we obtain
\[\thet(2)\geq \thet(3)\geq \dots\geq \thet(n-k+1).\]

Also, if $n\geq 3k$, then $2\leq \lfloor n/k\rfloor-1< \lfloor n/k\rfloor\leq \lceil n/k\rceil\leq n-k+1$ and therefore property (iv) of $\phi$ implies
\[2\thet(\lfloor n/k\rfloor)\leq \thet(\lfloor n/k\rfloor-1)+\thet(\lceil n/k\rceil).\]

Now, let us define a scaled distribution $\eta$ on $\lbrace 0,\dots, n-k\rbrace$ by $\eta(i)=\thet(i+1)$ for $i=0,\dots, n-k$. We will prove that $\eta$ is an $(n-k)$-tame scaled distribution and then use the induction assumption. So let us check that $\eta$ satisfies conditions (i) to (iv) for an $(n-k)$-tame scaled distribution (see Definition \ref{defi-tame}):

\begin{itemize}
\item[(i)] Recall that $\thet(0)=0$ and $\thet(n)=\thet(n-1)=\dots=\thet(n-k+2)=0$. Now, using (\ref{eq-thet1}), we obtain
\begin{multline*}
\sum_{i=0}^{n-k} i\eta(i)=\sum_{i=1}^{n-k+1} (i-1)\thet(i)=\sum_{i=0}^{n} (i-1)\thet(i)=\sum_{i=0}^{n} i\thet(i)-\sum_{i=0}^{n}\thet(i)=\\
=\left(\frac{n}{k}-1\right) \sum_{i=0}^{n}\thet(i)=\frac{n-k}{k} \sum_{i=1}^{n-k+1}\thet(i)=\frac{n-k}{k} \sum_{i=0}^{n-k}\eta(i).
\end{multline*}
\item[(ii)] Recall that $\thet(2)\geq \thet(3)\geq \dots\geq \thet(n-k+1)$. This directly implies $\eta(1)\geq \eta(2)\geq \dots\geq \eta(n-k)$.
\item[(iii)] Assume that $n-k\geq k$, which means $n\geq 2k$. Then (\ref{eq-phi2}) implies
\[\eta(0)\geq (k-1)\eta(n-k)+(k-2)\eta(n-k-1)+\dots+\eta(n-2k+2).\]
\item[(iv)] Assume that $n-k\geq 2k$, which means $n\geq 3k$. Recall that we have $2\thet(\lfloor n/k\rfloor)\leq \thet(\lfloor n/k\rfloor-1)+\thet(\lceil n/k\rceil)$ in this case. Thus,
\[2\eta(\lfloor (n-k)/k\rfloor)=2\thet(\lfloor n/k\rfloor)\leq \thet(\lfloor n/k\rfloor-1)+\thet(\lceil n/k\rceil)=\eta(\lfloor (n-k)/k\rfloor-1)+\eta(\lceil (n-k)/k\rceil).\]
\end{itemize}
Hence, $\eta$ is indeed an $(n-k)$-tame scaled distribution. By the induction assumption for $n-k$, we can conclude that $\eta$ can be written as a non-negative linear combination of $(n-k)$-simple scaled distributions. So let
\[\eta=\sum_{(a_1,\dots,a_k)\in T_{n-k,k}}\lambda(a_1,\dots,a_k)\cdot (\indic_{a_1}+\dots+\indic_{a_k})\]
with $\lambda(a_1,\dots,a_k)\geq 0$ for all $(a_1,\dots,a_k)\in T_{n-k,k}$. Using that $\eta(i)=\thet(i+1)$ for $i=0,\dots,n-k$ as well as $\thet(0)=0$ and $\thet(n)=\thet(n-1)=\dots=\thet(n-k+2)=0$, we obtain
\[\thet=\sum_{(a_1,\dots,a_k)\in T_{n-k,k}}\lambda(a_1,\dots,a_k)\cdot (\indic_{a_1+1}+\dots+\indic_{a_k+1}).\]
Note that for every $(a_1,\dots,a_k)\in T_{n-k,k}$ we have $(a_1+1,\dots,a_k+1)\in T_{n,k}$. Thus, for every $(a_1,\dots,a_k)\in T_{n-k,k}$ the scaled distribution $\indic_{a_1+1}+\dots+\indic_{a_k+1}$ is $n$-simple. So the above equation establishes that $\thet$ is a non-negative linear combination of $n$-simple scaled distributions. This finishes the proof of Proposition \ref{propo-tame-main}.

\section{Proof of Lemmas \ref{lemma-alpha-j}, \ref{lemma-slack} and \ref{lemma-inequality}}
\label{sect-lemmas2}

\subsection{Proof of Lemma \ref{lemma-alpha-j}}

We will prove Lemma \ref{lemma-alpha-j} by constructing the scaled distributions $\alpha_j$ explicitly (distinguishing several cases).

\begin{proof}[Proof of Lemma \ref{lemma-alpha-j}] Recall that $n\geq 1$ and $j\in \lbrace 1,\dots,n\rbrace$ with $j+1\geq 2n/k$. First, let us check that $j\geq n/k$. If $n\geq k$, this follows from
\[j\geq \frac{2n}{k}-1\geq \frac{n}{k}+1-1=\frac{n}{k}.\]
If $n<k$, then $j\geq 1\geq n/k$. So we indeed have $j\geq n/k$ in either case.

Now let $n=\l j+r$ for integers $\l$ and $r$ with $0\leq r\leq j-1$ (so $r$ is the remainder of $n$ upon division by $j$). As $n/k\leq j\leq n$, we have $1\leq \l\leq k$.
If $\l=k$, then from $j\geq n/k$ we get that actually $j=n/k$ and we can just take
\[\alpha_j=\underbrace{\indic_{n/k}+\dots +\indic_{n/k}}_{k}=k\indic_{n/k}.\]
It is easy to check that this $\alpha_j$ satisfies the conditions (a) to (f) in Lemma \ref{lemma-alpha-j}.
So from now on let us assume that $1\leq \l\leq k-1$. We will distinguish several cases. In each case, we will give an explicit definition of $\alpha_j$ and then check the conditions (a) to (f) in Lemma \ref{lemma-alpha-j}.

\textbf{Case 1: $r\neq \lceil n/k\rceil$ or $n<2k$.} In this case, set
\[\alpha_j=\sum_{i=r}^{j} (\underbrace{\indic_j+\dots+\indic_j}_{\l-1}+\indic_{i}+\indic_{j+r-i}+ \underbrace{\indic_0+\dots+\indic_0}_{k-\l-1}).\]
Evaluating this gives
\[\alpha_j(i)=\begin{cases}
(k-\l-1)(j+1-r)&\text{if }i=0\\
0&\text{if }1\leq i\leq r-1\\
2&\text{if }r\leq i\leq j-1\\
(\l-1)(j+1-r)+2&\text{if }i=j\\
0&\text{if }j+1\leq i
\end{cases}\]
if $r\geq 1$ and
\[\alpha_j(i)=\begin{cases}
(k-\l-1)(j+1-r)+2&\text{if }i=0\\
2&\text{if }1\leq i\leq j-1\\
(\l-1)(j+1-r)+2&\text{if }i=j\\
0&\text{if }j+1\leq i
\end{cases}\]
if $r= 0$.
By definition, $\alpha_j$ satisfies (a) and therefore automatically also (b). By looking at the evaluations above, it is easy to see that $\alpha_j$ satisfies (c), (d) and (e). It remains to check (f). If $n<2k$, then (f) is vacuously true, so let us assume $n\geq 2k$. Then
\[j\geq \frac{2n}{k}-1\geq \frac{n}{k}+2-1= \frac{n}{k}+1>\lceil n/k\rceil.\]
Thus, $1\leq \lfloor n/k\rfloor-1\leq \lfloor n/k\rfloor \leq \lceil n/k\rceil\leq j-1$. Hence, each $i\in \lbrace \lfloor n/k\rfloor-1, \lfloor n/k\rfloor, \lceil n/k\rceil\rbrace$ satisfies $\alpha_j(i)\in \lbrace 0,2\rbrace$, and $\alpha_j(i)=2$ if and only if $i\geq r$. We need to show
\begin{equation}\label{eq-f}
2\alpha_j(\lfloor n/k\rfloor)\geq \alpha_j(\lfloor n/k\rfloor-1)+\alpha_j(\lceil n/k\rceil)
\end{equation}

Recall that we assumed $r\neq \lceil n/k\rceil$ for this case. If $r>\lceil n/k\rceil$, then each of the values  $\alpha_j(\lfloor n/k\rfloor)$, $ \alpha_j(\lfloor n/k\rfloor-1)$ and $\alpha_j(\lceil n/k\rceil)$ is zero, hence (\ref{eq-f}) is satisfied. If $r<\lceil n/k\rceil$, then $r\leq \lfloor n/k\rfloor$ and so $\alpha_j(\lfloor n/k\rfloor)=2$. Thus
\[2\alpha_j(\lfloor n/k\rfloor)=4=2+2\geq \alpha_j(\lfloor n/k\rfloor-1)+\alpha_j(\lceil n/k\rceil)\]
and (\ref{eq-f}) is satisfied as well.

\textbf{Case 2: $r= \lceil n/k\rceil$, $n\geq2k$, and $\l\leq k-2$.} In this case, set
\[\alpha_j=\sum_{i=r+1}^{j-1} (\underbrace{\indic_j+\dots+\indic_j}_{\l-1}+\indic_{i}+\indic_{j+r-i}+ \underbrace{\indic_0+\dots+\indic_0}_{k-\l-1})
+\frac{2}{r+1}\sum_{i=0}^{r} (\underbrace{\indic_j+\dots+\indic_j}_{\l}+\indic_{i}+\indic_{r-i}+ \underbrace{\indic_0+\dots+\indic_0}_{k-\l-2}). \]
Evaluating this gives
\[\alpha_j(i)=\begin{cases}
(k-\l-1)(j-1-r)+2(k-\l-2)+\frac{4}{r+1}&\text{if }i=0\\
\frac{4}{r+1}&\text{if }1\leq i\leq r\\
2&\text{if }r+1\leq i\leq j-1\\
(\l-1)(j-1-r)+2\l&\text{if }i=j\\
0&\text{if }j+1\leq i.
\end{cases}\]
Recall that $j-1-r\geq 0$ and $\l\geq 1$. Also, recall that $r=\lceil n/k\rceil\geq 2$. By definition, $\alpha_j$ satisfies (a) and therefore also (b). By looking at the evaluation above, it is easy to see that $\alpha_j$ satisfies (c), (d) and (e). In order to check (f), note that $1\leq \lfloor n/k\rfloor-1\leq \lfloor n/k\rfloor \leq \lceil n/k\rceil=r$. Hence we have $\alpha_j(\lfloor n/k\rfloor)=\alpha_j(\lfloor n/k\rfloor-1)=\alpha_j(\lceil n/k\rceil)=\frac{4}{r+1}$ and so (f) is satisfied as well.

\textbf{Case 3: $r= \lceil n/k\rceil$, $n\geq2k$, and $\l= k-1$.} In this case, we have $(k-1)j+r=n$, thus $kj=n+(j-r)$. Hence
\[n+(j-r)=kj\geq k\left(\frac{2n}{k}-1\right)\geq k\left(\frac{n}{k}+2-1\right)=k\left(\frac{n}{k}+1\right)=n+k.\]
So $j-r\geq k\geq 3$ and therefore $j-1-r\geq 2$.
Now set
\[\alpha_j=\sum_{i=r+1}^{j-1} (\underbrace{\indic_j+\dots+\indic_j}_{k-2}+\indic_{i}+\indic_{j+r-i})\]
Evaluating this gives
\[\alpha_j(i)=\begin{cases}
0&\text{if }0\leq i\leq r\\
2&\text{if }r+1\leq i\leq j-1\\
(k-2)(j-1-r)&\text{if }i=j\\
0&\text{if }j+1\leq i.
\end{cases}\]
Recall that $j-1-r\geq 2$ and $k\geq 3$. Hence $(k-2)(j-1-r)\geq 2$. Also recall that $r= \lceil n/k\rceil$.
By definition, $\alpha_j$ satisfies (a) and therefore automatically also (b). By looking at the evaluation above, it is easy to see that $\alpha_j$ satisfies (c), (d) and (e). In order to check (f), note that $\lfloor n/k\rfloor-1\leq \lfloor n/k\rfloor \leq \lceil n/k\rceil=r$. Hence we have $\alpha_j(\lfloor n/k\rfloor)=\alpha_j(\lfloor n/k\rfloor-1)=\alpha_j(\lceil n/k\rceil)=0$ and so (f) is satisfied as well.
This finishes the proof of Lemma \ref{lemma-alpha-j}.\end{proof}

\subsection{Proof of Lemma \ref{lemma-slack}}

Here, we will derive Lemma \ref{lemma-slack} from Lemma \ref{lemma-alpha-j}. However, we first need another lemma:

\begin{lemma}\label{lemma-j-value} Let $n\geq 1$ and let $\psi$ be an $n$-tame scaled distribution on $\lbrace 0,\dots,n\rbrace$ that is not identically zero. Let $j\in \lbrace 0,\dots,n\rbrace$ be chosen maximal with $\psi(j)>0$. Then $j+1\geq 2n/k$.
\end{lemma}
\begin{proof}By property (ii) in Definition \ref{defi-tame} we have $\psi(i)\geq  \psi(j+1-i)$ for $1\leq i\leq (j+1)/2$. Hence,
\begin{multline*}
i\psi(i)+(j+1-i)\psi(j+1-i)=\frac{j+1}{2}(\psi(i)+\psi(j+1-i))-\left(\frac{j+1}{2}-i\right)(\psi(i)-\psi(j+1-i))\\
\leq \frac{j+1}{2}(\psi(i)+\psi(j+1-i))
\end{multline*}
for $1\leq i\leq (j+1)/2$. Thus, by symmetry between $i$ and $j+1-i$ we actually obtain
\begin{equation}\label{eq-psi1}
i\psi(i)+(j+1-i)\psi(j+1-i)\leq \frac{j+1}{2}(\psi(i)+\psi(j+1-i))
\end{equation}
for $1\leq i\leq j$.
Note that by the choice of $j$ we have $\psi(i)=0$ for all $i>j$. So using (\ref{eq-psi1}) and property (i) in Definition \ref{defi-tame}, we obtain
\begin{multline*}
\frac{2n}{k}\sum_{i=0}^{n} \psi(i)=2\sum_{i=0}^{n} i\psi(i)=2\sum_{i=1}^{j} i\psi(i)=\sum_{i=1}^{j} i\psi(i)+\sum_{i=1}^{j} (j+1-i)\psi(j+1-i)\\
\leq \sum_{i=1}^{j} \frac{j+1}{2}(\psi(i)+\psi(j+1-i))=\frac{j+1}{2}\sum_{i=1}^{j} \psi(i)+\frac{j+1}{2}\sum_{i=1}^{j} \psi(j+1-i)=(j+1)\sum_{i=1}^{j} \psi(i)\leq (j+1)\sum_{i=0}^{n} \psi(i).
\end{multline*}
As $\sum_{i=0}^{n} \psi(i)>0$, this yields $j+1\geq 2n/k$ as desired.
\end{proof}

Now we are ready for the proof of Lemma \ref{lemma-slack}.

\begin{proof}[Proof of Lemma \ref{lemma-slack}]Recall that $n\geq k$. Suppose there is a counterexample to Lemma \ref{lemma-slack}. That means, one can find an $n$-tame scaled distribution $\psi$ for which there exists no $n$-tame scaled distribution $\phi$ with the desired conditions. Clearly, such a $\psi$ is not identically zero (because then we could just take $\phi$ to be identically zero as well). So for any such $\psi$, let $j\in \lbrace 0,\dots,n\rbrace$ be chosen maximal with $\psi(j)>0$. Among all possible scaled distributions $\psi$ that contradict Lemma \ref{lemma-slack}, let us choose one where this $j$ is minimal. Thus, $\psi$ is an $n$-tame scaled distribution on $\lbrace 0,\dots,n\rbrace$, but there exists no $\phi$ with the desired properties. And $j\in \lbrace 0,\dots,n\rbrace$ is maximal with $\psi(j)>0$. This means $\psi(j+1)=\dots=\psi(n)=0$, but $\psi(j)>0$. Now, by the choice of $\psi$, for any $n$-tame scaled distribution $\psi'$ with $\psi'(j)=\psi'(j+1)=\dots=\psi'(n)=0$ we can find an $n$-tame scaled distribution $\phi'$ satisfying (\ref{eq-phi-equal}) such that $\psi'-\phi'$ is a non-negative linear combination of $n$-simple scaled distributions.

Note that by Lemma \ref{lemma-j-value} we have $j+1\geq 2n/k$ and in particular $j\geq (2n/k)-1\geq 1$. Thus, Lemma \ref{lemma-alpha-j} gives a scaled distribution $\alpha_j$ satisfying conditions (a) to (f). For each $x\in \mathbb{R}_{\geq 0}$ define a map $\psi_x: \lbrace 0,\dots,n\rbrace\to \mathbb{R}$ by
\[\psi_x=\psi-x\alpha_j.\]
Now, choose the maximum $x\in \mathbb{R}_{\geq 0}$ such that
\[\psi_x(j)=\psi(j)-x\alpha_j(j)\geq 0\]
and
\[\psi_x(0)\geq (k-1)\psi_x(n)+(k-2)\psi_x(n-1)+\dots+\psi_x(n-k+2).\]
Such a maximum $x$ exists since $x=0$ satisfies the conditions (for the second condition see property (iii) of $\psi$), but due to $\alpha_j(j)>0$ by (e) there is an upper bound for $x$. Note that for this maximum $x$ we have $\psi_x(j)=0$ or $\psi_x(0)= (k-1)\psi_x(n)+(k-2)\psi_x(n-1)+\dots+\psi_x(n-k+2)$.

We claim that $\psi_x$ is an $n$-tame scaled distribution. Recall that $\psi_x(j)\geq 0$. By (ii) for $\psi$ and (c) we have
$\psi(1)\geq \psi(2)\geq \dots\geq \psi(j)$ and $\alpha_j(1)\leq \alpha_j(2)\leq \dots\leq \alpha_j(j)$, hence
\[\psi_x(1)\geq \psi_x(2)\geq \dots\geq \psi_x(j)\geq 0.\]
From $\psi(j+1)=\dots=\psi(n)=0$ and (d) we can deduce that
\begin{equation}\label{eq-phi-x}
\psi_x(j+1)=\dots=\psi_x(n)=0.
\end{equation}
In particular, $\psi_x(i)\geq 0$ for $1\leq i\leq n$. Now recall that $n\geq k$ and 
\[\psi_x(0)\geq (k-1)\psi_x(n)+(k-2)\psi_x(n-1)+\dots+\psi_x(n-k+2),\]
hence $\psi_x(0)\geq 0$ as well. So we have established that $\psi_x$ has non-negative values, so it is a scaled distribution. We also have
\[\psi_x(1)\geq \psi_x(2)\geq \dots\geq \psi_x(j)\geq 0=\psi_x(j+1)=\dots=\psi_x(n),\]
so $\psi_x$ satisfies (ii). Furthermore $\psi_x=\psi-x\alpha_j$ has mean $n/k$ because both $\psi$ and $\alpha_j$ have mean $n/k$ (see (i) for $\psi$ and (b)). Thus, $\psi_x$ also satisfies (i). Note that (iii) is satisfied since we chose $x$ such that
\[\psi_x(0)\geq (k-1)\psi_x(n)+(k-2)\psi_x(n-1)+\dots+\psi_x(n-k+2).\]
It remains to check (iv). If $n<2k$, then (iv) is vacuous, so assume $n\geq 2k$. Then by (iv) for $\psi$ and (f) we have $2\psi(\lfloor n/k\rfloor)\leq \psi(\lfloor n/k\rfloor-1)+\psi(\lceil n/k\rceil)$ and $2\alpha_j(\lfloor n/k\rfloor)\geq \alpha_j(\lfloor n/k\rfloor-1)+\alpha_j(\lceil n/k\rceil)$, hence
\[2\psi_x(\lfloor n/k\rfloor)\leq \psi_x(\lfloor n/k\rfloor-1)+\psi_x(\lceil n/k\rceil).\]

So $\psi_x$ is indeed an $n$-tame scaled distribution. Recall that by the choice of $x$ we have $\psi_x(j)=0$ or $\psi_x(0)= (k-1)\psi_x(n)+(k-2)\psi_x(n-1)+\dots+\psi_x(n-k+2)$. Suppose the latter, then $\psi_x$ would be an $n$-tame scaled distribution satisfying (\ref{eq-phi-equal}) and $\psi-\psi_x=x\alpha_j$ would be a non-negative linear combination of $n$-tame scaled distributions (see (a)). This contradicts our assumption of $\psi$ being a counterexample to Lemma \ref{lemma-slack}. 

Hence we must have $\psi_x(j)=0$. Thus, together with (\ref{eq-phi-x}) we obtain $\psi_x(j)=\psi_x(j+1)=\dots=\psi_x(n)=0$. So $\psi_x$ is an $n$-tame scaled distribution with $\psi_x(j)=\psi_x(j+1)=\dots=\psi_x(n)=0$. We saw above that by the choice of $\psi$ this implies that for $\psi_x$ we can find an $n$-tame scaled distribution $\phi$ satisfying (\ref{eq-phi-equal}) such that $\psi_x-\phi$ is a non-negative linear combination of $n$-simple scaled distributions. But then, using (a), we obtain that $\psi-\phi=(\psi_x-\phi)+x\alpha_j$ is also a non-negative linear combination of $n$-simple scaled distributions. This contradicts our choice of $\psi$. Hence there cannot be any counterexamples to  Lemma \ref{lemma-slack}, so Lemma \ref{lemma-slack} is true.
\end{proof}

\subsection{Proof of Lemma \ref{lemma-inequality}}
\label{sect-inequality}

The goal of this subsection is to prove Lemma \ref{lemma-inequality}. So assume $n\geq 2k$ and let $\phi$ be an $n$-tame scaled distribution on $\lbrace 0,\dots,n\rbrace$ satisfying
\begin{equation}\label{phi-eq-ineq}
\phi(0)= (k-1)\phi(n)+(k-2)\phi(n-1)+\dots+\phi(n-k+2).
\end{equation}
Let $s=\lfloor n/k\rfloor$ and $r=\frac{n}{k}-s$, so $0\leq r<1$ and $n=k(r+s)$. Also note that $s\geq 2$.

Set
\[\lambda(0)=\phi(0)-\sum_{i=n-(k-1)+1}^{n}(i-n+(k-1))\phi(i)
=\phi(0) - \phi(n-k+2)-2\phi(n-k+3)-\dots-(k-1)\phi(n)\]
and note that by (\ref{phi-eq-ineq}) we have $\lambda(0)=0$. Furthermore, for each $1\leq \l\leq s-1$, set
\[\lambda(\l)=\phi(\l)-\sum_{i=n-(\l+1)(k-1)+1}^{n-\l(k-1)}(i-n+(\l+1)(k-1))\phi(i)
-\sum_{i=n-\l(k-1)+1}^{n-(\l-1)(k-1)}(n-(\l-1)(k-1)-i)\phi(i).\]
For each index $i$ occurring in the sums we have $i\geq n-(\l+1)(k-1)+1\geq n-s(k-1)+1\geq s+1$. Note that we can rewrite $\lambda(\l)$ for $1\leq \l\leq s-1$ as
\[\lambda(\l)=\phi(\l)-\sum_{j=1}^{k-1}j\cdot \phi(n-(\l+1)(k-1)+j)
-\sum_{j=1}^{k-1}(k-1-j)\cdot \phi(n-\l(k-1)+j).\]
Recalling that $\phi(1)\geq \dots\geq \phi(n)$ by property (ii) in Definition \ref{defi-tame}, we obtain
\[\lambda(1)\geq \dots\geq \lambda(s-1).\]
The claim of Lemma \ref{lemma-inequality} is equivalent to $\lambda(1)\geq 0$. Let us assume for contradiction that the claim is false, that is $\lambda(1)< 0$. Then $0> \lambda(1)\geq \dots\geq \lambda(s-1)$.
Let us consider the term
\begin{equation}\label{eq-term1}
r\phi(s)+\sum_{\l=0}^{s-1}\left(\frac{n}{k}-\l \right)\lambda(\l).
\end{equation}
We now plug in the definition of $\lambda(\l)$ for $\l=0,\dots,s-1$. For every $1\leq \l\leq s-1$ and every $i$ with $n-\l(k-1)+1\leq i\leq n-(\l-1)(k-1)$, the coefficient of $\phi(i)$ in $\lambda(\l)$ is $-(n-(\l-1)(k-1)-i)$ and its coefficient in $\lambda(\l-1)$ is $-(i-n+\l(k-1))$. Hence the total coefficient of $\phi(i)$ in (\ref{eq-term1}) is
\begin{multline*}
-\left(\frac{n}{k}-\l \right)(n-(\l-1)(k-1)-i)-\left(\frac{n}{k}-\l +1\right)(i-n+\l(k-1))\\
=-\left(\frac{n}{k}-\l \right)(k-1)-(i-n+\l(k-1))
=-\frac{n}{k}(k-1)+n-i=\frac{n}{k}-i=-\left(i-\frac{n}{k}\right).
\end{multline*}
Furthermore for all  $i$ with $n-s(k-1)+1\leq i\leq n-(s-1)(k-1)$, the coefficient of $\phi(i)$ in $\lambda(s-1)$ is $-(i-n+s(k-1))$, and so the total coefficient of $\phi(i)$ in (\ref{eq-term1}) is
\[-\left(\frac{n}{k}-s+1\right)(i-n+s(k-1))=-(r+1)(i-n+s(k-1)). \]
Hence, when plugging in the definition of $\lambda(\l)$ for $\l=0,\dots,s-1$ into the term (\ref{eq-term1}),  we obtain
\begin{multline*}
r\phi(s)+\sum_{\l=0}^{s-1}\left(\frac{n}{k}-\l \right)\lambda(\l)=\left(\frac{n}{k}-s \right)\phi(s)+\sum_{\l=0}^{s-1}\left(\frac{n}{k}-\l \right)\lambda(\l)\\
=\sum_{\l=0}^{s}\left(\frac{n}{k}-\l \right)\phi(\l)-\sum_{i=n-s(k-1)+1}^{n-(s-1)(k-1)}(r+1)(i-n+s(k-1))\phi(i)-\sum_{i=n-(s-1)(k-1)+1}^{n}\left(i-\frac{n}{k}\right)\phi(i)
\end{multline*}
Since $\phi$ has mean $n/k$ by property (i) in Definition \ref{defi-tame}, we have $\sum_{i=0}^{n} i\phi(i)=\frac{n}{k} \sum_{i=0}^{n}\phi(i)$ and therefore
\[\sum_{\l=0}^{s}\left(\frac{n}{k}-\l\right)\phi(\l)=\sum_{i=s+1}^{n}\left(i-\frac{n}{k}\right)\phi(i).\]
Thus, recalling $n-s(k-1)+1\geq s+1$, we obtain
\begin{align*}
r\phi(s)&+\sum_{\l=0}^{s-1}\left(\frac{n}{k}-\l \right)\lambda(\l)\\
&=\sum_{i=s+1}^{n}\left(i-\frac{n}{k}\right)\phi(i)-\sum_{i=n-s(k-1)+1}^{n-(s-1)(k-1)}(r+1)(i-n+s(k-1))\phi(i)-\sum_{i=n-(s-1)(k-1)+1}^{n}\left(i-\frac{n}{k}\right)\phi(i)\\
&=\sum_{i=s+1}^{n-s(k-1)}\left(i-\frac{n}{k}\right)\phi(i)+\sum_{i=n-s(k-1)+1}^{n-(s-1)(k-1)}\left(i-\frac{n}{k}-(r+1)(i-n+s(k-1))\right)\phi(i).
\end{align*}
Note that
\begin{multline*}
i-\frac{n}{k}-(r+1)(i-n+s(k-1)) =i-\frac{n}{k}-i+n-s(k-1)-r(i-n+s(k-1))=\frac{n}{k}(k-1)-s(k-1)-r(i-n+s(k-1))\\
=r(k-1)-r(i-n+s(k-1))=r(n-(s-1)(k-1)-i).
\end{multline*}
Hence
\[r\phi(s)+\sum_{\l=0}^{s-1}\left(\frac{n}{k}-\l \right)\lambda(\l)
=\sum_{i=s+1}^{n-s(k-1)}\left(i-\frac{n}{k}\right)\phi(i)+\sum_{i=n-s(k-1)+1}^{n-(s-1)(k-1)}r(n-(s-1)(k-1)-i)\phi(i).\]
Recall that $\lambda(0)=0$ and $0>\lambda(1)\geq \dots\geq  \lambda(s-1)$. Noting that the coefficient of each $\lambda(\l)$ on the left-hand side of the last equation is strictly positive, this implies (recalling $s\geq 2$)
\begin{equation}\label{eq-term2}
r\phi(s)>\sum_{i=s+1}^{n-s(k-1)}\left(i-\frac{n}{k}\right)\phi(i)+\sum_{i=n-s(k-1)+1}^{n-(s-1)(k-1)}r(n-(s-1)(k-1)-i)\phi(i).
\end{equation}
As $s+1>\frac{n}{k}$, all terms on the right-hand side are non-negative. Hence the left-hand side must be positive. In particular, we must have $r>0$.
So $s<\frac{n}{k}$ and therefore $n-s(k-1)\geq s+1$. Furthermore, we obtain $\lceil n/k\rceil = \lfloor n/k \rfloor +1=s+1$. Thus, property (iv) in Definition \ref{defi-tame} gives $2\phi(s)\leq \phi(s-1)+\phi(s+1)$. Hence
\[r\phi(s)=r^{2}\phi(s)+\frac{1}{2}r(1-r)\cdot 2\phi(s)\leq r^{2}\phi(s)+\frac{1}{2}r(1-r)\cdot(\phi(s-1)+\phi(s+1)).\]
Together with (\ref{eq-term2}), this gives
\[r^{2}\phi(s)+\frac{1}{2}r(1-r)\phi(s-1)+\frac{1}{2}r(1-r)\phi(s+1)>\sum_{i=s+1}^{n-s(k-1)}\left(i-\frac{n}{k}\right)\phi(i)+\sum_{i=n-s(k-1)+1}^{n-(s-1)(k-1)}r(n-(s-1)(k-1)-i)\phi(i).\]
Note that the coefficient of $\phi(s+1)$ on the right-hand side is $s+1-\frac{n}{k}=1-r$. So subtracting $\frac{1}{2}r(1-r)\phi(s+1)$ from both sides gives
\begin{multline*}
r^{2}\phi(s)+\frac{1}{2}r(1-r)\phi(s-1)
>(1-r)\left(1-\frac{1}{2}r\right)\phi(s+1)+\sum_{i=s+2}^{n-s(k-1)}\left(i-\frac{n}{k}\right)\phi(i)\\
+\sum_{j=1}^{k-1}r(k-1-j)\phi(n-s(k-1)+j).
\end{multline*}
Using again $\phi(1)\geq  \dots\geq\phi(n)$ and $s\geq 2$ as well as $n-s(k-1)\geq s+1$, this yields
\begin{multline}\label{eq-ineq-term3}
r^{2}\phi(1)+\frac{1}{2}(r-r^{2})\phi(1)
>\frac{1}{2}(2-3r+r^{2})\phi(n-2(k-1))+\sum_{i=s+2}^{n-s(k-1)}\left(i-\frac{n}{k}\right)\phi(n-2(k-1))\\
+\sum_{j=1}^{k-2}r(k-1-j)\phi(n-2(k-1)+j).
\end{multline}
Note that
\[\sum_{i=s+2}^{n-s(k-1)}\left(i-\frac{n}{k}\right)=\sum_{i=s+2}^{n-s(k-1)}\left(i-s-1+(1-r)\right)=(n-s(k-1)-s-1)(1-r)+\sum_{j=1}^{n-s(k-1)-s-1}j.\]
Using $n-s(k-1)-s-1=n-sk-1=kr-1$, this yields
\[\sum_{i=s+2}^{n-s(k-1)}\left(i-\frac{n}{k}\right)=(kr-1)(1-r)+\frac{(kr-1)kr}{2}=(k+1)r-kr^2-1+\frac{k^2r^2-kr}{2}=\frac{1}{2}(-2+(k+2)r+(k^2-2k)r^2).\]
Plugging this into (\ref{eq-ineq-term3}) and simplifying, we obtain
\[\frac{r}{2}(1+r)\phi(1)
>\frac{1}{2}((k-1)r+(k-1)^2r^2)\phi(n-2(k-1))+\sum_{j=1}^{k-2}r(k-1-j)\phi(n-2(k-1)+j).\]
As $\phi(n-2(k-1)+1)\geq \dots\geq \phi(n-2(k-1)+(k-2))$ and $r(k-2)\geq \dots\geq r\cdot 1$, we have by Chebyshev's sum inequality (or alternatively by the rearrangement inequality)
\[\sum_{j=1}^{k-2}r(k-1-j)\phi(n-2(k-1)+j)\geq \sum_{j=1}^{k-2}\frac{r(k-2)+ \dots + r\cdot 1}{k-2} \phi(n-2(k-1)+j)=\sum_{j=1}^{k-2}\frac{r(k-1)}{2}\phi(n-2(k-1)+j).\]
Thus, using $\phi(1)\geq \dots\geq \phi(n)$ again,
\begin{multline*}
\frac{r}{2}(1+r)\phi(1)>\frac{r}{2}((k-1)+(k-1)^2r)\phi(n-2(k-1))+\sum_{j=1}^{k-2}\frac{r}{2}(k-1)\phi(n-2(k-1)+j)\\
\geq \frac{r}{2}(1+r)(k-1)\phi(n-(k-1))+\frac{r^2}{2}(k-1)(k-2)\phi(n-2(k-1))+\sum_{j=1}^{k-2}\frac{r}{2}(k-1)\phi(n-2(k-1)+j)\\
\geq \frac{r}{2}(1+r)(k-1)\phi(n-(k-1))+\sum_{j=1}^{k-2}\left(\frac{r}{2}+\frac{r^2}{2}\right)(k-1)\phi(n-2(k-1)+j).
\end{multline*}
Dividing by $\frac{r}{2}(1+r)=\frac{r}{2}+\frac{r^2}{2}$ yields
\begin{multline*}
\phi(1)>(k-1)\phi(n-(k-1))+\sum_{j=1}^{k-2}(k-1)\phi(n-2(k-1)+j)\\
\geq (k-1)\phi(n-(k-1))+\sum_{j=1}^{k-2}[j\cdot \phi(n-2(k-1)+j)+(k-1-j)\cdot \phi(n-(k-1)+j)].
\end{multline*}
In other words,
\begin{multline*}
\phi(1)>\phi(n-1)+2\phi(n-2)+\dots+(k-2)\phi(n-k+2)\\
+(k-1)\phi(n-k+1)+(k-2)\phi(n-k)+\dots+\phi(n-2k+3),
\end{multline*}
so the claim of Lemma \ref{lemma-inequality} is true. This is a contradiction to our assumption (recall that we assumed that the lemma is false), which finishes the proof of the lemma.

\section{Upper Bound}
\label{sect-upperbound}

Here, we give a proof of Theorem \ref{thmupperbound}. The proof is very similar to that of Theorem 4 in \cite{NASLUND18}, which in turn was inspired by the proof of Theorem 4.14 in \cite{BCCGNSU16}. We repeat these arguments here for the reader's convenience.

Let $m=p^\l$ for a prime number $p$ and an integer $\l\geq 1$. For every integer $0\leq a\leq m-1$, we have
\[\binom{z}{a}\equiv \binom{z'}{a}\pmod{p}\]
if $z$ and $z'$ are non-negative integers with $z\equiv z'\pmod{m}$. This can be derived from Lucas' theorem. Hence there is a well-defined map $\mathbb{Z}_m\to \Fp$ given by $z\mapsto \binom{z}{a}$.

\begin{lemma}[Lemma 9 in \cite{NASLUND18}]\label{lemma-naslund} Let $p$ be a prime and $m=p^\l$ be a prime power, and let $z_1,\dots,z_k\in \mathbb{Z}_m$. Then over $\Fp$ we have the identity
	\[\sum_{\substack{a_1,\dots,a_k \in \lbrace 0,\dots, m-1\rbrace\\a_1+\dots+a_k\leq m-1}}(-1)^{a_1+\dots+a_k}\binom{z_1}{a_1}\dots \binom{z_k}{a_k}=\begin{cases}
	1&\text{if }z_1+\dots+z_k=0\text{ in }\mathbb{Z}_m\\
	0&\text{otherwise.}
	\end{cases}\]
\end{lemma}
\begin{proof}Note that for every integer $0\leq z\leq m-1$ we have
	\[\sum_{0\leq a\leq m-1}(-1)^a\binom{z}{a}=(1-1)^z=\begin{cases}
	1&\text{if }z=0\\
	0&\text{if }1\leq z\leq m-1.
	\end{cases}\]
	Hence for every $z\in \mathbb{Z}_m$ we have
	\[\sum_{0\leq a\leq m-1}(-1)^a\binom{z}{a}=\begin{cases}
	1&\text{if }z=0\text{ in }\mathbb{Z}_m\\
	0&\text{if }z\neq 0\text{ in }\mathbb{Z}_m
	\end{cases}\]
	over $\Fp$. Furthermore note that for $0\leq a\leq m-1$ and for non-negative integers $z_1,\dots,z_k$ we have
	\[\binom{z_1+\dots+z_k}{a}=\sum_{\substack{a_1,\dots,a_k \in \lbrace 0,\dots, m-1\rbrace\\a_1+\dots+a_k=a}}\binom{z_1}{a_1}\dots \binom{z_k}{a_k},\]
	hence the same identity holds over $\Fp$ for $z_1,\dots,z_k\in \mathbb{Z}_m$. Now, for all $z_1,\dots,z_k\in \mathbb{Z}_m$ we obtain
	\begin{multline*}
	\sum_{\substack{a_1,\dots,a_k \in \lbrace 0,\dots, m-1\rbrace\\a_1+\dots+a_k\leq m-1}}(-1)^{a_1+\dots+a_k}\binom{z_1}{a_1}\dots \binom{z_k}{a_k}=\sum_{0\leq a\leq m-1}(-1)^a\binom{z_1+\dots+z_k}{a}\\
	=\begin{cases}
	1&\text{if }z_1+\dots+z_k=0\text{ in }\mathbb{Z}_m\\
	0&\text{otherwise}
	\end{cases}
	\end{multline*}
	over $\Fp$.\end{proof}

Let us now prove Theorem \ref{thmupperbound} using Tao's slice rank method \cite{Taoblog16}.

\begin{proof}[Proof of Theorem \ref{thmupperbound}] Let $(x_{1,j}, x_{2,j}, \dots, x_{k,j})_{j=1}^L$ be a $k$-colored sum-free set in $\Zmn$. We need to prove that $L\leq (\Gamma_{m,k})^n$. We will first prove that $L\leq k\cdot (\Gamma_{m,k})^n$, and the additional factor $k$ will then be removed using a power trick.
	
	Let us define a tensor $G:\lbrace 1,\dots,L\rbrace^k\to \Fp$ by setting
	\begin{equation}\label{eq-def-tensor}
	G(j_1,\dots, j_k)=\prod_{i=1}^{n}\left(\sum_{\substack{a_1,\dots,a_k \in \lbrace 0,\dots, m-1\rbrace\\a_1+\dots+a_k\leq m-1}}(-1)^{a_1+\dots+a_k}\binom{x_{1,j_1}^{(i)}}{a_1}\dots \binom{x_{k,j_k}^{(i)}}{a_k}\right)
	\end{equation}
	for all $j_1,\dots,j_k\in \lbrace 1,\dots,L\rbrace$. It follows from Lemma \ref{lemma-naslund} that for each $i=1,\dots,n$ the sum on the right-hand side is $1$ if and only if $x_{1,j_1}^{(i)}+\dots+x_{k,j_k}^{(i)}=0$ and zero otherwise. Thus, we obtain
	\[G(j_1,\dots, j_k)=\begin{cases}
	1&\text{if }x_{1,j_1}+\dots+x_{k,j_k}=0\\
	0&\text{otherwise}
	\end{cases}\]
	for all $j_1,\dots,j_k\in \lbrace 1,\dots,L\rbrace$. Hence $G$ is a diagonal tensor and by Tao's slice rank Lemma \cite[Lemma 1]{Taoblog16} the tensor $G$ has slice rank $L$.
	
	On the other hand, by multiplying (\ref{eq-def-tensor}) out, we can write $G(j_1,\dots, j_k)$ as a linear combination of terms of the form
	\[\left(\binom{x_{1,j_1}^{(1)}}{a_{1,1}}\binom{x_{1,j_1}^{(2)}}{a_{1,2}}\dots \binom{x_{1,j_1}^{(n)}}{a_{1,n}}\right)\dots \left(\binom{x_{k,j_k}^{(1)}}{a_{k,1}}\binom{x_{k,j_k}^{(2)}}{a_{k,2}}\dots \binom{x_{k,j_k}^{(n)}}{a_{k,n}}\right)\]
	with $a_{1,i}+\dots+a_{k,i}\leq m-1$ for each $i=1,\dots,n$. Thus, $\sum_{s=1}^{k} (a_{s,1}+\dots+a_{s,n})\leq n(m-1)$, so for each of these terms in the linear combination we can choose some $s\in \lbrace 1,\dots,k \rbrace$ with $a_{s,1}+\dots+a_{s,n}\leq n(m-1)/k$. Let us now sort the terms into groups depending on the chosen index $s\in \lbrace 1,\dots,k \rbrace$ and on the $n$-tuple $(a_{s,1},\dots,a_{s,n})$. Then each group gives a term of the form 
	\[\left(\binom{x_{s,j_s}^{(1)}}{a_{s,1}}\binom{x_{s,j_s}^{(2)}}{a_{s,2}}\dots \binom{x_{s,j_s}^{(n)}}{a_{s,n}}\right)\cdot G'(j_1,\dots,j_{s-1},j_{s+1},\dots,j_n)\]
	for some function $G'$ depending only on $j_1,\dots,j_{s-1},j_{s+1},\dots,j_n$ and not on $j_s$. In other words, we obtain a slice rank decomposition of $G$ with one slice for each group given by $s\in \lbrace 1,\dots,k \rbrace$ and an $n$-tuple $(a_{s,1},\dots,a_{s,n})$. Since $G$ has slice rank $L$, the number of groups must be at least $L$, so
	\[L\leq k\cdot\vert\lbrace (a_1,\dots,a_n) \in \lbrace 0,\dots, m-1\rbrace^n\mid a_1+\dots+a_n\leq n(m-1)/k\rbrace\vert.\]
	
	The following lemma gives an upper bound for the quantity on the right-hand side.
	\begin{lemma}\label{lemma-count-tuples-upperbound}
		$\vert\lbrace (a_1,\dots,a_n) \in \lbrace 0,\dots, m-1\rbrace^n\mid a_1+\dots+a_n\leq n(m-1)/k\rbrace\vert\leq (\Gamma_{m,k})^n$.
	\end{lemma}
	We postpone the proof of this lemma for a moment, in order to first finish the proof of \ref{thmupperbound}.
	Applying Lemma \ref{lemma-count-tuples-upperbound}, we obtain $L\leq k\cdot (\Gamma_{m,k})^n$ for every $k$-colored sum-free set $(x_{1,j}, x_{2,j}, \dots, x_{k,j})_{j=1}^L$ in $\Zmn$ (for all $n$). Note that if $(x_{1,j}, x_{2,j}, \dots, x_{k,j})_{j=1}^L$ is a $k$-colored sum-free set in $\Zmn$, then for every integer $\l\geq 1$, we can construct a $k$-colored sum-free set of size $L^\l$ in $\mathbb{Z}_m^{n\l}=\Zmn\times\dots\times\Zmn$ by taking the collection of $k$-tuples 
	\[\big((x_{1,j_1}, x_{1,j_2},\dots,x_{1,j_\l}), \dots, (x_{k,j_1}, x_{k,j_2},\dots,x_{k,j_\l})\big)_{(j_1,\dots,j_\l)\in \lbrace 1,\dots,L\rbrace^\l}.\]
	Thus, $L^\l\leq k\cdot  (\Gamma_{m,k})^{n\l}$, and we can conclude $L\leq (\Gamma_{m,k})^n$ by taking $\l\to\infty$.
\end{proof}

Lemma  \ref{lemma-count-tuples-upperbound} has a standard proof, it was given for example in \cite[Proposition 4.12]{BCCGNSU16}, see also \cite[Lemma 5]{NASLUND18}. For the reader's convenience, we repeat the proof here:

\begin{proof}[Proof of Lemma \ref{lemma-count-tuples-upperbound}]
	Let $Z_1,\dots,Z_n$ be independent random variables, uniformly distributed on $\lbrace 0,\dots, m-1\rbrace$. Then the desired number of $n$-tuples $(a_1,\dots,a_n) \in \lbrace 0,\dots, m-1\rbrace^n$ with $a_1+\dots+a_n\leq n(m-1)/k$ equals $m^n\P(Z_1+\dots+Z_n\leq n(m-1)/k)$. So we need to prove $\P\big(Z_1+\dots+Z_n\leq n(m-1)/k\big)\leq m^{-n}(\Gamma_{m,k})^n$.
	
	For every $0<\gamma<1$ we have by Markov's inequality
	\begin{multline*}
	\P\big(Z_1+\dots+Z_n\leq n(m-1)/k\big)=\P\big(\gamma^{Z_1+\dots+Z_n}\geq \gamma^{n(m-1)/k}\big)\leq \gamma^{-n(m-1)/k}\E\big(\gamma^{Z_1+\dots+Z_n}\big)\\
	=\gamma^{-n(m-1)/k}\left(\E\big(\gamma^{Z_1}\big)\right)^n=\gamma^{-n(m-1)/k}\left(\frac{1+\gamma+\dots+\gamma^{m-1}}{m}\right)^n=m^{-n}\left(\frac{1+\gamma+\dots+\gamma^{m-1}}{\gamma^{(m-1)/k}}\right)^n.
	\end{multline*}
	Taking $\gamma=\gamma_{m,k}$, this gives $\P\big(Z_1+\dots+Z_n\leq n(m-1)/k\big)\leq m^{-n}(\Gamma_{m,k})^n$, as desired.
\end{proof}

\textit{Acknowledgements.} We would like to thank Jacob Fox and Terence Tao for helpful discussions.

\bibliography{ref}

\providecommand{\bysame}{\leavevmode\hbox to3em{\hrulefill}\thinspace}
\providecommand{\MR}{\relax\ifhmode\unskip\space\fi MR }
\providecommand{\MRhref}[2]{%
  \href{http://www.ams.org/mathscinet-getitem?mr=#1}{#2}
}
\providecommand{\href}[2]{#2}
\begin{thebibliography}{10}

\bibitem{Alon01}
N.~Alon, \emph{Testing subgraphs in large graphs}, 42nd {IEEE} {S}ymposium on
  {F}oundations of {C}omputer {S}cience ({L}as {V}egas, {NV}, 2001), IEEE
  Computer Soc., Los Alamitos, CA, 2001, pp.~434--441.

\bibitem{ASU13}
N.~Alon, A.~Shpilka, and C.~Umans, \emph{On sunflowers and matrix
  multiplication}, Comput. Complexity \textbf{22} (2013), 219--243.

\bibitem{AlonSpencer}
N.~Alon and J.~Spencer, \emph{The probabilistic method}, 4th ed., Wiley, 2016.

\bibitem{Behrend46}
F.~A. Behrend, \emph{On sets of integers which contain no three terms in
  arithmetical progression}, Proc. Nat. Acad. Sci. U. S. A. \textbf{32} (1946),
  331--332.

\bibitem{BENNETT18}
M.~Bennett, \emph{Bounds on sizes of caps in {$AG(n,q)$} via the croot-lev-pach
  polynomial method}, preprint at arXiv:1806.05303, 2018.

\bibitem{BGRS12}
A.~Bhattacharyya, E.~Grigorescu, P.~Raghavendra, and A.~Shapira, \emph{Testing
  odd-cycle-freeness in {B}oolean functions}, Combin. Probab. Comput.
  \textbf{21} (2012), 835--855.

\bibitem{BX09}
A.~Bhattacharyya and N.~Xie, \emph{Lower bounds for testing triangle-freeness
  in {B}oolean functions}, Comput. Complexity \textbf{24} (2015), 65--101, A
  preliminary version appeared in SODA 2010, pp. 87 -- 98.

\bibitem{BCCGNSU16}
J.~{Blasiak}, T.~{Church}, H.~{Cohn}, J.~A. {Grochow}, E.~Naslund, W.~F. Sawin,
  and C.~{Umans}, \emph{{On cap sets and the group-theoretic approach to matrix
  multiplication}}, Discrete Analysis \textbf{2017:3}, 27pp.

\bibitem{CW90}
D.~Coppersmith and S.~Winograd, \emph{Matrix multiplication via arithmetic
  progressions}, J. Symbolic Comput. \textbf{9} (1990), 251--280.

\bibitem{CLP16}
E.~Croot, V.~F. Lev, and P.~P. Pach, \emph{Progression-free sets in {$\mathbb
  Z^n_4$} are exponentially small}, Ann. of Math. (2) \textbf{185} (2017),
  331--337.

\bibitem{DM18}
Z.~Dvir and S.~Moran, \emph{A {S}auer-{S}helah-{P}erles lemma for sumsets},
  preprint at arXiv:1806.05737, 2018.

\bibitem{EG16}
J.~S. Ellenberg and D.~Gijswijt, \emph{On large subsets of {$\mathbb F^n_q$}
  with no three-term arithmetic progression}, Ann. of Math. (2) \textbf{185}
  (2017), 339--343.

\bibitem{ELLENBERG17}
J.~S. Ellenberg, \emph{{Sumsets as unions of sumsets of subsets}}, Discrete
  Analysis \textbf{2017:14}, 5pp.

\bibitem{FOX11}
J.~Fox, \emph{A new proof of the graph removal lemma}, Ann. of Math. (2)
  \textbf{174} (2011), 561--579.

\bibitem{FL17}
J.~Fox and L.~M. Lov\'{a}sz, \emph{A tight bound for {Green's} arithmetic
  triangle removal lemma in vector spaces}, Adv. Math. \textbf{321} (2017),
  287--297.

\bibitem{FLS17}
J.~Fox, L.~M. Lov\'{a}sz, and L.~Sauermann, \emph{{A polynomial bound for the
  arithmetic $k$-cycle removal lemma in vector spaces}}, J. Combin. Theory Ser.
  A \textbf{160} (2018), 186--201.

\bibitem{FS18}
J.~Fox and L.~Sauermann, \emph{Erd{\H o}s-{G}inzburg-{Z}iv constants by
  avoiding three-term arithmetic progressions}, Electron. J. Combin.
  \textbf{25} (2018), Paper 2.14, 9pp.

\bibitem{FK14}
H.~Fu and R.~Kleinberg, \emph{Improved lower bounds for testing
  triangle-freeness in {B}oolean functions via fast matrix multiplication},
  Approximation, randomization, and combinatorial optimization, LIPIcs. Leibniz
  Int. Proc. Inform., vol.~28, Schloss Dagstuhl. Leibniz-Zent. Inform., Wadern,
  2014, pp.~669--676.

\bibitem{GS16}
G.~Ge and C.~Shangguan, \emph{Rank counting and maximum subsets of
  $\mathbb{F}_q^n$ containing no right angles}, preprint at arXiv:1612.08255,
  2016.

\bibitem{GREEN05}
B.~Green, \emph{A {S}zemer\'edi-type regularity lemma in abelian groups, with
  applications}, Geom. Funct. Anal. \textbf{15} (2005), 340--376.

\bibitem{GREEN17}
B.~Green, \emph{S\'{a}rk\"{o}zy's theorem in function fields}, Q. J. Math.
  \textbf{68} (2017), 237--242.

\bibitem{HX15}
I.~Haviv and N.~Xie, \emph{Sunflowers and testing triangle-freeness of
  functions}, I{TCS}'15---{P}roceedings of the 6th {I}nnovations in
  {T}heoretical {C}omputer {S}cience, ACM, New York, 2015, pp.~357--366.

\bibitem{HEGEDUS18}
G.~Heged\H{u}s, \emph{The {E}rd{\H o}s-{G}inzburg-{Z}iv constant and
  progression-free subsets}, J. Number Theory \textbf{186} (2018), 238--247.

\bibitem{KSS17}
R.~Kleinberg, W.~F. Sawin, and D.~E. Speyer, \emph{{The growth rate of
  tri-colored sum-free sets}}, Discrete Analysis \textbf{2018:12}, 10pp.

\bibitem{LOVETT18}
S.~Lovett, \emph{The analytic rank of tensors and its applications}, preprint
  at arXiv:1806.09179, 2018.

\bibitem{NASLUND18}
E.~Naslund, \emph{Exponential bounds for the {Erd{\H o}s-Ginzburg-Ziv}
  constant}, preprint at arXiv:1701.04942, 2018.

\bibitem{NASLUND18b}
E.~Naslund, \emph{The partition rank of a tensor and $k$-right corners in
  $\mathbb{F}^n_q$}, preprint at arXiv:1701.04475, 2018.

\bibitem{NS17}
E.~Naslund and W.~Sawin, \emph{Upper bounds for sunflower-free sets}, Forum
  Math. Sigma \textbf{5} (2017), e15, 10pp.

\bibitem{NORIN16}
S.~{Norin}, \emph{A distribution on triples with maximum entropy marginal},
  preprint at arXiv:1608.00243, 2016.

\bibitem{PEBODY17}
L.~{Pebody}, \emph{Proof of a conjecture of {Kleinberg-Sawin-Speyer}}, Discrete
  Analysis \textbf{2018:13}, 7pp.

\bibitem{PEBODY16}
L.~{Pebody}, \emph{Proof of a conjecture of {Kleinberg-Sawin-Speyer}}, preprint
  at arXiv:1608.05740v1, 2016.

\bibitem{PETROV16}
F.~Petrov, \emph{{Many Zero Divisors in a Group Ring Imply Bounds on
  Progression--Free Subsets}}, preprint at arXiv:1606.03256, 2016.

\bibitem{PP18}
F.~Petrov and C.~Pohoata, \emph{{Improved Bounds for Progression-Free Sets in
  $C_8^n$}}, preprint at arXiv:1805.05549, 2018.

\bibitem{SAWIN18}
W.~Sawin, \emph{Bounds for matchings in nonabelian groups}, Electron. J.
  Combin. \textbf{25} (2018), Paper 4.23, 21pp.

\bibitem{Taoblog16}
T.~Tao, \emph{A symmetric formulation of the
  {Croot-Lev-Pach-Ellenberg-Gijswijt} capset bound}, blog post,
  https://terrytao.wordpress.com/2016/05/18/a-symmetric-formulation-of-the-croot-lev-pach-ellenberg-gijswijt-capset-bound,
  2016.

\end{thebibliography}
\bibliographystyle{amsplain_mod2}

\end{document}